\documentclass[11pt,a4paper]{article}
\usepackage[utf8]{inputenc}
\usepackage{tabularx, graphicx, array} 
\usepackage{amsmath}
\usepackage{amsfonts}
\usepackage{amsthm}
\usepackage{amssymb}
\usepackage{graphicx}
\usepackage{float}
\usepackage[left=2cm, right=2cm, top=2cm, bottom=2cm]{geometry}

\usepackage{hyperref}

\usepackage{multirow}
\usepackage{mathtools}
\usepackage{xcolor}
\usepackage{enumitem}
\usepackage{mathabx}
\usepackage{cases}
\usepackage{stmaryrd}
\usepackage[framemethod=tikz]{mdframed}
\usepackage{rotating}
\usepackage{tabularx}
\usepackage{pifont}

\usepackage{enumerate}
\usepackage{mathtools}
\usepackage{diagbox}
\mathtoolsset{showonlyrefs}
\usepackage{xcolor}
\usepackage{afterpage}

\usepackage{subcaption}
\usepackage{booktabs}

\newtheorem{thm}{Theorem}
\newtheorem{lem}[thm]{Lemma}
\newtheorem{prop}[thm]{Proposition}

\newtheorem{algo}{Algorithm}

\theoremstyle{definition}

\newtheorem{rmk}[thm]{Remark}

\newcommand{\norm}[1]{\left\lVert #1 \right\rVert}
\newcommand{\normsq}[1]{\norm{#1}^{2}}

\newcommand{\gra}[1][]{\operatorname{gra}}

\newcommand{\minus}{\scalebox{0.75}[1.0]{$-$}}
\usepackage{nicefrac}

\setlist[enumerate]{label=$\rm{(\roman*)}$,leftmargin=\parindent}
\numberwithin{equation}{section}
\numberwithin{thm}{section}
\numberwithin{table}{section}
\numberwithin{figure}{section}

\newcommand{\R}{\mathbb{R}}
\newcommand{\cX}{\mathcal{X}}
\newcommand{\cY}{\mathcal{Y}}
\newcommand{\cH}{\mathcal{H}}
\newcommand{\cG}{\mathcal{G}}
\newcommand{\cK}{\mathcal{K}}
\newcommand{\cE}{\mathcal{E}}

\newcommand{\prox}{\mathrm{prox}}
\newcommand{\Id}{\mathrm{Id}}
\newcommand{\bO}{\mathcal{O}}

\newcommand{\Lag}{\mathcal{L}}

\newcommand{\zer}{{\rm Zer}}
\newcommand{\gr}{{\rm Graph}}

\newcommand{\dist}{{\rm dist}}

\allowdisplaybreaks

\title{Fast Reflected Forward-Backward algorithm: achieving fast convergence rates for convex optimization with linear cone
constraints}
\author{Radu Ioan Bo\c{t}\footnote{Faculty of Mathematics, University of Vienna, Oskar-Morgenstern-Platz 1, 1090 Vienna, Austria, e-mail: \url{radu.bot@univie.ac.at}. Research partially supported by Austrian Science Fund (FWF), projects W 1260-N35 and P 34922-N.}
\and Dang-Khoa Nguyen \footnote{Faculty of Mathematics and Computer Science, University of Science, Ho Chi Minh City 700000, Vietnam, e-mail: \url{ndkhoa@hcmus.edu.vn}.} \footnote{Vietnam National University, Ho Chi Minh City 700000, Vietnam}
\and Chunxiang Zong\footnote{Department of Mathematics, Northwest Normal University, Lanzhou 730070, People's Republic of China, e-mail: \url{zongchx@nwnu.edu.cn}.  Research supported by Outstanding Youth Science Fund of Gansu Province (No.  25JRRA001).}}

\begin{document}
	
\maketitle	

\begin{abstract}
In this paper, we derive a Fast Reflected Forward-Backward (Fast RFB) algorithm to solve the problem of finding a zero of the sum of a maximally monotone operator and a monotone and Lipschitz continuous operator in a real Hilbert space. Our approach extends the class of reflected forward-backward methods by introducing a Nesterov momentum term and a correction term, resulting in enhanced convergence performance.
The iterative sequence of the proposed algorithm is proven to converge weakly, and the Fast RFB algorithm demonstrates impressive convergence rates, achieving $o\left( \frac{1}{k} \right)$ as $k \to +\infty$ for both the discrete velocity and the tangent residual at the \emph{last-iterate}.
When applied to minimax problems with a smooth coupling term and nonsmooth convex regularizers, the resulting algorithm demonstrates significantly improved convergence properties compared to the current state of the art in the literature.  For convex optimization problems with linear cone constraints, our approach yields a fully splitting primal-dual algorithm that ensures not only the convergence of iterates to a primal-dual solution, but also a \emph{last-iterate} convergence rate of $o\left( \frac{1}{k} \right)$ as $k \to +\infty$ for the objective function value, feasibility measure, and complementarity condition. This represents the most competitive theoretical result currently known for algorithms addressing this class of optimization problems. Numerical experiments are performed to illustrate the convergence behavior of Fast RFB.
\end{abstract}	

\noindent \textbf{Key Words.} monotone inclusion,
reflected forward-backward splitting algorithm,
Nesterov momentum,
Lyapunov analysis,
fast convergence rates,
convergence of the iterates, saddle point problem,
fast primal-dual algorithm
\vspace{1ex}

\noindent \textbf{AMS subject classification.} 49M29, 65K05, 68Q25, 90C25, 90C47
	
\section{Introduction}\label{sec1}	

\subsection{Problem formulation}\label{subsec11}

In recent years, there has been a significant surge in research on minimax problems, primarily driven by their emerging applications in machine learning and robust optimization. Notable instances include generative adversarial networks (GANs) \cite{Arjovsky2017Chintala,Goodfellow2014Pouget}, which use minimax frameworks to improve data generation and adversarial training methods, and distributionally robust optimization \cite{Levy2020Carmon,Lin2022Fang}, which employs minimax principles to ensure model performance under varying distributional shifts. Beyond these areas, minimax problems have also found applications in online learning \cite{Bhatia2020Sridharan}, where they help develop algorithms that adapt to dynamic environments, and in reinforcement learning \cite{Azar2017Osband,Dai2018Shaw}, contributing to more efficient decision-making processes. This highlights the versatility and fundamental significance of minimax approaches across a wide range of domains.

Consider a minimax problem of the form
\begin{equation}\label{minmax}
\min_{x\in \cX}\max_{\lambda \in \cY} \Psi \left( x , \lambda \right) \coloneqq f(x) + \Phi\left( x , \lambda \right) - g(\lambda),
\end{equation}
where $\cX$ and $ \cY$ are two real Hilbert spaces, $f \colon \cX \to \R \cup \left\lbrace + \infty \right\rbrace$ and $g \colon \cY \to \R \cup \left\lbrace + \infty \right\rbrace$ are proper, convex, and lower semicontinuous functions, and
$\Phi \colon \cX \times \cY \rightarrow \R$ is a convex-concave and differentiable coupling function with Lipschitz continuous gradient. We are interested in finding saddle points of $\Psi$, which are pairs $\left( x_{*} , \lambda_{*} \right) \in  \cX \times \cY$ fulfilling
\begin{equation*}
\Psi(x_{*},\lambda) \leq \Psi\left( x_{*} , \lambda_{*} \right) \leq \Psi(x,\lambda_{*}) \textrm{ for every } \left( x , \lambda \right)\in  \cX \times \cY.
\end{equation*}
The minimax setting \eqref{minmax} is highly versatile, providing a general framework for studying a wide range of problems, including unconstrained composite convex minimization, mixed variational inequalities, and constrained convex minimization problems, see \cite{Condat2013,Lin2018Nadarajah,Chambolle2016Pock,Nesterov2013,Cai2009Osher,Tran_Dinh2019,Goldstein2014,Xu2017}.

An element $\left( x_{*} , \lambda_{*} \right)\in \cX \times \cY$ is a saddle point of \eqref{minmax} if and only if it is a solution of the system of optimality conditions
\begin{equation}\label{minimax:optcond}
\begin{pmatrix}0 \\ 0 \end{pmatrix} \in \begin{pmatrix} \partial f(x) \\ \partial g(\lambda) \end{pmatrix} + \begin{pmatrix} \nabla_{x} \Phi \left(x, \lambda\right) \\ -\nabla_{\lambda} \Phi \left(x, \lambda\right) \end{pmatrix},
\end{equation}
where $\partial f : \cX \rightarrow 2^{\cX} $ and $\partial g : \cY \rightarrow 2^{\cY}$ denote the convex subdifferentials of $f$ and $g$, respectively.

This motivates us to develop solution methods for solving the following monotone inclusion problem
\begin{equation}\label{intro:pb:eq}
0\in M(z) + F(z),
\end{equation}
where $\cH$ is a real Hilbert space, $M \colon \cH \rightarrow 2^{\cH}$ is a (possibly set-valued) maximally monotone operator and $F \colon \cH \rightarrow \cH$ is a single-valued monotone and $L$-Lipschitz continuous operator. We assume $\zer(M + F):=\{z \in \cH : 0 \in  M(z) + F(z)\} \neq \emptyset$.

The graph of a set-valued operator $M \colon \cH \rightarrow 2^{\cH}$ is defined as $\gr(M):=\{(z,u) \in \cH \times \cH : u \in M(z)\}.$ The operator $M$ is said to be monotone if $\langle v-u, y-z \rangle \geq 0$ for all $(z,u), (y,v) \in \gr(M)$. A monotone operator $M \colon \cH \rightarrow 2^{\cH}$ is said to be maximal monotone if there exists no other monotone operator $M' \colon \cH \rightarrow 2^{\cH}$ such that $\gr(M) \subsetneq \gr(M')$. The convex subdifferential of a proper, convex and lower semicontinuous function defined on a real Hilbert space is a maximally monotone operator \cite{Bot2010, BauschkeCombettes2}.
For $\cH := \cX \times \cY$, and
\begin{align}
& M \colon \cH \rightarrow 2^{\cH},  \ \left( x , \lambda \right) \mapsto \big(\partial f \left( x \right), \partial g\left( \lambda\right)\big), \label{op:A} \\
& F \colon \cH \rightarrow \cH, \ \left( x , \lambda \right) \mapsto (\nabla_{x} \Phi \left(x, \lambda\right),-\nabla_{\lambda} \Phi \left(x, \lambda\right)),\label{op:F}
\end{align}
both maximally monotone operators, the system of optimality conditions \eqref{minimax:optcond} reduces to the monotone inclusion problem \eqref{intro:pb:eq}.

If $M:=N_C$ with $C \subseteq \cH$ being a nonempty convex and closed set, \eqref{intro:pb:eq} becomes
\begin{equation}\label{Problem1}
0\in N_{C}(z) + F(z).
\end{equation}
This nothing else than the variational inequality problem
\begin{equation}
\textnormal{find}\,\,\, z\in C\,\,\, \textnormal{such that}\,\,\, \langle F(z), u-z\rangle\geq 0 \,\,\,\textnormal{for all}\,\,\, u\in C,\nonumber
\end{equation}
that has been extensively studied in the literature -- see, for instance, \cite{Diakonikolas2020,Wei2021Lee,Golowich2020aPattathil,Golowich2020bPattathil,Gorbunov2021Loizou,Sedlmayer2023Nguyen}.

In the following, we will review several methods for solving monotone inclusions of the form \eqref{intro:pb:eq}, as well as saddle point problems of the form \eqref{minmax}.

\subsection{Numerical methods for monotone inclusions with monotone and Lipschitz continuous operators}\label{subsec12}

In this subsection, we will survey the most prominent numerical methods for solving \eqref{intro:pb:eq}, excluding algorithms that rely on  $F$ being cocoercive and thus belong to the framework of the classical Forward-Backward (FB) method.

The Extragradient (EG) method was introduced by Korpelevich \cite{Korpelevich1976} and Antipin \cite{Antipin1976} and is one of the first and most famous algorithms for solving \eqref{minimax:optcond}. Based on the EG method, Tran-Dinh \cite{Tran-Dinh2023} has recently developed the following algorithm to solve \eqref{intro:pb:eq}
\begin{equation}\
(\forall k \geq 0) \ \left\{
\begin{aligned}\label{algo:EG}
		w_{k} & =  J_{\frac{\gamma}{\eta}M} \left( z_{k} - \frac{\gamma}{\eta} F(z_{k}) \right), \\
		z_{k+1} & =  J_{\gamma M} \left( z_{k} - \gamma F(w_{k}) \right),
\end{aligned}\right.
\end{equation}
where $J_{\gamma M} \coloneq \left( \Id + \gamma M \right) ^{-1} : \cH \to \cH$ denotes the resolvent of $M$ with parameter $\gamma >0$, which plays the role of a step size, and $\eta$ is a scaling factor. The scaling factor $\eta$ allows to provide a unified framework for different methods. For example, the classical EG method is obtained from  \eqref{algo:EG} for $\eta = 1$, $M$ given by \eqref{op:A}, with $f$ and $g$ the indicator functions of two nonempty, convex and closed subsets of $\cX$ and $\cY$, respectively, $F$ given by \eqref{op:F}, and the step size required to satisfy $0<\gamma<\frac{1}{L}$ (see, also, \cite{Facchinei2003Pang}). Assuming that $M$ is maximally 3-cyclically monotone, the author demonstrated convergence for the iterates generated by \eqref{algo:EG}, and that the tangent residual achieves best-iterate and last-iterate convergence rates of $\bO \left( \frac{1}{\sqrt{k}} \right)$ as $k \to +\infty$.

By tangent residual we  mean the quantity
\begin{equation*}
r_{tan}(z) \coloneq \dist(0, M(z) + F(z)) = \inf_{\xi\in M(z)} \left\lVert \xi+F(z) \right\rVert.
\end{equation*}
Since \eqref{intro:pb:eq} can be rewritten as the following fixed point problem
\begin{equation*}
z = J_{\gamma M} \left( z - \gamma F \left( z \right) \right),
\end{equation*}
another widely used measure of optimality is the so-called fixed-point residual
\begin{equation*}
r_{fix}(z) \coloneq \left\lVert z - J_{\gamma M} \left( z - \gamma F \left( z \right) \right) \right\rVert.
\end{equation*}
However, the derivation of convergence rates in terms of the tangent residual is more desirable. This is not only because it gives an upper bound on the fixed point residual, i.e. (see, for instance, \cite{Cai2022Zheng})
\begin{equation}\label{tanfixres}
0 \leq r_{fix}(z) \leq  r_{tan}(z) \quad \forall z \in \cH,
\end{equation}
but also because it allows the convergence rates to be transferred to function values when applied to convex optimization problems and minimax problems such as \eqref{minmax}.

In order to reduce the computational cost of the EG method caused by evaluating the operator $F$ at two different points in each iteration, Popov introduced in \cite{Popov1980} the Optimistic Gradient Descent Ascent (OGDA) method, which requires only one evaluation of the operator per iteration. Its extension to solving \eqref{intro:pb:eq} provided in \cite{Tran-Dinh2023} is as follows
\begin{equation}\
(\forall k \geq 1) \ \left\{
\begin{aligned}\label{algo:OGDA}
		w_{k} & =  J_{\frac{\gamma}{\eta}M} \left( z_{k} - \frac{\gamma}{\eta} F(w_{k-1}) \right), \\
		z_{k+1} & =  J_{\gamma M} \left( z_{k} - \gamma F(w_{k}) \right),
\end{aligned}\right.
\end{equation}
and differs from \eqref{algo:EG} in that in the first block $F(w_{k-1})$ replaces $F(z_k)$, see also \cite{Bohm2023}.
The classical OGDA method is obtained from \eqref{algo:OGDA} in the same way as described above, as is the EG method from \eqref{algo:EG}, but with the step size required to satisfy $0<\gamma<\frac{1}{2L}$.  Assuming that $M$ is maximally 3-cyclically monotone, in \cite{Tran-Dinh2023}, convergence for the iterates generated by \eqref{algo:OGDA}, and that the tangent residual achieves best-iterate and last-iterate convergence rates of $\bO \left( \frac{1}{\sqrt{k}} \right)$ as $k \to +\infty$ are shown.

Both the EG and OGDA methods require two evaluations of the resolvent $J_{\gamma M}$ per iteration. To address this computational demand, Tseng~\cite{Tseng2000} proposed the Forward-Backward-Forward (FBF) method, inspired by the EG framework. The FBF method reduces the per-iteration complexity by requiring only a single evaluation of $J_{\gamma M}$ per iteration, making it particularly advantageous when $J_{\gamma M}$ is expensive to compute. Its iterative scheme is as follows
\begin{equation}\
(\forall k \geq 0) \  \left\{
\begin{aligned}\label{algo:Tseng}
		w_{k} & = J_{\gamma M} \left( z_{k} - \gamma F(z_{k}) \right), \\
		z_{k+1} & = w_{k} - \gamma F(w_{k}) + \gamma F(z_{k}),
\end{aligned}\right.
\end{equation}
and generates a sequence $(z_k)_{k \geq 0}$ that converges to a solution of \eqref{intro:pb:eq} for $ 0 < \gamma < \frac{1}{L} $. If $M=0$, the FBF method reduces to the classical EG method \cite{Korpelevich1976}. For cases where $M$ represents the convex subdifferential of a proper, convex and lower semicontinuous function, an ergodic convergence rate of $\bO \left( \frac{1}{k} \right)$ for the restricted gap function as $k \to +\infty$ was established in \cite{Bohm2022Sedlmayer}. Furthermore, the best-iterate convergence rate of $\bO \left( \frac{1}{\sqrt{k}} \right)$ for the tangent residual as $k \to +\infty$ has been shown in two distinct contexts: using a potential function approach for the star-co-monotone case in \cite{Luo2022Tran-Dinh}, and through an alternative method in \cite{Tran-Dinh2023} that leverages results from \cite{Facchinei2003Pang} combined with the concept of star co-hypomonotonicity.

By replacing $F(z_{k})$ with $F(w_{k-1})$ in \eqref{algo:Tseng}, the following Past Forward-Backward-Forward (PFBF) method is obtained
\begin{equation}\
(\forall k \geq 1) \ \left\{
\begin{aligned}\label{algo:PFBF}
		w_{k} & = J_{\gamma M} \left( z_{k} - \gamma F(w_{k-1}) \right), \\
		z_{k+1} & = w_{k} - \gamma F(w_{k}) + \gamma F(w_{k-1}),
\end{aligned}\right.
\end{equation}
which requires only one forward estimation per iteration. For this method, the best-iterate convergence rate of $\bO \left( \frac{1}{\sqrt{k}} \right)$ for the tangent residual as $k \to +\infty$ was established in \cite{Luo2022Tran-Dinh} under the assumption that $M$ is star-co-monotone. By interchanging the roles of $w_{k}$ of $z_{k}$ in \eqref{algo:PFBF} and simplifying to a single sequence, one arrives at the Forward-Reflected-Backward (FRB) method proposed by Malitsky and Tam in \cite{Malitsky2020Tam}, described as
\begin{equation}\label{algo:Malitsky-Tam}
	(\forall k \geq 1) \ z_{k+1} = J_{\gamma M} \left( z_{k} - 2 \gamma F(z_{k}) + \gamma F(z_{k-1}) \right),
\end{equation}
which converges to a solution of \eqref{intro:pb:eq} provided $0 < \gamma < \frac{1}{2L}$. This iterative scheme can also be derived from \eqref{algo:Tseng} by reusing $F(w_{k-1})$ instead of $F(z_{k})$ in the first line, similar to how the OGDA method is derived from the EG method. For variational inequalities,  \cite{Cai2022Zheng} demonstrated that the FRB method achieves a last-iterate convergence rate of $\bO \left( \frac{1}{\sqrt{k}} \right)$ for the tangent residual as $k \to +\infty$.

In \cite{Cevher2020Vu}, Cevher and V\~{u} proposed the following Reflected Forward-Backward (RFB) method
\begin{equation}\label{algo:Malitsky}
			(\forall k \geq 1) \ z_{k+1} = J_{\gamma M} \left( z_{k} - \gamma F(2 z_{k} - z_{k-1}) \right),
\end{equation}
which converges to a solution of \eqref{intro:pb:eq} provided $0 < \gamma < \frac{\sqrt{2}-1}{L} $. In \eqref{algo:Malitsky}, the evaluation of $F$ through a second forward step is circumvented by using a suitable linear combination of the iterates. If $F$ is linear, \eqref{algo:Malitsky} is equivalent to \eqref{algo:Malitsky-Tam}. In \cite{Tran-Dinh2023}, best-iterate and last-iterate convergence rates of $\bO \left( \frac{1}{\sqrt{k}} \right)$ as $k \to +\infty$ for the tangent residual were established. For the variational inequality problem \eqref{Problem1}, \eqref{algo:Malitsky} reduces to the Projected Reflected Gradient (PRG) method introduced in \cite{Malitsky2015}. Last-iterate convergence rates for the PRG were provided in \cite{Cai2022Zheng}.

In recent years, there has been significant interest in the development of numerical methods with fast convergence properties for solving monotone inclusions. Using the performance estimation problem framework, Kim introduced in \cite{Kim2021} an Accelerated Proximal Point (APP) method for solving \eqref{intro:pb:eq} in the special case where $F \equiv 0$. This method achieves a convergence rate of  $\bO \left( \frac{1}{k} \right)$ as $k +\infty$ for the fixed-point residual, thereby outperforming the classical Proximal Point method \cite{BauschkeCombettes2}. For problems without a set-valued operator ($M \equiv 0$ in \eqref{intro:pb:eq}), a Fast Optimistic Gradient Descent Ascent (Fast OGDA) method was proposed in \cite{Bot2023RobertNguyen}. This method, derived as a discretization of a fast-converging continuous time model, not only ensures the convergence of the iterates but also achieves a last-iterate convergence rate of $o \left( \frac{1}{k} \right)$ as $k \to +\infty$ for the operator norm of $F$.

The Extra-Anchored Gradient (EAG) method, inspired by Halpern iteration \cite{Halpern1967}, has been developed to address \eqref{intro:pb:eq}. This algorithm achieves a last-iterate convergence rate of $\bO  \left( \frac{1}{k} \right)$  as $k \to +\infty$ for the tangent residual when $M$ is maximally 3-cyclically monotone \cite{Tran-Dinh2023}. The method extends an earlier algorithm proposed by Yoon and Ryu in \cite{Yoon2021Ryu}. Building on these ideas, the Past Extra-Anchored Gradient (PEAG) method was introduced in \cite{Tran-Dinh2024}, leveraging concepts from \cite{Tran-Dinh2021Luo} and the Optimistic Gradient Descent Ascent (OGDA) method \cite{Popov1980}. The PEAG method is designed to solve \eqref{intro:pb:eq} under the assumption that $M+F$ is co-hypomonotone. It also guarantees a last-iterate convergence rate of $\bO \left( \frac{1}{k} \right)$ for the tangent residual as $k \to +\infty$.

Also building on Halpern iteration \cite{Halpern1967}, Cai and Zheng proposed in \cite{Cai2022Zheng} an Accelerated Reflected Gradient (ARG) method. This method, which can be viewed as an acceleration of the RFB method, addresses the monotone inclusion \eqref{intro:pb:eq} under the assumptions that $M$ is maximally monotone, $F$ is Lipschitz continuous, and $M+F$ is negatively comonotone. The ARG method achieves a convergence rate of $\bO \left( \frac{1}{k} \right)$ as $k \to +\infty$ for the tangent residual.

Additional fast methods that exploit either Nesterov momentum \cite{Nesterov2004} or Halpern iteration can be found in \cite{Cai2022Oikonomou,Tran-Dinh2024,Mainge2021,Mainge2023Weng-Law}. A more detailed discussion of these approaches, along with comparisons to our proposed method, will be presented in the next section.

\subsection{Numerical methods for saddle point problems with convex regularizes}\label{subsec13}

In the literature, several attempts have been made to solve the saddle point problem \eqref{minmax} directly, without relying on the more general formulation \eqref{intro:pb:eq}.

A particularly well-studied instance of this problem involves a bilinear coupling term, namely, $\Phi(x, \lambda):= \langle Ax, \lambda \rangle$, where $A\colon\cX\rightarrow \cY$ is a linear continuous operator. In this case, the problem \eqref{minmax} reduces to
\begin{equation}\label{minmax_linear}
\min_{x \in \cX}\max_{\lambda \in \cY} f(x) + \langle Ax, \lambda \rangle - g(\lambda).
\end{equation}
To address \eqref{minmax_linear}, a primal-dual approach was first proposed in \cite{Arrow1958Hurwicz} and further developed in \cite{Zhu2008Chan}, with convergence properties guaranteed under the assumption that $f$ is strongly convex. Later, Chambolle and Pock \cite{Chambolle2011Pock} introduced a fully splitting primal-dual algorithm to solve \eqref{minmax_linear} in the finite-dimensional setting. They demonstrated that the sequence of iterates $\left( x_{k} , \lambda_{k} \right) _{k\geq 0}$ converges to a saddle point of \eqref{minmax_linear}, and they established an ergodic convergence rate of $\bO \left(\frac{1}{k}\right)$ as $k \to +\infty$ for the so-called restricted primal-dual gap. When $f$ is strongly convex, an accelerated version of this primal-dual algorithm achieves an improved ergodic convergence rate of  $\bO \left(\frac{1}{k^2}\right)$ as $k \to +\infty$ again with respect to the restricted primal-dual gap \cite{Chambolle2016Pock}. The Chambolle-Pock algorithm has since inspired a variety of primal-dual methods for solving \eqref{minmax_linear}, including those proposed in \cite{Davis2015,Chambolle2016Pock,Jiang2021Cai,Tran-Dinh2021A,Tran-Dinh2020Zhu,Valkonen2020,Tran-Dinh2018Fercoq}, to name just a few.

Compared to the bilinear case, the study of \eqref{minmax} in its general form has been less extensive. Nemirovski and Juditsky \cite{Nemirovski2004,Juditsky2011Nemirovski} introduced the Mirror-Prox method to address \eqref{minmax} in the absence of regularizers  ($f=g=0$) achieving an ergodic convergence rate of $\bO  \left( \frac{1}{k} \right)$ as $k \to +\infty$ for the restricted gap function. This method was later extended in \cite{He2015Juditsky} to handle convex regularizers using Bregman distances, while maintaining the same ergodic convergence rate. Hamedani and Aybat \cite{Hamedani2021Aybat} proposed an Accelerated Primal-Dual (APD) algorithm incorporating a Nesterov momentum term, which generalizes the Chambolle-Pock approach \cite{Chambolle2011Pock} to the saddle point problem \eqref{minmax}. The APD algorithm is defined as
\begin{equation*}\
(\forall k \geq 1) \ \left\{
\begin{aligned}
	s_{k}      & = (1+ \theta_k)\nabla_{\lambda} \Phi (x_k,\lambda_k)-\theta_k \nabla_{\lambda} \Phi (x_{k-1},\lambda_{k-1}),\\
    \lambda_{k+1}    & = \prox_{\sigma_k g} \left( \lambda_{k} + \sigma_k s_k \right), \\
	x_{k+1}    & = \prox_{\tau_k f} \left( x_{k} - \tau_k \nabla_{x} \Phi (x_k,\lambda_{k+1})\right).
\end{aligned}\right.
\end{equation*}
Under suitable conditions on the parameter sequence $\left( \tau_{k} , \sigma_{k} \right) _{k \geq 0}$, the authors proved that the iterates converge to a saddle point. Moreover, the ergodic sequence $\left( \bar{x}_{k} , \bar{\lambda}_{k} \right)_{k\geq 0}$ satisfied $\Psi\left( \bar{x}_{k},\lambda_{*} \right) -\Psi\left( x_{*},\bar{\lambda}_{k} \right) \to 0$ with a convergence rate of $\bO \left( \frac{1}{k} \right)$ in the general convex setting and of $\bO  \left( \frac{1}{k^2} \right)$ as $k \to +\infty$ when $\Phi(x,\cdot)$ is linear for each fixed $x$ and $f$ is strongly convex.

Recently, Chang, Yang, and Zhang \cite{Chang2024Yang} introduced an enhancement of the APD algorithm by employing adaptive linesearch techniques, which only assume local Lipschitz continuity for $\nabla_x \Phi$ and $\nabla_\lambda \Phi$. This enhanced method retains similar convergence and convergence rate properties as the original APD algorithm.

\subsection{Our contribution}\label{subsec14}

This paper introduces an accelerated first-order method for solving the monotone inclusion problem \eqref{intro:pb:eq}. The proposed approach ensures the weak convergence of iterates to a solution of \eqref{intro:pb:eq} and achieves last-iterate convergence rates of $o \left( \frac{1}{k} \right)$ as $k \to +\infty$ for both the discrete velocity and the tangent residual. We demonstrate the versatility of the algorithm by applying it to the minimax problem \eqref{minmax} and convex optimization problems with linear cone constraints. Finally, we validate the theoretical results and explore the impact of algorithm parameters through comprehensive numerical experiments.

The contributions of the paper are as follows:

\begin{itemize}
\item We propose a Fast Reflected Forward-Backward (Fast RFB) algorithm, which incorporates a Nesterov momentum term and a correction term, for solving the monotone inclusion problem \eqref{intro:pb:eq}. This method requires only a single operator evaluation and one resolvent computation $J_{\gamma M}$ per iteration. The iterative sequence $(z_k)_{k \geq 0}$ generated by the algorithm weakly converges to a solution of \eqref{intro:pb:eq}. Furthermore, the Fast RFB algorithm achieves a last-iterate convergence rate of $o \left( \frac{1}{k} \right)$ as $k \rightarrow +\infty$ for both the discrete velocity $\left\lVert z_k - z_{k-1} \right\rVert$ and the tangent residual $r_{\mathrm{tan}} \left( z_k \right) = \dist (0, M(z_k) + F(z_k)) = \inf_{\xi\in M(z_k)} \left\lVert \xi + F \left( z_k \right) \right\rVert$.

\item Building on the Fast RFB method, we develop a primal-dual full-splitting algorithm for solving the saddle point problem \eqref{minmax}. The proposed algorithm ensures the weak convergence of the sequence of primal-dual iterates $\left(x_{k},\lambda_{k} \right)_{k \geq 0}$ to a saddle point. Additionally, it achieves last-iterate convergence rates of  $o \left( \frac{1}{k} \right)$ as $k \to +\infty$ or the discrete primal and dual velocities, the tangent residual, and the primal-dual gap.

\item As a particular instance of the saddle problem \eqref{minmax}, we apply the proposed primal-dual full splitting algorithm to solve optimization problems of the form
\begin{equation}
\begin{aligned}\label{convexcone1}
\min \ & f \left( x \right)+h \left( x \right),\\
\textrm{subject to} \	& Ax - b \in -\mathcal{K},
\end{aligned}
\end{equation}
where $\cX$ and $\cY$ are real Hilbert spaces, $\mathcal{K}$ is a nonempty, convex and closed cone in $\cY$, $f \colon \cX \rightarrow \R \cup \left\lbrace + \infty \right\rbrace$ is a proper, convex, and lower semicontinuous function, $h \colon \cX \rightarrow \R$ is a convex and differentiable function such that $\nabla h$ is $L_{\nabla h}$-Lipschitz continuous, and $A \colon \cX\rightarrow \cY$ is a linear continuous operator.

We generate a sequence  $\left(x_{k},\lambda_{k} \right)_{k \geq 0}$ of primal-dual iterates which converges weakly to a primal-dual solution of \eqref{convexcone1}. In addition, we achieve as $k \to +\infty$ convergence rates for the velocities
\begin{equation*}
\left\lVert x_{k} - x_{k-1} \right\rVert = o \left( \dfrac{1}{k} \right)	
\qquad \textrm{ and } \qquad
\left\lVert \lambda_{k} - \lambda_{k-1} \right\rVert = o \left( \dfrac{1}{k} \right),
\end{equation*}
the tangent residual
\begin{equation*}
\left\lVert u_k + \nabla h(x_k) + A^{*} \lambda_k \right\rVert = o \left( \dfrac{1}{k} \right)
\qquad \textrm{ and } \qquad
\left\lVert v_k - Ax_{k} + b \right\rVert = o \left( \dfrac{1}{k} \right),
\end{equation*}
where $u_k \in \partial f(x_k)$ and $v_k \in N_{{\cal K}^*}(\lambda_k)$, for all $k \geq 0$, the primal-dual gap
\begin{equation*}
{\cal L} \left( x_{k} , \lambda_{*} \right)- {\cal L} \left( x_{*} , \lambda_{k} \right) = o \left( \dfrac{1}{k} \right),
\end{equation*}
the complementarity condition
\begin{equation*}
|\langle \lambda_k, Ax_{k} - b \rangle| = o \left( \dfrac{1}{k} \right),
\end{equation*}
and the objective function values
\begin{equation*}
\left\vert \left( f+h \right) \left( x_{k} \right) - \left( f+h \right) \left( x_{*} \right) \right\rvert = o \left( \dfrac{1}{k} \right).
\end{equation*}
Here, ${\cal L}$ denotes the Lagrangian attached to \eqref{convexcone1}, and $\left( x_{*} , \lambda_{*} \right)$ a primal-dual optimal solution.

\item The approach we consider involves the minimization of the sum of a nonsmooth convex and a smooth convex function subject to linear equality constraints. While there is an extensive body of work on full splitting primal-dual methods for solving this class of problems, few results exist regarding fast-converging methods in terms of objective function values and feasibility measures, while also ensuring the convergence of the iterates. Our approach contributes to filling this gap.
\end{itemize}

\section{A Fast Reflected Forward-Backward algorithm for monotone inclusions}\label{sec2}

\subsection{The Fast RFB algorithm}\label{subsec21}

In this section, we formulate the algorithm and conduct a thorough analysis of its convergence.

\begin{mdframed}
\begin{algo}
\label{algo:im}
Let
$$\alpha>2, \quad\frac{\alpha}{2}< c <\alpha - 1 ,
 \quad \textnormal{and}\quad 0<\gamma <\frac{1}{2L}.$$
For initial points $z_0, y_{0}, w_{0} \in \cH$, and $z_{1} = J_{\gamma M} \left( y_{0} - \gamma F \left( w_{0} \right)\right)$, we set
\begin{equation}
\label{algo:z-y}
(\forall k \geq  1) \ \left \{\begin{aligned}
y_{k}
& = z_{k} + \left( 1 - \dfrac{\alpha}{k + \alpha} \right) \left( z_{k} - z_{k-1} \right) + \left( 1 - \dfrac{c}{  k + \alpha } \right) \left( y_{k-1} - z_{k} \right),  \\
w_{k}
&= z_{k} + \left( y_{k} - y_{k-1} \right),\\
z_{k+1}
& = J_{\gamma M} \left( y_{k} -\gamma F \left( w_{k} \right) \right).
\end{aligned} \right.
\end{equation}
\end{algo}
\end{mdframed}

In the following we give an equivalent formulation for Algorithm \ref{algo:im}, which will play a central role in the convergence analysis.

\begin{prop}
\label{rmk:sub}
Let $z_{0}, y_{0}, w_{0} \in \cH$, $z_{1} = J_{\gamma M} \left( y_{0}  -\gamma F \left( w_{0} \right)\right)$, and $\xi_{1} = \frac{1}{\gamma} \left( y_{0} - z_{1} \right) - F \left( w_{0} \right)\in M \left( z_{1} \right)$. Then the sequence $(z_k)_{k \geq 0}$ generated in Algorithm \ref{algo:im} can also be generated equivalently by the following iterative scheme
\begin{equation}
\label{algo:z-xi}
(\forall k \geq  1) \ \left \{\begin{aligned}
w_{k} 		& = z_{k} + \left( 1 - \dfrac{\alpha}{k + \alpha} \right) \left( z_{k} - z_{k-1} \right) - \frac{c}{k+\alpha}  \gamma \big(\xi_k  + F \left( w_{k-1} \right)\big),  \\
z_{k+1} 	& =  J_{\gamma M} \big( w_{k} - \gamma F \left( w_{k} \right) + \gamma \left(F \left( w_{k-1} \right) + \xi_{k}\right) \big),\\
\xi_{k+1}	& = \frac{1}{\gamma}\left( w_{k}-z_{k+1}\right)  - F \left( w_{k} \right) + F \left( w_{k-1} \right) + \xi_{k}.
\end{aligned} \right.
\end{equation}
In addition, it holds
\begin{equation}
\label{algo:inc-M}
\xi_{k}\in M \left( z_{k} \right) \qquad \quad \forall k\geq 1.
\end{equation}
\end{prop}

\begin{proof}
Given $z_{0}, w_{0}, y_{0} \in \cH$ and $z_{1} = J_{\gamma M} \left( y_{0}  -\gamma F \left( w_{0} \right)\right)$,
it holds
\begin{equation*}
\xi_{1} = \frac{1}{\gamma} \left( y_{0} - z_{1} \right)  - F \left( w_{0} \right)\in M \left( z_{1} \right) .
\end{equation*}
In the same way, by invoking also the third update block of Algorithm \ref{algo:im}, we obtain
\begin{equation*}
\xi_{k+1} := \frac{1}{\gamma}\left( y_{k}-z_{k+1}\right)  - F \left( w_{k} \right)\in M(z_{k+1}) \quad \forall k \geq 0.
\end{equation*}
Therefore, the iterative scheme in Algorithm \ref{algo:im} can be equivalently written as
\begin{equation}\
(\forall k \geq  1) \ \left\{
\begin{aligned}\label{ouralg1}	
y_{k} 		& = z_{k} + \left( 1 - \dfrac{\alpha}{k + \alpha} \right) \left( z_{k} - z_{k-1} \right) + \left( 1 - \frac{c}{k+\alpha} \right) \gamma \left(\xi_k + F \left( w_{k-1} \right)\right),  \\
w_{k}
&= y_{k} - \gamma \left( \xi_k + F \left( w_{k-1} \right) \right),\\
z_{k+1} 	& =  J_{\gamma M} \left( y_{k} -\gamma F \left( w_{k} \right)\right),\\
\xi_{k+1}	& = \frac{1}{\gamma}\left( y_{k}-z_{k+1}\right)  -  F \left( w_{k} \right),
\end{aligned}\right.
\end{equation}
which, after some simplifications, it transforms into \eqref{algo:z-xi}.

Conversely, starting from \eqref{algo:z-xi}, we can define a sequence $\left( y_{k} \right)_{k \geq 1}$ following the first update block in  \eqref{ouralg1}. Consequently, for every $k \geq 1$ it holds
\begin{equation*}
\gamma \left( \xi_{k+1}  +  F \left( w_{k} \right) \right) = y_k-z_{k+1} ,
\end{equation*}
leading to the transformation of \eqref{algo:z-xi} into \eqref{algo:z-y}.
\end{proof}

\begin{rmk}\label{rem22}
If the \emph{momentum term} $z_{k} - z_{k-1}$ and the \emph{correction term} $y_{k-1} - z_{k}$ are removed from the first update block of \eqref{algo:z-y}, the resulting algorithm reduces to the Reflected Forward-Backward (RFB) method
\begin{equation}\
(\forall k \geq 1) \ \left\{
\begin{aligned}\label{RFB}
w_{k}
& = 2z_{k} - z_{k-1},\\
z_{k+1}
& = J_{\gamma M} \left( z_{k} -\gamma F \left( w_{k} \right) \right),
\end{aligned}\right.
\end{equation}
as introduced in \cite{Cevher2020Vu}. This method extends the PRG method proposed by Malitsky in \cite{Malitsky2015} for variational inequalities. For \eqref{RFB}, best-iterate and last-iterate convergence rates of $\bO \left( \frac{1}{\sqrt{k}} \right)$ as $k \to +\infty$ for the tangent residual were established in \cite{Tran-Dinh2023}.

When $F$ is linear, \eqref{RFB} simplifies to the Forward-Reflected-Backward (FRB) method \eqref{algo:Malitsky-Tam} by Malitsky and Tam \cite{Malitsky2020Tam}
\begin{equation*}
(\forall k \geq 1) \ z_{k+1} := J_{\gamma M} \left( z_{k} -2\gamma F \left( z_{k} \right) + \gamma F \left( z_{k-1} \right) \right).
\end{equation*}
While convergence results for the iterates were established in \cite{Malitsky2020Tam} in the general setting, a last-iterate convergence rate of $\bO \left( \frac{1}{\sqrt{k}} \right)$ as $k \to +\infty$ for the tangent residual was demonstrated in \cite{Cai2022Zheng} in the specific case of variational inequalities.
\end{rmk}

\begin{rmk}\label{rem23}
The Accelerated Reflected Gradient (ARG) method, introduced by Cai and Zheng in \cite{Cai2022Zheng}, is defined as follows for initial points $z_{0}=z_{1} \in \cH$ and $z_{2}=J_{\gamma M}\left(z_{1} -\gamma F(z_{1})\right)$,
\begin{equation*}\
(\forall k \geq 1) \ \left\{
\begin{aligned}
x_{k} & = 2 z_{k} -  z_{k-1} + \frac{1}{k+1}\left(z_{0} - z_{k} \right) - \frac{1}{k}\left(z_{0} - z_{k-1} \right), \\
z_{k+1} & = J_{\gamma M} \left( z_{k} - \gamma F\left( x_{k} \right) + \frac{1}{k+1} \left(z_{0} -z_{k} \right)\right),
\end{aligned}\right.
\end{equation*}
where $0<\gamma\leq\frac{1}{2\sqrt{6}L}$. The ARG method builds on the anchoring technique employed in \eqref{RFB}, a concept originating in \cite{Halpern1967}. It was shown to achieve a last-iterate convergence rate of $\bO \left(\frac{1}{k}\right)$ as $k \to +\infty$ for the tangent residual.

The Accelerated Extragradient (AEG) method was introduced by Tran-Dinh in \cite{Tran-Dinh2024} accelerates the FBF method by employing the momentum term $\frac{k+1}{k+3}\left( x_{k}-x_{k-1} \right)$ alongside additional correction terms. It is defined as follows for initial points $x_0, x_1, w_0, z_1 \in \cH$
\begin{equation*}\
(\forall k \geq 1) \ \left\{
\begin{aligned}\label{AEG}
x_{k} & = J_{\gamma M} \left( z_{k} - \gamma F(z_{k}) + \frac{k+1}{k+2}\gamma w_{k-1} \right), \\
w_{k} & = \frac{1}{\gamma} \left( z_{k} - x_{k} + \frac{k+1}{k+2} \gamma w_{k-1}\right) + F (x_{k}) -  F(z_{k}),\\
z_{k+1} & = x_{k} + \frac{k+1}{k+3}\left( x_{k}-x_{k-1} \right) - \frac{k+2}{k+3}\gamma \left( F (x_{k}) -  F(z_{k}) \right),
\end{aligned}\right.
\end{equation*}
where $0<\gamma< \frac{1}{L}$.

To further reduce the number of evaluations of $F$, Tran-Dinh proposed the Accelerated Past Extragradient (APEG) method in \cite{Tran-Dinh2024}, defined as follows for initial points $x_0, x_1, w_0, z_1 \in \cH$
\begin{equation}\
(\forall k \geq 1) \ \left\{
\begin{aligned}
x_{k} & = J_{\gamma M} \left( z_{k} -\gamma F(z_{k}) + \frac{k+1}{k+2}\gamma w_{k-1} \right), \\
w_{k} & = \frac{1}{\gamma} \left( z_{k} - x_{k} + \frac{k+1}{k+2} \gamma w_{k-1}\right),\\
z_{k+1} & = x_{k} + \frac{k+1}{k+3}\left( x_{k}-x_{k-1} \right) + \frac{5(k+2)}{6(k+3)}\gamma w_{k} -\frac{5(k+1)}{6(k+3)}\gamma w_{k-1},\nonumber
\end{aligned}\right.
\end{equation}
where $0<\gamma\leq \frac{3}{2\sqrt{29}L}$.

Both the AEG and APEG methods have been shown to achieve last-iterate convergence rates of $\bO \left(\frac{1}{k}\right)$ as $k \to +\infty$ for the tangent residual.

In contrast, the Fast RFB algorithm introduces both momentum and correction terms to the Reflected Forward-Backward framework \eqref{RFB}. This design not only ensures weak convergence to a solution of \eqref{intro:pb:eq} but also achieves a superior convergence rate of $\bO\left(\frac{1}{k}\right)$ as $k \to +\infty$ for the discrete velocity and the tangent residual.
\end{rmk}

\subsection{Convergence analysis}\label{subsec22}

In the following, we will use the notation
\begin{equation}\label{eq:vk}
v_{k} := F(w_{k-1}) + \xi_{k} \quad \forall k \geq 1.
\end{equation}
With this notation, from Algorithm~\ref{algo:im} we obtain
\begin{equation}
\label{algo:im:z}
w_{k} = z_{k} + \dfrac{k}{k + \alpha} \Big( z_{k} - z_{k-1} \Big) -  \dfrac{c}{k + \alpha} \gamma v_{k} \quad \mbox{and} \quad
z_{k+1}	= w_{k} - \gamma \Big( v_{k+1} - v_{k} \Big) \quad \forall k \geq 1.
\end{equation}
Since $ 0 < \gamma < \frac{1}{2L}$, the definition of $v_{k}$ together with the Lipschitz continuity of $F$ give for all $k \geq 1$
\begin{align}\label{split:Lip}
\left\lVert \xi_{k+1} + F \left( z_{k+1} \right) - v_{k+1} \right\rVert
& = \left\lVert F \left( z_{k+1} \right) - F \left( w_{k} \right) \right\rVert
\leq L \left\lVert z_{k+1} - w_{k} \right\rVert \nonumber\\
&= \gamma L \left\lVert v_{k+1} - v_{k} \right\rVert
\leq \frac{1}{2}  \left\lVert v_{k+1} - v_{k} \right\rVert
\leq \left\lVert v_{k+1} - v_{k} \right\rVert.
\end{align}
By summing up the two equations in \eqref{algo:im:z}, we obtain that for all $k \geq 1$ it holds
\begin{equation}\label{split:d-u}
\left( k + \alpha \right) \left( z_{k+1} - z_{k} \right) - k \left( z_{k} - z_{k-1} \right)
= \minus c \gamma v_{k+1} - \gamma (k+\alpha-c) \left( v_{k+1} - v_{k} \right) .
\end{equation}

Let $z_{\ast} \in \zer(M+F)$, $0 \leq \lambda \leq \alpha - 1$ and $1 < s < 2$. In the lines of \cite{Bot2023RobertNguyen}, we denote for all $k \geq 1$
\begin{align}
\label{split:defi:u-k-1-lambda}
u_{\lambda,s,k} := & 2 \lambda \left( z_{k} - z_{\ast} \right) + 2k \left( z_{k} - z_{k-1} \right) + s\gamma k v_{k},\\
\label{split:defi:E-k}
\cE_{\lambda,s,k} := & \dfrac{1}{2} \left\lVert u_{\lambda,s,k} \right\rVert ^{2} + 2 \lambda \left( \alpha - 1 - \lambda \right) \left\lVert z_{k} - z_{\ast} \right\rVert ^{2} + 2 \lambda \gamma \left((2-s)k+2(\alpha-c)\right) \left\langle z_{k} - z_{\ast}, v_{k} \right\rangle \nonumber\\
& + \frac{1}{2}\gamma^2\left((2-s)k+2(\alpha-c)\right) \left( s k + 2 c \right) \left\lVert v_{k} \right\rVert ^{2},
\end{align}
and for all $k \geq 2$
\begin{align}
\label{split:defi:F}
\cG_{\lambda,s,k} &:= \cE_{\lambda,s,k} - 2 \gamma \big((2-s)k+2(\alpha -c)\big)k \left\langle z_{k} - z_{k-1}, F \left( z_{k} \right) - F \left( w_{k-1} \right) \right\rangle \nonumber\\
&\quad\phantom{:} + \gamma^3 L\big((2-s)k+2(\alpha -c)\big) \left(  k+\alpha -c +c \gamma L \sqrt{(2-s)k+2(\alpha -c )} \right) \left\lVert v_{k} - v_{k-1} \right\rVert^{2}.
\end{align}
The proof of the following lemma is given in Appendix B.
\begin{lem}
\label{lem:dE}
Let $z_{\ast} \in \zer(M+F)$ and $\left(z_{k} \right)_{k \geq 0}$ be the sequence generated by Algorithm~\ref{algo:im}. For $0 \leq \lambda \leq \alpha - 1$ and $1 < s < 2$, the following identity holds for all $k \geq 1$
\begin{equation}\label{split:dE}
\begin{split}
\cE_{\lambda,s,k+1} - \cE_{\lambda,s,k} = & \ \minus 4 \left( c - 1 \right) \lambda \gamma \left\langle z_{k+1} - z_{\ast}, v_{k+1} \right\rangle+ 2 \left( \lambda + 1 - \alpha \right) \left( 2k + \alpha + 1 \right) \left\lVert z_{k+1} - z_{k} \right\rVert ^{2} \\
& + 2\gamma \Bigl(\big((2-s)\lambda + s  \left( \lambda + 1 - \alpha \right) +s - 2c\big) k + 2 \lambda \alpha +s -2\alpha c \Bigr)  \left\langle z_{k+1} - z_{k}, v_{k+1} \right\rangle \\
& - 2\gamma \left( k + \alpha \right) \Big((2-s)k + 2\left( \alpha - c \right) \Big) \left\langle z_{k+1} - z_{k}, v_{k+1} - v_{k} \right\rangle\\
& - \gamma^2 \left( k + \alpha \right) \Big((2-s)k + 2\left( \alpha - c \right) \Big) \left\lVert v_{k+1} - v_{k} \right\rVert ^{2} \\
& +\gamma^2  \Big( (1-c)(2sk + 2c +s) + s(\alpha - c )\Big) \left\lVert v_{k+1} \right\rVert ^{2} .
\end{split}
\end{equation}
\end{lem}

In the next result we prove a quasi-F\'ejer monotone property together with a lower bound for the sequence $(\cG_{\lambda,s,k})_{k \geq 1}$. For the proof we also refer the reader to Appendix B.

\begin{lem}
\label{lem:reg}
Let $z_{\ast} \in \zer(M+F)$ and $\left(z_{k} \right)_{k \geq 0}$ be the sequence generated by Algorithm~\ref{algo:im}. For $0 \leq \lambda \leq \alpha - 1$ and $1 < s < 2$, the following statements are true:
\begin{itemize}
\item[(i)] for all $k \geq 2$, it holds
\begin{equation*}
\begin{split}
& \ \cG_{\lambda,s,k+1} - \cG_{\lambda,s,k}\\
\leq & \frac{4\left( c - 1 \right)^2}{\left( k+1 \right) \sqrt{k+1}}  \lambda ^{2} \left\lVert z_{k+1} - z_{\ast} \right\rVert ^{2}- 4 \left( c - 1 \right) \lambda \gamma \left\langle z_{k+1} - z_{\ast}, \xi_{k+1} + F \left( z_{k+1} \right) \right\rangle\\
& + 2\gamma \left( \omega_{1} k + \omega_{2} \right) \left\langle z_{k+1} - z_{k}, v_{k+1} \right\rangle -  \mu_{k} \gamma^{2} \left\lVert v_{k+1} - v_{k} \right\rVert ^{2} \\
& + 2\left( \omega_{3} k + \sqrt{\omega_{5}(k+1)+\omega_{7}}\right) \left\lVert z_{k+1} - z_{k} \right\rVert ^{2} + \gamma^2 \left( \omega_{4} k + c \sqrt{\omega_{0}k+2 \omega_{6}} + s \omega_{6}\right)  \left\lVert v_{k+1} \right\rVert ^{2},
\end{split}
\end{equation*}
where
\begin{equation}\label{ene:const}
\begin{split}		
\omega_{0} = & \ 2-s > 0, \quad \omega_{1} = \omega_{0} \lambda +s  \left( \lambda + 1 - \alpha \right) + s - 2c, \quad \omega_{2} = 2\lambda \alpha +s -2 \alpha c, \\
\omega_{3} 	 = & \ 2 \left( \lambda + 1 - \alpha \right) \leq 0, \quad \omega_{4} = 2s( 1 - c) < 0, \quad \omega_{5}  = (\alpha - 2)\omega_{0} > 0, \\
\omega_{6} = & \ \alpha - c > 0, \quad \omega_{7}  = \ (\alpha -1)(2\omega_{6} - \omega_{0}) > 0,\\
\mu_{k} = & \ \omega_{0} \left(1 - 2\gamma L\right)k^2 + \left(2 \omega_{6} + \omega_{0}\alpha \right)k + 2 \omega_{6} \alpha  \\
& \ - 2 \gamma L\left(\left( 2(\omega_{0} + 2 \omega_{6})- s \omega_{6}\right)k + \left(\omega_{0} + 2 \omega_{6}\right)(\alpha +1 -c)\right)  - (k + 1)\sqrt{k + 1}\\
& \  -(\omega_{5}(k+1) + \omega_{7})\sqrt{\omega_{5}(k+1) + \omega_{7}} -\gamma^2 L^2 c\big(\omega_{0}(k + 1) + 2 \omega_{6}\big)\sqrt{\omega_{0}(k + 1) + 2 \omega_{6}}.
\end{split}
\end{equation}

\item[(ii)] there exists $k_1 \geq 2$ such that for all $k \geq k_1$ it holds
\begin{equation}\label{ene:low}
\begin{split}
\cG_{\lambda,s, k}
\geq & \ \dfrac{1}{4 s } \omega_{0} \left\lVert 4 \lambda \left( z_{k} - z_{\ast} \right) + 2 k \left( z_{k} - z_{k-1} \right) + 2s \gamma k v_{k} \right\rVert ^{2} \\
& +  \frac{1}{4s}\omega_{0}^2k^{2} \left\lVert z_{k} - z_{k-1} \right\rVert ^{2} + 2  \lambda \left(
\alpha - 1 - \dfrac{4\left( \alpha - 1\right) }{s\alpha} \lambda\right) \left\lVert z_{k} - z_{\ast} \right\rVert ^{2}.
\end{split}
\end{equation}
\end{itemize}
\end{lem}
The following lemma will play an essential role in the proof of the main convergence result of this section. Its proof is also given in Appendix B.
\begin{lem}
\label{lem:trunc}
Let $z_{\ast} \in \zer(M+F)$ and $\left(z_{k} \right)_{k \geq 0}$ be the sequence generated by Algorithm~\ref{algo:im}. The following statements are true:
\begin{itemize}
\item[(i)] if $s$ and $\delta$ are such that
\begin{equation}\label{condi:s}
1+ \frac{\alpha}{4 c - \alpha} <s < 2
\end{equation}
and
\begin{equation}\label{condi:delta}
\max\left\{\sqrt{\frac{s\left(\alpha - 2\right) + 2\left( 2c -s \right)}{ 4 s \left(c - 1\right)}}, \sqrt{\frac{-\left( 2 -s \right)\left(\alpha - 1 \right) - s +2c }{s\left(c - 1 \right)}}\right\}< \delta <1,
\end{equation}
then there exist
\begin{equation}
\label{trunc:fea}
0 \leq \underline{\lambda} \left( \alpha,c,s \right) < \overline{\lambda} \left( \alpha,c,s \right) \leq \dfrac{s \alpha}{4},
\end{equation}
such that for every $\lambda$ satisfying $\underline{\lambda} \left( \alpha,c,s \right) < \lambda < \overline{\lambda} \left( \alpha,c,s \right)$ one can find an integer $k_{\lambda} \geq 1$ with the property that for all $k \geq  k_{\lambda}$ the following inequality holds
\begin{equation}\label{Rk}
\begin{split}
R_{k} := & \ 2\gamma \left( \omega_{1} k + \omega_{2} \right) \left\langle z_{k+1} - z_{k}, v_{k+1} \right\rangle + \delta\gamma^2 \left( \omega_{4} k + c\sqrt{\omega_{0}k + 2 \omega_{6}}+ s \omega_{6}\right)  \left\lVert v_{k+1} \right\rVert ^{2}\\
& \ + 2\delta\left( \omega_{3} k + \sqrt{\omega_{5}(k + 1) + \omega_{7}} \right) \left\lVert z_{k+1} - z_{k} \right\rVert ^{2} \\
\leq & \ 0;
\end{split}
\end{equation}

\item[(ii)] there exists $k_2 \geq 2$ such that for all $k \geq k_2$ it holds
\begin{equation}
\label{muk}
\mu_{k} \geq 0.
\end{equation}
\end{itemize}
\end{lem}
We are now in a position to prove a proposition that allows us to make initial statements about convergence rates.
\begin{prop}
\label{prop:split:lim}
Let $z_{\ast} \in \zer(M+F)$ and $\left(z_{k} \right)_{k \geq 0}$ be the sequence generated by Algorithm~\ref{algo:im}. The following statements are true:
\begin{itemize}
\item[(i)] it holds
\begin{align*}
& \sum_{k \geq 1} \left\langle z_{k} - z_{\ast}, F (z_{k}) + \xi_{k} \right\rangle < + \infty, \quad  \sum_{k \geq 1} k^{2} \left\lVert v_{k+1} - v_{k} \right\rVert ^{2} < + \infty,\\
& \sum_{k \geq 1} k \left\lVert F(w_{k}) + \xi_{k+1} \right\rVert ^{2} < + \infty,  \qquad  \sum_{k \geq 1} k \left\lVert z_{k+1} - z_{k} \right\rVert ^{2} < + \infty;
\end{align*}

\item[(ii)] the sequence $\left(z_{k} \right) _{k \geq 0}$ is bounded and it holds as $ k \to +\infty $
\begin{align*}
& \left\lVert z_{k} - z_{k-1} \right\rVert = \bO  \left( \dfrac{1}{k} \right),
\qquad
\left\lVert \xi_{k} +  F \left( w_{k-1} \right) \right\rVert = \bO  \left( \dfrac{1}{k} \right),
\\
& \left\lVert \xi_{k} + F \left( z_{k} \right) \right\rVert = \bO  \left( \dfrac{1}{k} \right),
\quad
\left\langle z_{k} - z_{\ast}, \xi_{k} + F \left( z_{k} \right) \right\rangle = \bO  \left( \dfrac{1}{k} \right);
\end{align*}

\item[(iii)] for all $s \in \left(1+ \frac{\alpha}{4 c - \alpha}, 2 \right)$, there exist $0 \leq \underline{\lambda} \left( \alpha,c, s \right) < \overline{\lambda} \left( \alpha,c ,s \right) \leq \frac{s\alpha}{4}$ such that, for all $\underline{\lambda} \left( \alpha,c,s \right) < \lambda < \overline{\lambda} \left( \alpha,c,s \right)$, the sequences $\left( \cE_{\lambda,s,k} \right)_{k \geq 1}$ and $\left( \cG_{\lambda,s,k} \right) _{k \geq 2}$ are convergent.
\end{itemize}
\end{prop}

\begin{proof}
Let $s \in \left(1+ \frac{\alpha}{4 c - \alpha}, 2 \right)$, and $\delta \in (0,1)$ such that \eqref{condi:delta} is satisfied. According to Lemma~\ref{lem:trunc}(i) there exist $0 \leq \underline{\lambda} \left( \alpha,c,s \right) < \overline{\lambda} \left( \alpha,c,s \right) \leq \frac{s\alpha}{4}$ such that for all $ \lambda \in \left(\underline{\lambda} \left( \alpha,c,s \right), \overline{\lambda} \left( \alpha,c,s \right)\right)$ there exists an integer ${ k_{\lambda}} \geq 1$ with the property that \eqref{Rk} holds for all $k \geq { k_{\lambda}}$. In addition, according to Lemma~\ref{lem:trunc}(ii), we get a positive integer ${ k_{2}} \geq 2$ such that \eqref{muk} holds for all $k \geq { k_{2}}$.

This means that for all $k \geq { k_{0}} := \max \left\lbrace { k_{\lambda}},k_{1},  { k_{2}} \right\rbrace$, where $k_{1}$ is the positive integer given by Lemma \ref{lem:reg}(ii), according to Lemma \ref{lem:reg}(i) it holds
\begin{equation*}
\begin{split}
& \ \cG_{\lambda,s,k+1} - \cG_{\lambda,s,k}\\
\leq & \ \frac{4\left( c - 1 \right)^2}{\left( k+1 \right) \sqrt{k+1}}  \lambda ^{2} \left\lVert z_{k+1} - z_{\ast} \right\rVert ^{2}- 4 \left( c - 1 \right) \lambda \gamma \left\langle z_{k+1} - z_{\ast}, \xi_{k+1} + F \left( z_{k+1} \right) \right\rangle\\
& \ -  \mu_{k} \gamma^{2} \left\lVert v_{k+1} - v_{k} \right\rVert ^{2} + \left( 1 - \delta\right)\gamma^2 \left( \omega_{4} k + c\sqrt{\omega_{0}k + 2 \omega_{6}}+ s \omega_{6}\right)  \left\lVert v_{k+1} \right\rVert ^{2}\\
& \ + 2\left( 1 - \delta\right)\left( \omega_{3} k + \sqrt{\omega_{5}(k + 1) + \omega_{7}} \right) \left\lVert z_{k+1} - z_{k} \right\rVert ^{2}.
\end{split}
\end{equation*}
Since $\omega_{3}, \omega_{4} < 0$, there exists $k_{3} \geq k_{0}$ such that for all $k \geq  k_{3}$
\begin{equation}\label{ineqFk}
\begin{split}
\cG_{\lambda,s,k+1}
\leq & \ \cG_{\lambda,s,k} + \frac{4\left( c - 1 \right)^2}{\left( k+1 \right) \sqrt{k+1}}  \lambda ^{2} \left\lVert z_{k+1} - z_{\ast} \right\rVert ^{2}- 4 \left( c - 1 \right) \lambda \gamma \left\langle z_{k+1} - z_{\ast}, \xi_{k+1} + F \left( z_{k+1} \right) \right\rangle\\
& \ -  \mu_{k} \gamma^{2} \left\lVert v_{k+1} - v_{k} \right\rVert ^{2} + \frac{1}{2}\left( 1 - \delta\right)\gamma^2 \omega_{4} k   \left\lVert v_{k+1} \right\rVert ^{2}
+ \left( 1 - \delta\right)\omega_{3} k\left\lVert z_{k+1} - z_{k} \right\rVert ^{2}.
\end{split}
\end{equation}
In view of \eqref{ene:low}, we get that $\cG_{\lambda,s,k} \geq 0$ for every $k \geq 2$. By setting
\begin{equation*}
C_{0} := 2\lambda \left( c - 1 \right)^2 \left( \alpha -1 - \frac{ 4 \left(\alpha - 1 \right)}{s \alpha}\lambda\right) ^{-1} > 0,
\end{equation*}
it holds for all $k \geq 1$
\begin{align*}
& \ \frac{4\left( c - 1 \right)^2}{\left( k+1 \right) \sqrt{k+1}}  \lambda ^{2} \left\lVert z_{k+1} - z_{\ast} \right\rVert ^{2}
= \dfrac{C_{0}}{\left( k+1 \right) \sqrt{k+1}} \cdot 2 \lambda \left( \alpha - 1 - \dfrac{4\left( \alpha - 1 \right)}{s\alpha} \lambda\right) \left\lVert z_{k+1} - z_{\ast} \right\rVert ^{2} \nonumber \\
\leq & \ \dfrac{C_{0}}{\left( k+1 \right) \sqrt{k+1}} \cG_{\lambda,s,k+1}.
\end{align*}
Under these premises, we deduce from \eqref{ineqFk} that for all $k \geq  k_{3}$
\begin{equation}\label{inq:Fk-re}
\begin{split}
\left( 1 - \dfrac{C_{0}}{\left( k+1 \right) \sqrt{k+1}} \right) \cG_{\lambda, s, k+1}
\leq & \ \cG_{\lambda, s, k} - 4 \left( c - 1 \right) \lambda \gamma \left\langle z_{k+1} - z_{\ast}, \xi_{k+1} + F \left( z_{k+1} \right) \right\rangle\\
& \ -  \mu_{k} \gamma^{2} \left\lVert v_{k+1} - v_{k} \right\rVert ^{2}\\
& \ + \frac{1}{2}\left( 1 - \delta\right)\gamma^2 \omega_{4} k   \left\lVert v_{k+1} \right\rVert ^{2}
+ \left( 1 - \delta\right)\omega_{3} k\left\lVert z_{k+1} - z_{k} \right\rVert ^{2}.
\end{split}
\end{equation}
Choosing $ k_{4} := \max \left\lbrace k_{3}, \left\lceil C_{0}^{\frac{2}{3}}- 1 \right\rceil \right\rbrace$, we have that for all $ k \geq  k_{4} $
\begin{equation*}
\left( 1 - \dfrac{C_{0}}{\left( k+1 \right) \sqrt{k+1}} \right) ^{-1} = \dfrac{\left( k+1 \right) \sqrt{k+1}}{\left( k+1 \right) \sqrt{k+1} - C_{0}} = 1 + \dfrac{C_{0}}{\left( k+1 \right) \sqrt{k+1} - C_{0}} > 1.
\end{equation*}
Hence, using the monotonicity of $M+F$ and that $\omega_3, \omega_4 <0$, \eqref{inq:Fk-re} leads for all $k \geq { k_{4}}$ to
\begin{align*}
\cG_{\lambda,s, k+1} \leq & \ \left( 1 + \dfrac{C_{0}}{\left( k+1 \right) \sqrt{k+1} - C_{0}} \right) \cG_{\lambda,s,k} - 4 \left( c - 1 \right) \lambda \gamma \left\langle z_{k+1} - z_{\ast}, \xi_{k+1} + F \left( z_{k+1} \right) \right\rangle\\
& \ -  \mu_{k} \gamma^{2} \left\lVert v_{k+1} - v_{k} \right\rVert ^{2} + \frac{1}{2}\left( 1 - \delta\right)\gamma^2 \omega_{4} k   \left\lVert v_{k+1} \right\rVert ^{2}
+ \left( 1 - \delta\right)\omega_{3} k\left\lVert z_{k+1} - z_{k} \right\rVert ^{2}.
\end{align*}
Denoting
\begin{align*}
b_{\lambda, s, k} 	:= & \ 4 \left( c - 1 \right) \lambda \gamma \left\langle z_{k+1} - z_{\ast}, \xi_{k+1} + F \left( z_{k+1} \right) \right\rangle+  \mu_{k} \gamma^{2} \left\lVert v_{k+1} - v_{k} \right\rVert ^{2}\\
& \ - \frac{1}{2}\left( 1 - \delta\right)\gamma^2 \omega_{4} k   \left\lVert v_{k+1} \right\rVert ^{2}
-\left( 1 - \delta\right)\omega_{3} k\left\lVert z_{k+1} - z_{k} \right\rVert ^{2}\\
\geq & \ 0, \\
d_{\lambda, s, k} 	& := \ \dfrac{C_{0}}{\left( k+1 \right) \sqrt{k+1} - C_{0}} > 0,
\end{align*}
we see that we are in the context of Lemma~\ref{lem:quasi-Fej}. From here we get the summability statements in (i) as well as the convergence of the sequence $\left( \cG_{\lambda,s,k} \right)_{k \geq 2}$.

Since $\left( \cG_{\lambda,s,k}  \right) _{k \geq 2}$ converges, it is also bounded from above, thus, for all $k \geq 2$
\begin{align*}
& \ \frac{1}{4s}\omega_{0} \left\lVert 4 \lambda \left( z_{k} - z_{\ast} \right) + 2k \left( z_{k} - z_{k-1} \right) + 2s \gamma k v_{k} \right\rVert ^{2} \nonumber\\
& \ + \frac{\omega_{0}^2}{4s}k^{2} \left\lVert z_{k} - z_{k-1} \right\rVert ^{2}+ 2  \lambda \left( \alpha - 1 -\frac{4\left(\alpha -1\right)}{s\alpha}\lambda \right)\left\lVert z_{k} - z_{\ast} \right\rVert ^{2}\\
\leq & \ \cG_{\lambda,s, k} \leq \sup_{k \geq 1} \cG_{\lambda,s, k} < +\infty.
\end{align*}
From here we obtain that the sequences
\begin{equation*}
\begin{gathered}
\left( 4 \lambda \left( z_{k} - z_{\ast} \right) + 2k \left( z_{k} - z_{k-1} \right) + 2s \gamma k v_{k} \right) _{k \geq 1}, \quad \left( k \left( z_{k} - z_{k-1} \right) \right) _{k \geq 1}
\quad \textrm{and} \quad
\left(z_{k} \right) _{k \geq 0}
\end{gathered}
\end{equation*}
are bounded. In particular, for all $k \geq 2$
\begin{align}
\left\lVert 4 \lambda \left( z_{k} - z_{\ast} \right) + 2k \left( z_{k} - z_{k-1} \right) + 2s\gamma k v_{k} \right\rVert
\leq & \ C_{1} := \sqrt{\dfrac{4s}{\omega_{0}} \sup_{k \geq 1} \cG_{\lambda,k}}, \nonumber\\
k \left\lVert z_{k} - z_{k-1} \right\rVert
\leq & \ C_{2} :=  \sqrt{\dfrac{4s}{\omega_{0}^2} \sup_{k \geq 1} \cG_{\lambda,k}}, \nonumber\\
\left\lVert z_{k} - z_{\ast} \right\rVert
\leq & \ C_{3} := \sqrt{\frac{1}{2  \lambda \left( \alpha - 1 -\frac{4\left(\alpha -1\right)}{s\alpha}\lambda \right)} \sup_{k \geq 1} \cG_{\lambda,k}},\label{ineq-C3}
\end{align}
therefore
\begin{align}\label{ineq-C4}
\left\lVert v_{k} \right\rVert
\leq & \ \dfrac{1}{2 s \gamma k} \left\lVert 4 \lambda \left( z_{k} - z_{\ast} \right) + 2k \left( z_{k} - z_{k-1} \right) + 2s \gamma k v_{k} \right\rVert \nonumber \\
& \ + \dfrac{1}{s \gamma} \left\lVert z_{k} - z_{k-1} \right\rVert + \dfrac{2 \lambda}{s\gamma k} \left\lVert z_{k} - z_{\ast} \right\rVert \leq \dfrac{C_{4}}{k},
\end{align}
where
\begin{equation*}
C_{4} := \dfrac{1}{2 s\gamma} \left( C_{1} + 2C_{2} + 4 \overline{\lambda} \left( \alpha,c,s \right) C_{3} \right) > 0 .
\end{equation*}
From (i) we have
\begin{equation}
\lim\limits_{k \to + \infty} k \left\lVert v_{k+1} - v_{k} \right\rVert = 0
\quad \Rightarrow \quad
C_{5} := \sup_{k \geq 1} \left\lbrace k \left\lVert v_{k+1} - v_{k} \right\rVert \right\rbrace < + \infty, \label{split:lim:dV-0}
\end{equation}
which, together with \eqref{split:Lip}, implies that  for all $k \geq 1$
\begin{equation}
\begin{split}
\left\lVert \xi_{k+1} + F \left( z_{k+1} \right) \right\rVert
&\leq \left\lVert \xi_{k+1} + F \left( z_{k+1} \right) - v_{k+1} \right\rVert + \left\lVert v_{k+1} \right\rVert\\
&\leq \left\lVert v_{k+1} - v_{k} \right\rVert + \left\lVert v_{k+1} \right\rVert \label{split:lim:V-dV}
\leq \dfrac{C_{6}}{k},
\end{split}
\end{equation}
where
\begin{equation*}
C_{6} := C_{4} + C_{5} > 0 .
\end{equation*}
The Cauchy-Schwarz inequality and the boundedness of $\left( z_{k} \right) _{k \geq 0}$ allow us to provide a similar estimate for $\left\langle z_{k} - z_{\ast}, \xi_{k} + F \left( z_{k} \right) \right\rangle$. This proves (ii).
To complete the proof, we are going to show that
\begin{equation*}
\lim\limits_{k \to + \infty} \cE_{\lambda,s,k} = \lim\limits_{k \to + \infty} \cG_{\lambda,s,k}  \in \R .
\end{equation*}
Indeed, we have already seen that
\begin{equation*}
\lim\limits_{k \to + \infty} k \left\lVert v_{k} - v_{k-1} \right\rVert = 0,
\end{equation*}
which, by the Cauchy-Schwarz inequality and \eqref{split:Lip}, yields
\begin{align*}
0 \leq \lim\limits_{k \to + \infty} k^{2} \left\lvert \left\langle z_{k} - z_{k-1}, F \left( z_{k} \right) - F \left( w_{k-1} \right) \right\rangle \right\rvert
& \leq C_{2} \lim\limits_{k \to + \infty} k \left\lVert F \left( z_{k} \right) - F \left( w_{k-1} \right) \right\rVert \nonumber \\
& \leq C_{2} \lim\limits_{k \to + \infty} k \left\lVert v_{k} - v_{k-1} \right\rVert = 0 .
\end{align*}
It is from here that we get the statement we want.
\end{proof}

Next we will prove the convergence of the sequence of iterates generated by Algorithm~\ref{algo:im}.

\begin{thm}\label{thm:conv}
Let $\zer(M+F) \neq \emptyset$. The sequence $\left(z_{k} \right)_{k \geq 0}$ generated by Algorithm~\ref{algo:im} converges weakly to a solution of \eqref{intro:pb:eq}.
\end{thm}
\begin{proof}
Let $z_\ast \in \zer(M+F)$. Further, let $s \in \left(1+ \frac{\alpha}{4 c - \alpha}, 2 \right)$, and $0 \leq \underline{\lambda} \left( \alpha,c, s \right) < \overline{\lambda} \left( \alpha,c ,s \right) \leq \frac{s\alpha}{4}$ be the parameters provided by Proposition \ref{prop:split:lim}(iii) with the property that, for all $\underline{\lambda} \left( \alpha,c,s \right) < \lambda < \overline{\lambda} \left( \alpha,c,s \right)$, the sequence $\left( \cE_{\lambda,s,k} \right)_{k \geq 1}$ is convergent.

For all $k \geq 1$ and any $\lambda \in \left(\underline{\lambda} \left( \alpha,c,s \right), \overline{\lambda} \left( \alpha,c,s \right)\right)$ we have
\begin{equation}\label{defi:E-lambda}
\begin{split}
\cE_{\lambda,s, k}
= & \ \dfrac{1}{2} \left\lVert 2 \lambda \left( z_{k} - z_{\ast} \right) + 2k \left( z_{k} - z_{k-1} \right) + s \gamma k v_{k} \right\rVert ^{2}
+ 2 \lambda \left( \alpha - 1 - \lambda \right) \left\lVert z_{k} - z_{\ast} \right\rVert ^{2} \\
& + 2\lambda \gamma \left(\left(2 - s \right)k + 2\left( \alpha - c \right) \right) \left\langle z_{k} - z_{\ast}, v_{k} \right\rangle + \dfrac{1}{2} \gamma^{2}  \left( \left( 2 -s \right)k + 2\left( \alpha - c \right)\right)\left(sk + 2c\right) \left\lVert v_{k} \right\rVert ^{2} \\
= & \ 2 \lambda \left( \alpha - 1 \right) \left\lVert z_{k} - z_{\ast} \right\rVert ^{2} + 4 \lambda k \left\langle z_{k} - z_{\ast}, z_{k} - z_{k-1} +  \gamma v_{k} \right\rangle  \\
& + \dfrac{k^{2}}{2}  \left\lVert 2 \left( z_{k} - z_{k-1} \right) + s \gamma v_{k} \right\rVert ^{2} + 4\left(\alpha - c \right) \lambda \gamma \langle z_{k} - z_{\ast}, v_{k} \rangle \\
& + \frac{1}{2}\gamma^2\big(\left( 2 - s\right)k + 2\left( \alpha - c \right)\big)\left(sk + 2c\right) \left\lVert v_{k} \right\rVert ^{2}.
\end{split}
\end{equation}
We choose $\underline{\lambda} \left( \alpha,c,s \right) < \lambda_{1} < \lambda_{2} < \overline{\lambda} \left( \alpha,c,s \right)$, and get
\begin{align*}
& \ \cE_{\lambda_{2},s,k} - \cE_{\lambda_{1},s,k} \\
= & \ 4 \left( \lambda_{2} - \lambda_{1} \right) \bigg( \dfrac{1}{2} \left( \alpha - 1 \right) \left\lVert z_{k} - z_{\ast} \right\rVert ^{2} +   k \left\langle z_{k} - z_{\ast}, z_{k} - z_{k-1}  +
\gamma v_{k} \right\rangle + \gamma \left(\alpha - c \right)\langle z_k - z_{\ast}, v_k\rangle\bigg) \nonumber \\
= & \ 4 \left( \lambda_{2} - \lambda_{1} \right)\big( p_{k} + \gamma \left(\alpha - c \right)\langle z_k - z_{\ast}, v_k\rangle\big),
\end{align*}
where for all $k \geq 1$
\begin{equation}
\label{split:defi:p-k}
p_{k} := \dfrac{1}{2} \left( \alpha - 1 \right) \left\lVert z_{k} - z_{\ast} \right\rVert ^{2} + k \left\langle z_{k} - z_{\ast}, z_{k} - z_{k-1} + \gamma v_{k} \right\rangle.
\end{equation}
According to \eqref{ineq-C3} and \eqref{ineq-C4}, we have
\begin{align}\label{deczk}
0 \leq \lim_{k\rightarrow +\infty} | \langle z_k-z_{\ast}, v_k\rangle |
\leq C_3\lim_{k\rightarrow +\infty}\left\lVert v_{k} \right\rVert=0,
\end{align}
which, together with the fact that the limit $\lim_{k \to + \infty} \left(\cE_{\lambda_{2},s,k} - \cE_{\lambda_{1},s,k} \right)\in \R$ exists, leads to
\begin{equation}
\label{split:lim-p-k}
\lim\limits_{k \to + \infty} p_{k} \in \R \textrm{ exists}.
\end{equation}

We define for all $k \geq 1$
\begin{equation*}
\label{split:defi:q-k}
q_{k} := \dfrac{1}{2} \left\lVert z_{k} - z_{\ast} \right\rVert ^{2} +  \gamma \sum_{i = 1}^{k} \left\langle z_{i} - z_{\ast}, v_{i} \right\rangle,
\end{equation*}
and notice that for all $k \geq 2$ it holds
\begin{align*}
q_{k} - q_{k-1}
& = \left\langle z_{k} - z_{\ast}, z_{k} - z_{k-1} \right\rangle - \dfrac{1}{2} \left\lVert z_{k} - z_{k-1} \right\rVert ^{2} +  \gamma \left\langle z_{k} - z_{\ast}, v_{k} \right\rangle,
\end{align*}
thus
\begin{equation*}
\left( \alpha - 1 \right) q_{k} + k \left( q_{k} - q_{k-1} \right) = p_{k} +  \left( \alpha - 1 \right) \gamma \sum_{i = 1}^{k} \left\langle z_{i} - z_{\ast}, v_{i} \right\rangle - \dfrac{k}{2} \left\lVert z_{k} - z_{k-1} \right\rVert ^{2} .
\end{equation*}
Thanks to Proposition \ref{prop:split:lim}(i), we have $\lim_{k \to + \infty} k \left\lVert z_{k+1} - z_{k} \right\rVert ^{2} = 0$.
So if we can show that  the sequence $\left( \sum_{i = 1}^{k} \left\langle z_{i} - z_{\ast}, v_{i} \right\rangle \right)_{k \geq 1}$ also converges, then we know that
\begin{equation}
\label{conv:q:dq}
\lim\limits_{k \to + \infty} \left( \alpha - 1 \right) q_{k} + k \left( q_{k} - q_{k-1} \right) \in \R \textrm{ exists} .
\end{equation}
To show that the series above converges, we first observe that for every $k \geq 2$
\begin{align}
\sum_{i = 2}^{k} \left\lvert \left\langle z_{i} - z_{\ast}, F \left( w_{i-1} \right) - F \left( z_{i} \right) \right\rangle \right\rvert
& \leq \sum_{i = 2}^{k} \left\lVert z_{i} - z_{\ast} \right\rVert \left\lVert F \left( w_{i-1} \right) - F \left( z_{i} \right) \right\rVert \label{conv:inn:C-S} \\
&  \leq \frac{1}{2} \sum_{i=2}^{k} \frac{1}{i^{2}} \normsq{z_{i} - z_{\ast}} + \frac{1}{2} \sum_{i=2}^{k} i^{2}  \normsq{F(w_{i-1}) - F(z_{i})}  \nonumber\\
&  \leq \frac{1}{2} \sum_{i \geq 2} \frac{1}{i^{2}} \normsq{z_{i} - z_{\ast}} + \frac{1}{2} \sum_{i \geq 2} i^{2} \normsq{F(w_{i-1}) - F(z_{i})} \nonumber\\
&< + \infty, \label{conv:inn:sup}
\end{align}
where the first series in~\eqref{conv:inn:sup} converges due to~\eqref{ineq-C3}, and the second series converges due to~\eqref{split:Lip} and Proposition \ref{prop:split:lim}(i). This proves that the series $\sum_{k \geq 2} \left \langle z_{k} - z_{\ast}, F \left( w_{k-1} \right) - F \left( z_{k} \right) \right\rangle$ is absolutely convergent, so convergent. Again using Proposition \ref{prop:split:lim}(i), we see that the limit
\begin{equation*}
\begin{split}
\MoveEqLeft \lim_{k \to + \infty} \sum_{i = 1}^{k} \left\langle z_{i} - z_{\ast}, v_{i} \right\rangle = \lim_{k \to + \infty} \sum_{i = 1}^{k} \left\langle z_{i} - z_{\ast}, \xi_{i} + F \left( z_{i} \right) \right\rangle + \lim_{k \to + \infty} \sum_{i = 1}^{k} \left\langle z_{i} - z_{\ast}, F \left( w_{i-1} \right) - F \left( z_{i} \right) \right\rangle \in \R
\end{split}
\end{equation*}
exists. Consequently, \eqref{conv:q:dq} holds. Therefore, we can apply Lemma~\ref{lem:lim-u-k} to guarantee that the limit $\lim_{k \to + \infty} q_{k} \in \R$ also exists. The required boundedness of $\left(z_{k} \right)_{k \geq 0}$ follows from Proposition~\ref{prop:split:lim}(ii) and the fact that $\lim_{k \to + \infty} \sum_{i = 1}^{k} \left\langle z_{i} - z_{\ast}, v_{i} \right\rangle \in \R$ exists.  Given the definition of $(q_{k})_{k \geq 1}$, the latter property also guarantees that $\lim_{k \to + \infty} \left\lVert z_{k} - z_{\ast} \right\rVert \in \R$ exists. The hypothesis (i) in the Opial Lemma (see Lemma~\ref{lem:opial}) is therefore fulfilled.

Let $z$ be a weak sequential cluster point of $\left(z_{k} \right)_{k \geq 0}$, which means that there exists a subsequence $\left (z_{k_{l}} \right )_{l \geq 0}$ which converges weakly to $z$ as $l \to +\infty$. It follows from Proposition~\ref{prop:split:lim}(ii) that $\xi_{k_{l}} + F \left( z_{k_{l}} \right)$ strongly converges to $0$ as $l \to +\infty$.  Since $F \colon \mathcal{H} \rightarrow \mathcal{H} $ is a single-valued monotone and $L$-Lipschitz continuous operator,  it is a maximally monotone operator (see \cite[Corollary 20.28]{BauschkeCombettes2}) with full domain.  Therfore,  the sum $M+F$ is also a maximally monotone operator, (see \cite[Corollary 25.5]{BauschkeCombettes2}), given that $M$ is maximally monotone.  The fact that $\xi_{k_{l}} \in M(z_{k_{l}})$ for all $l \geq 0$, and the maximal monotonicity of $M + F$ implies that $0 \in \left( M + F \right) \left( z \right)$, see \cite[Proposition 20.33]{BauschkeCombettes2}, meaning that hypothesis (ii) of Lemma~\ref{lem:opial} is also verified. The weak convergence of the iterates to an element in $\zer(M+F)$ is therefore a consequence of the Opial Lemma.
\end{proof}

We will conclude the convergence analysis by proving that the Fast RFB algorithm does indeed have convergence rates of $o(\frac{1}{k})$ as $k \to +\infty$.

\begin{thm}\label{thm:o-rates}
Let $z_{\ast} \in \zer(M+F)$ and $\left(z_{k} \right)_{k \geq 0}$ be the sequence generated by Algorithm~\ref{algo:im}. The following holds as $ k \to +\infty $
\begin{align*}
\norm{z_{k} - z_{k-1}} =  o \left( \dfrac{1}{k} \right), \quad \|\xi_k + F(z_k) \|=  o \left( \frac{1}{k} \right), \quad \langle\xi_{k} + F(z_{k}),z_{k} - z_{\ast}\rangle  = & o \left( \frac{1}{k} \right),\\
r_{tan}(z_{k})= \dist (0, M(z_k) + F(z_k)) = o \left( \frac{1}{k} \right), \quad r_{fix}(z_{k})= \left\lVert z_k-J_{\gamma M}\big(z_k-\gamma F \left( z_{k} \right)\big)\right\rVert = & o \left( \dfrac{1}{k} \right).
\end{align*}
\end{thm}

\begin{proof}
Let $s \in \left(1+ \frac{\alpha}{4 c - \alpha}, 2 \right)$, and $0 \leq \underline{\lambda} \left( \alpha,c, s \right) < \overline{\lambda} \left( \alpha,c ,s \right) \leq \frac{s\alpha}{4}$ be the parameters provided by Proposition \ref{prop:split:lim}(iii) with the property that, for all $\underline{\lambda} \left( \alpha,c,s \right) < \lambda < \overline{\lambda} \left( \alpha,c,s \right)$ the sequence $\left( \cE_{\lambda,s,k} \right)_{k \geq 1}$, is convergent.

We choose $\lambda \in \left(\underline{\lambda} \left( \alpha,c,s \right), \overline{\lambda} \left( \alpha,c,s \right)\right)$ and set for all $k \geq 1$
\begin{equation*}
h_{s,k} := \dfrac{k^{2}}{2} \left( \left\lVert 2 \left( z_{k} - z_{k-1} \right) + s \gamma v_{k} \right\rVert ^{2} + (2-s)s \gamma^{2} \left\lVert v_{k} \right\rVert ^{2} \right).
\end{equation*}
We are going to show that
\begin{equation}
\label{thm:o-rates:lim-hk}
\lim\limits_{k \to + \infty} h_{s,k} = 0 .
\end{equation}
This assertion will immediately imply
\begin{equation*}
\lim\limits_{k \to +\infty} k \left\lVert 2 \left( z_{k} - z_{k-1} \right) + s \gamma v_{k} \right\rVert = \lim\limits_{k \to +\infty} k \left\lVert v_{k} \right\rVert = 0,
\end{equation*}
and further $\lim_{k \to + \infty} k\left\lVert z_{k} - z_{k-1} \right\rVert = 0$. The fact that
\begin{equation*}
\lim_{k \to + \infty} k \left\lVert \xi_{k} + F \left( z_{k} \right) \right\rVert = 0
\end{equation*}
will follow from \eqref{split:Lip}, \eqref{split:lim:dV-0} and \eqref{split:lim:V-dV}, since
\begin{equation*}
0 \leq \lim\limits_{k \to + \infty} k \left\lVert \xi_{k} + F \left( z_{k} \right) \right\rVert
\leq \lim\limits_{k \to + \infty} k \left\lVert v_{k} - v_{k-1} \right\rVert + \lim\limits_{k \to + \infty} k \left\lVert v_{k} \right\rVert = 0 .
\end{equation*}
The convergence statement above will imply that
\begin{equation*}
\lim_{k \to + \infty} k \ \dist \left(0, M(z_k)+ F \left( z_{k} \right) \right) = 0,
\end{equation*}
and, by using the boundedness of $(z_k)_{k \geq 0}$,
\begin{equation*}
\lim\limits_{k \to + \infty} k \left\langle \xi_{k} + F(z_{k}),z_{k} - z_{\ast} \right\rangle = 0.
\end{equation*}
The convergence rate
\begin{equation}
r_{fix}(z_{k})= \left\lVert z_k-J_{\gamma M}\big(z_k-\gamma F \left( z_{k} \right)\big)\right\rVert = o \left( \dfrac{1}{k } \right) \ \mbox{as} \ k \to +\infty\nonumber
\end{equation}
is a consequence of \eqref{tanfixres}.

What remains to be shown is that \eqref{thm:o-rates:lim-hk} does in fact hold. In view of \eqref{defi:E-lambda}, we have for all $k \geq 1$
\begin{equation*}
h_{s,k} = \cE_{\lambda,s,k} - 4 \lambda p_{k} -  \gamma^2\left(\big( \left( 2 - s\right)c + \left( \alpha - c \right)s \big)k + 2\left( \alpha - c \right)c\right) \left\lVert v_{k} \right\rVert ^{2} -4\left(\alpha - c \right)\lambda \gamma \langle z_k - z_{*},v_k\rangle.
\end{equation*}
From Proposition~\ref{prop:split:lim}(iii) and \eqref{split:lim-p-k}, we have $\lim_{k \to + \infty} \cE_{\lambda,s,k} \in \R$ and $\lim_{k \to + \infty} p_{k} \in \R$, respectively.
From \eqref{ineq-C4} it yields
\begin{equation*}
\lim\limits_{k \to +\infty} k \left\lVert v_{k} \right\rVert ^{2} = 0\quad \textnormal{and} \quad \lim\limits_{k \to +\infty} \left\lVert v_{k} \right\rVert ^{2} = 0,
\end{equation*}
which implies the existence of
\begin{equation*}
\lim\limits_{k \to + \infty} h_{s,k} \in \R.
\end{equation*}
To get the precise value, we recall that the summability results in Proposition~\ref{prop:split:lim}(i) guarantee that
\begin{align*}
\sum_{k \geq 1} \dfrac{1}{k} h_{s,k} \leq 4 \sum_{k\geq 1} \left\lVert z_{k} - z_{k-1} \right\rVert ^{2} + \frac{1}{2}(2+s)s \gamma^{2} \sum_{k \geq 1} k \left\lVert v_{k} \right\rVert ^{2} < + \infty .
\end{align*}
We must have $\lim_{k \to + \infty} h_{k} = 0$, and the proof is finished.
\end{proof}

\section{A fast primal-dual full splitting algorithm}\label{sec3}

In this section we will apply the Fast RFB algorithm to the solution of the saddle point problem and further to a convex optimization problem with linear cone constraints. For the resulting primal-dual full splitting methods, we will formulate the convergence and convergence rates statements that follow from those proved in the general setting.

\subsection{Convex-concave saddle point problems with smooth coupling term}\label{subsec31}

First, we consider the application the proposed fast algorithm to the saddle point problem \eqref{minmax}
\begin{equation}\label{Lag:saddle}
\min_{x\in \cX}\max_{\lambda \in \cY} \Psi \left( x , \lambda \right) \coloneqq f(x) + \Phi\left( x , \lambda \right) - g(\lambda),
\end{equation}
where $\cX$ and $ \cY$ are two real Hilbert spaces, $f \colon \cX \to \R \cup \left\lbrace + \infty \right\rbrace$ and $g \colon \cY \to \R \cup \left\lbrace + \infty \right\rbrace$ are proper, convex, and lower semicontinuous functions, and
$\Phi \colon \cX \times \cY \rightarrow \R$ is a differentiable function with Lipschitz continuous gradient, convex in $x$ and concave in $\lambda$.

Finding a solution to \eqref{minmax} reduces to the solving of the monotone inclusion \eqref{intro:pb:eq}, for $\cH:= \cX \times \cY$, $M : \cH \to 2^\cH, M(x,\lambda) = (\partial f(x), \partial g(\lambda))$, and $F : \cH \to \cH, F(x,\lambda) = (\nabla_x \Phi(x,\lambda), -\nabla_\lambda \Phi (x,\lambda))$. Let $L >0$ be the Lipschitz constant of $F$. Applying the Fast RFB algorithm to this particular setting leads to the following full splitting algorithms, which have the feature that each function is evaluated separately: the non-smooth tones $f$ and $g$ via their proximal operators, and the smooth one via a gradient step.

\begin{mdframed}
\begin{algo}
\label{algo:fpdVu}
Let
$$\alpha>2, \quad\frac{\alpha}{2}< c <\alpha - 1 ,
 \quad \textnormal{and}\quad 0<\gamma <\frac{1}{2L}.$$
For initial points $x_{0}, w_{1,0}, y_{1,0}\in \cX $, $\lambda_{0}, w_{2,0}, y_{2,0} \in \cY$, $x_{1}= \prox_{\gamma f}\left(y_{1,0}- \gamma\nabla_{x}\Phi \left( w_{1,0},w_{2,0}\right) \right)$ and $\lambda_{1} = \prox_{\gamma g}\left(y_{2,0} + \gamma \nabla_{\lambda} \Phi\left( w_{1,0},w_{2,0}\right) \right)$, we set
\begin{equation}
(\forall k \geq 1) \quad \begin{dcases}
& y_{1,k}	= x_{k}+\left(1-\frac{\alpha}{k+\alpha}\right)\left( x_{k}-x_{k-1} \right)+\left(1-\frac{c}{k+\alpha}\right) \left( y_{1,k-1} - x_{k} \right),\\
& y_{2,k} = \lambda_{k}+\left(1-\frac{\alpha}{k+\alpha}\right)\left( \lambda_{k}-\lambda_{k-1} \right)+\left(1-\frac{c}{k+\alpha}\right) \left( y_{2,k-1} - \lambda_{k} \right),\\
& w_{1,k}  =y_{1,k} - \left( y_{1,k-1} -x_k \right),\\
& w_{2,k}  =y_{2,k} - \left( y_{2,k-1} -\lambda_k \right),\\
& x_{k+1} = \prox_{\gamma f}\left(y_{1,k} - \gamma\nabla_{x}\Phi \left( w_{1,k},w_{2,k}\right) \right),\\
& \lambda_{k+1}	= \prox_{\gamma g}\left(y_{2,k} + \gamma \nabla_{\lambda} \Phi\left( w_{1,k},w_{2,k}\right) \right).\nonumber
\end{dcases}
\end{equation}
\end{algo}
\end{mdframed}
The following result is based on the theorems \ref{thm:conv} and \ref{thm:o-rates}.

\begin{thm}\label{thm:comp:conv}
Let $\left( x_{*} , \lambda_{*} \right) \in \cX \times \cY$ be a saddle point of \eqref{minmax}, and $\left(x_{k},\lambda_{k} \right)_{k \geq 0}$ the sequence generated by Algorithm \ref{algo:fpdVu}. The following statements are true:
\begin{itemize}
\item[(i)]	the sequence $\left(x_{k} , \lambda_{k} \right)_{k \geq 0}$ converges weakly to a saddle point of \eqref{minmax};

\item[(ii)] the following holds as $k \to +\infty$
\begin{align*}
\left\lVert x_{k} - x_{k-1} \right\rVert = o \left( \dfrac{1}{k} \right)	
\,\, \textrm{ and }\,\,
\left\lVert \lambda_{k} - \lambda_{k-1} \right\rVert & = o \left( \dfrac{1}{k} \right) \nonumber \\
\left\lVert u_k + \nabla_{x} \Phi(x_{k} , \lambda_{k}) \right\rVert  = o \left( \dfrac{1}{k} \right)
 \,\, \textrm{ and } \,\,
\left\lVert v_{k}- \nabla_{\lambda} \Phi(x_{k} , \lambda_{k}) \right\rVert & = o \left( \dfrac{1}{k} \right) \nonumber \\
\left\langle x_{k} - x_{*} , u_k + \nabla_{x} \Phi(x_{k} , \lambda_{k})\right\rangle + \left\langle \lambda_{k} - \lambda_{*} , v_{k} - \nabla_{\lambda} \Phi(x_{k} , \lambda_{k})\right\rangle & = o \left( \dfrac{1}{k} \right)\nonumber \\
\Psi \left( x_{k} , \lambda_{*} \right)-\Psi \left( x_{*} , \lambda_{k} \right) & = o \left( \dfrac{1}{k } \right),
\end{align*}
where, for all $k \geq 0$
\begin{align}
u_{k+1} :&= \frac{1}{\gamma}y_{1,k}-\nabla_{x} \Phi\left( w_{1,k},w_{2,k}\right) - \frac{1}{\gamma}x_{k+1} \in \partial f(x_{k+1}) \nonumber\\
v_{k+1}:&= \frac{1}{\gamma}y_{2,k}  + \nabla_{\lambda} \Phi\left( w_{1,k},w_{2,k}\right)-\frac{1}{\gamma}\lambda_{k+1} \in \partial g \left( \lambda_{k+1} \right).\nonumber
\end{align}
\end{itemize}
\end{thm}
\begin{proof}
Algorithm \ref{algo:fpdVu} is a special instance of Algorithm \ref{algo:im} when applied to the monotone inclusion \eqref{minimax:optcond}, for $z_k:=(x_{k},\lambda_{k})$, $y_k:=(y_{1,k},y_{2,k})$ and $w_k:=(w_{1,k},w_{2,k})$ for all $k \geq 0$. The third block Algorithm \ref{algo:im} is obviously equivalent for all $k \geq 1$ to
\begin{equation*}
\left\{
\begin{aligned}
&x_{k+1}=\prox_{\gamma f}\left(y_{1,k}  - \gamma \nabla_{x} \Phi\left( w_{1,k},w_{2,k}\right) \right),\\
&\lambda_{k+1}	 = \prox_{\gamma g}\left(y_{2,k}   + \gamma \nabla_{\lambda} \Phi\left( w_{1,k},w_{2,k}\right) \right).
\end{aligned}\right.
\end{equation*}
The sequence $(\xi_{k})_{k \geq 1}$ introduced in Proposition \ref{rmk:sub}, and defined for all $k \geq 1$ as
\begin{align*}
\xi_{k}=\frac{1}{\gamma}\left(y_{k-1} - z_{k}\right) - F(w_{k-1})
\end{align*}
plays a crucial role in the formulation of the convergence rates. In the context Algorithm \ref{algo:fpdVu}, we have for $\xi_k:= (u_k,v_{k})$ and all $k \geq 1$
\begin{equation}\label{eq:uv}
\left\{
\begin{aligned}
u_{k}&=\frac{1}{\gamma}y_{1,k-1} - \nabla_{x} \Phi\left( w_{1,k-1},w_{2,k-1}\right)-\frac{1}{\gamma}x_{k},\\
v_{k}&= \frac{1}{\gamma}y_{2,k-1} + \nabla_{\lambda} \Phi\left( w_{1,k-1},w_{2,k-1}\right)-\frac{1}{\gamma}\lambda_{k}.
\end{aligned}\right.
\end{equation}
Furthermore, $\xi_{k} \in M(z_k)$ becomes for all $k \geq 1$
\begin{equation*}
\left\{
\begin{aligned}
u_{k}&\in \partial f(x_{k}),\\
v_{k}&\in \partial g(\lambda_k).
\end{aligned}\right.
\end{equation*}

The weak convergence of the sequence $\left(x_{k} , \lambda_{k} \right)_{k \geq 0}$ to a saddle point of \eqref{minmax} is a direct consequence of Theorem \ref{thm:conv}. In addition, Theorem \ref{thm:o-rates} yields
\begin{align*}
\left\lVert z_k - z_{k-1} \right\rVert = o \left( \dfrac{1}{k } \right), \quad \left\lVert \xi_k + F(z_k)\right\rVert = o \left( \dfrac{1}{k} \right), \quad \left\langle z_k - z_{*} , \xi_k + F(z_k) \right\rangle = o \left( \dfrac{1}{k } \right)  \quad \mbox{as} \quad k \rightarrow +\infty,
\end{align*}
where $z_*:=(x_*, \lambda_*)$. From the first statement, we obtain
\begin{align*}
\left\lVert x_k - x_{k-1} \right\rVert = o \left( \dfrac{1}{k } \right) \quad \mbox{and} \quad \left\lVert \lambda_k-\lambda_{k-1} \right\rVert = o \left( \dfrac{1}{k } \right) \ \mbox{as} \ k \rightarrow +\infty.
\end{align*}
For $z_k=(x_{k},\lambda_{k})$ and $\xi_k= (u_k,v_k)$, given by \eqref{eq:uv}, the other two statements become
\begin{align*}
\left\lVert \begin{pmatrix} u_k + \nabla_{x} \Phi(x_{k} , \lambda_{k}), v_{k} - \nabla_{\lambda} \Phi(x_{k} , \lambda_{k}) \end{pmatrix} \right\rVert = o \left( \dfrac{1}{k} \right) \ \mbox{as} \ k \rightarrow +\infty,
\end{align*}
and
\begin{align*}
\left\langle \begin{pmatrix} x_{k} -  x_{*} , \lambda_{k} - \lambda_{*} \end{pmatrix} , \begin{pmatrix} u_k  + \nabla_{x} \Phi(x_{k} , \lambda_{k}), v_{k} - \nabla_{\lambda} \Phi(x_{k} , \lambda_{k}) \end{pmatrix} \right\rangle = o \left( \dfrac{1}{k } \right) \ \mbox{as} \ k \rightarrow +\infty,
\end{align*}
respectively. Obviously,
\begin{align*}
\left\lVert u_k  + \nabla_{x} \Phi(x_{k} , \lambda_{k}) \right\rVert = o \left( \dfrac{1}{k } \right)
\quad\mbox{and} \quad \left\lVert v_{k} - \nabla_{\lambda} \Phi(x_{k} , \lambda_{k}) \right\rVert  = o \left( \dfrac{1}{k} \right) \ \mbox{as} \ k \rightarrow +\infty.
\end{align*}
Using the convexity of $f$ and $g$, and the convexity and concavity of $\Phi$ in its first and second variables respectively, it follows for all $k \geq 1$ that
\begin{align*}
& \left\langle \begin{pmatrix} x_{k} -  x_{*} , \lambda_{k} - \lambda_{*} \end{pmatrix} , \begin{pmatrix} u_k  + \nabla_{x} \Phi(x_{k} , \lambda_{k}), v_{k} - \nabla_{\lambda} \Phi(x_{k} , \lambda_{k}) \end{pmatrix} \right\rangle\\
 = \ & \langle u_k + \nabla_{x} \Phi(x_{k} , \lambda_{k}),x_{k}-x_{*}\rangle+\langle v_{k}- \nabla_{\lambda} \Phi(x_{k} , \lambda_{k}) ,\lambda_{k}-\lambda_{*}\rangle \nonumber \\
= \ 	& \langle u_k , x_{k}-x_{*}\rangle + \langle\nabla_{x} \Phi(x_{k} , \lambda_{k}), x_{k}-x_{*} \rangle +\langle v_{k},\lambda_{k}-\lambda_{*}\rangle+\langle - \nabla_{\lambda} \Phi(x_{k} , \lambda_{k}), \lambda_{k}-\lambda_{*} \rangle \nonumber\\
\geq \ 	& f\left( x_{k} \right) - f(x_*) + \Phi\left( x_{k},\lambda_{*} \right) - \Phi\left( x_*,\lambda_{*} \right) + g(\lambda_k) - g \left( \lambda_{*} \right)  - \Phi\left( x_{*},\lambda_{k} \right) + \Phi\left( x_{*},\lambda_{*} \right) \nonumber\\
= \ 	& \Psi \left( x_{k} , \lambda_{*} \right)-\Psi \left( x_{*} , \lambda_{k} \right) \geq 0.
\end{align*}
This yields
\begin{align*}
&\Psi \left( x_{k} , \lambda_{*} \right)-\Psi \left( x_{*} , \lambda_{k} \right) = o \left( \dfrac{1}{k } \right) \quad \mbox{as} \quad k \rightarrow +\infty.
\end{align*}
\end{proof}

\subsection{Composite convex optimization problems}

The aim of this subsection is to show that the algorithm proposed in this paper leads to a fast primal-dual full splitting algorithm which solves the composite convex optimization problem
\begin{equation}\label{leeq1}
 \min_{x\in \cX} f(x)+g(Ax)+h(x),
\end{equation}
where $\mathcal{X}$ and $\mathcal{Y}$ are two real Hilbert spaces, $f, g \colon \cX \to \R \cup \left\lbrace + \infty \right\rbrace$ are proper, convex, and lower semicontinuous functions, $h: \cX\rightarrow \mathbb{R}$ is a convex and differentiable function such that $\nabla h$ is $L_{\nabla h}$-Lipschitz continuous, and $A: \cX\rightarrow \cY$ is a linear continuous operator.  The corresponding Fenchel dual problem is
\begin{equation}\label{Def:dual}
 \max_{\lambda\in \cY} -\big(f+h\big)^*\big(-A^*\lambda\big)-g^*(\lambda).
\end{equation}
If $(x_*, \lambda_*) \in \cX \times \cY$ is a primal-dual solution of \eqref{leeq1}-\eqref{Def:dual},  in other words, a solution of the KKT system
\begin{equation}
\label{condit}
\begin{cases}
0 \in  \partial f \left(x\right) +\nabla h \left(x\right) + A^*\lambda	\\
A x \in \partial g^*(\lambda),				
\end{cases}
\end{equation}
then $x_*\in\cX$ is an optimal solution of the primal problem,  $\lambda_*\in\cY$ is an optimal solution of the dual problem,  and strong duality holds.  Viceversa,  under suitable constraint qualification (see \cite{Bot2010,BauschkeCombettes2}),  if $x_* \in \cX$ is a solution of \eqref{leeq1}, then there exists an optimal solution $\lambda_* \in \cY$ of \eqref{Def:dual} such that $(x_*, \lambda_*)$ solves \eqref{condit}.  It is easy to see that the solutions of \eqref{condit} are nothing else than the saddle points of the Lagrangian
\begin{equation}\label{Lag:line}
{\cal L} \left( x,\lambda \right):= f \left( x \right)+h \left( x \right) + \left\langle \lambda , Ax  \right\rangle - g^*(\lambda),
\end{equation}
which provides a compelling motivation for treating this problem as a specific instance of the framework
developed in the previous subsection.  Thus, the algorithm and the convergence theorem of the previous subsection lead to the following statements, respectively.
\vskip 2mm
\begin{mdframed}
\begin{algo}
\label{algo:line}
Let
$$\alpha>2, \quad\frac{\alpha}{2}< c <\alpha - 1 ,
 \quad \textnormal{and}\quad 0<\gamma < \frac{1}{2\sqrt{\left(L_{\nabla h}+\left\lVert A  \right\rVert\right)^2 + \left\lVert A  \right\rVert^2}}.$$
For initial points $x_{0}, w_{1,0}, y_{1,0}\in \cX $, $\lambda_{0}, w_{2,0}, y_{2,0} \in \cY$, $x_{1}= \prox_{\gamma f}\left(y_{1,0}-\gamma \nabla h\left( w_{1,0} \right) - \gamma A^{*}w_{2,0}  \right)$ and $\lambda_{1}=\prox_{\gamma g^{*}} \left(y_{2,0} +\gamma A w_{1,0} \right)$,  we set
\begin{equation}
(\forall k \geq 1) \quad \begin{dcases}
& y_{1,k}	= x_{k}+\left(1-\frac{\alpha}{k+\alpha}\right)\left( x_{k}-x_{k-1} \right)+\left(1-\frac{c}{k+\alpha}\right) \left( y_{1,k-1} - x_{k} \right),\\
& y_{2,k}	= \lambda_{k}+\left(1-\frac{\alpha}{k+\alpha}\right)\left( \lambda_{k}-\lambda_{k-1} \right)+\left(1-\frac{c}{k+\alpha}\right) \left( y_{2,k-1} - \lambda_{k} \right),\\
& w_{1,k} =y_{1,k} - \left( y_{1,k-1} -x_k \right),\\
& w_{2,k} =y_{2,k} - \left( y_{2,k-1} -\lambda_k \right),\\
& x_{k+1}	= \prox_{\gamma f}\left(y_{1,k}-\gamma \nabla h\left( w_{1,k} \right) - \gamma A ^{*}w_{2,k} \right),\\
& \lambda_{k+1} = \prox_{\gamma g^{*}}\left(y_{2,k}+ \gamma  A w_{1,k}\right).\nonumber
\end{dcases}
\end{equation}
\end{algo}
\end{mdframed}
\begin{thm}\label{thm:line}
Let $\left( x_{*} , \lambda_{*} \right) \in \cX \times \cY$ be a primal-dual optimal solution of \eqref{leeq1}-\eqref{Def:dual},  and $\left(x_{k},\lambda_{k} \right)_{k \geq 0}$ the sequence generated by Algorithm \ref{algo:line}. The following statements are true:
\begin{itemize}
\item[(i)]	the sequence $\left(x_{k} , \lambda_{k} \right)_{k \geq 0}$ converges weakly to a primal-dual optimal solution of \eqref{leeq1}-\eqref{Def:dual};
\item[(ii)] the following holds as $k \to +\infty$
\begin{align*}
\left\lVert x_{k} - x_{k-1} \right\rVert = o \left( \dfrac{1}{k} \right)	
\quad \textrm{and} \quad
\left\lVert \lambda_{k} - \lambda_{k-1} \right\rVert & = o \left( \dfrac{1}{k} \right) \nonumber \\
\left\lVert u_k + \nabla h\left( x_{k} \right) + A^{*}\lambda_{k}\right\rVert = o \left( \dfrac{1}{k} \right)
\quad \textrm{and} \quad
\left\lVert v_{k} -A x_{k} \right\rVert & = o \left( \dfrac{1}{k} \right) \nonumber \\
{\cal L} \left( x_{k} , \lambda_{*} \right)- {\cal L}\left( x_{*} , \lambda_{k} \right) & = o \left( \dfrac{1}{k } \right),
\end{align*}
where, for every $k \geq 0$
\begin{align*}
u_{k+1} &:= \frac{1}{\gamma}y_{1,k}- \nabla h\left( w_{1,k} \right)- A^{*}w_{2,k} - \frac{1}{\gamma}x_{k+1} \in \partial f(x_{k+1}),\\
v_{k+1}&:= \frac{1}{\gamma}y_{2,k}  + A w_{1,k}-\frac{1}{\gamma}\lambda_{k+1}\in \partial g^{*}(\lambda_{k+1}).
\end{align*}
\end{itemize}
\end{thm}

\begin{rmk}
To the best of our knowledge, these are the strongest convergence rate results among primal–dual full-splitting algorithms for composite convex optimization problems.  In the merely convex case, the algorithms proposed in the literature typically achieve an ergodic convergence rate of$\bO\left(\frac{1}{k}\right)$ as $k \to +\infty$ for the primal–dual gap (see, e.g., \cite{Chambolle2016Pock}). In contrast, Algorithm~\ref{algo:line} establishes two advances: (i) global convergence of the entire primal–dual sequence to a primal–dual optimal solution, and (ii) a nonergodic convergence rate of $o\left(\frac{1}{k}\right)$ as $k \to +\infty$ for the primal and dual discrete velocities,  the tangent residual and the primal-dual gap.
\end{rmk}

\subsection{Convex optimization problems with linear cone constraints}\label{subsec33}

In this subsection, we will study the optimization problem \eqref{convexcone1}
\begin{align*}
\min \ & f \left( x \right)+h \left( x \right),\\
\textrm{subject to} \	& Ax - b \in -\mathcal{K}
\end{align*}
where $\cX$ and $\cY$ are real Hilbert spaces, $\mathcal{K}$ is a nonempty, convex and closed cone in $\cY$, $f \colon \cX \rightarrow \R \cup \left\lbrace + \infty \right\rbrace$ is a proper, convex, and lower semicontinuous function, $h \colon \cX \rightarrow \R$ is a convex and differentiable function such that $\nabla h$ is $L_{\nabla h}$-Lipschitz continuous, and $A \colon \cX\rightarrow \cY$ is a linear continuous operator.

We aim to design efficient algorithms for detecting primal-dual optimal solutions  $(x_*, \lambda_*)$  of \eqref{convexcone1}, which correspond to solutions of the associated system of optimality conditions
\begin{equation}
\label{intro:cone-Lag}
\begin{cases}
0 \in  \partial f \left(x\right) +\nabla h \left(x\right) + A^*\lambda	\\
A x - b  \in N_{\mathcal{K}^{*}}(\lambda),				
\end{cases}
\end{equation}
where $\cK^{*}:=\{\lambda \in \cY : \langle \lambda, \zeta \rangle \geq 0 \ \forall \zeta \in \cK \}$ denotes the dual cone of $\cK$. Given a primal-dual optimal solution $(x_*, \lambda_*)$  of \eqref{convexcone1}, $x_*$ is an optimal solution of \eqref{convexcone1} and $\lambda_*$ is an optimal solution of the Lagrange dual problem of \eqref{convexcone1}.

Solving \eqref{intro:cone-Lag} is equivalent to finding the saddle points of the associated Lagrangian
\begin{equation*}
{\cal L} \left( x,\lambda \right):= f \left( x \right)+h \left( x \right) + \left\langle \lambda , Ax - b \right\rangle - \delta_{\cK^{*}}\left(\lambda\right),
\end{equation*}
which is a specific case of \eqref{Lag:line}. Therefore, it follows from Algorithm \ref{algo:line} and Theorem \ref{thm:line} that obtain
 the following statements.
\vskip 2mm
\begin{mdframed}
\begin{algo}
\label{algo:cone}
Let
$$\alpha>2, \quad\frac{\alpha}{2}< c <\alpha - 1 ,
 \quad \textnormal{and}\quad 0<\gamma < \frac{1}{2\sqrt{\left(L_{\nabla h}+\left\lVert A  \right\rVert\right)^2 + \left\lVert A  \right\rVert^2}}.$$
For initial points $x_{0}, w_{1,0}, y_{1,0}\in \cX $, $\lambda_{0}, w_{2,0}, y_{2,0} \in \cY$, $x_{1}= \prox_{\gamma f}\left(y_{1,0}-\gamma \nabla h\left( w_{1,0} \right) - \gamma A^{*}w_{2,0}  \right)$ and $\lambda_{1}=P_{\cK^{*}} \left(y_{2,0} +\gamma \left(A w_{1,0} -b\right)\right)$. We set
\begin{equation}
(\forall k \geq 1) \quad \begin{dcases}
& y_{1,k}	= x_{k}+\left(1-\frac{\alpha}{k+\alpha}\right)\left( x_{k}-x_{k-1} \right)+\left(1-\frac{c}{k+\alpha}\right) \left( y_{1,k-1} - x_{k} \right),\\
& y_{2,k}	= \lambda_{k}+\left(1-\frac{\alpha}{k+\alpha}\right)\left( \lambda_{k}-\lambda_{k-1} \right)+\left(1-\frac{c}{k+\alpha}\right) \left( y_{2,k-1} - \lambda_{k} \right),\\
& w_{1,k} =y_{1,k} - \left( y_{1,k-1} -x_k \right),\\
& w_{2,k} =y_{2,k} - \left( y_{2,k-1} -\lambda_k \right),\\
& x_{k+1}	= \prox_{\gamma f}\left(y_{1,k}-\gamma \nabla h\left( w_{1,k} \right) - \gamma A ^{*}w_{2,k} \right),\\
& \lambda_{k+1} = \textnormal{P}_{\mathcal{K}^{*}}\left(y_{2,k}+ \gamma \left( A w_{1,k} -b\right)\right).\nonumber
\end{dcases}
\end{equation}
\end{algo}
\end{mdframed}

\begin{thm}\label{thm:comp:cone}
Let $\left( x_{*} , \lambda_{*} \right) \in \cX \times \cY$ be a primal-dual optimal solution of \eqref{convexcone1}, and $\left(x_{k},\lambda_{k} \right)_{k \geq 0}$ the sequence generated by Algorithm \ref{algo:cone}. The following statements are true:
\begin{itemize}
\item[(i)]	the sequence $\left(x_{k} , \lambda_{k} \right)_{k \geq 0}$ converges weakly to a primal-dual optimal solution of \eqref{convexcone1};

\item[(ii)] the following holds as $k \to +\infty$
\begin{align*}
\left\lVert x_{k} - x_{k-1} \right\rVert = o \left( \dfrac{1}{k} \right)	
\quad \textrm{and} \quad
\left\lVert \lambda_{k} - \lambda_{k-1} \right\rVert & = o \left( \dfrac{1}{k} \right) \nonumber \\
\left\lVert u_k + \nabla h\left( x_{k} \right) + A^{*}\lambda_{k}\right\rVert = o \left( \dfrac{1}{k} \right)
\quad \textrm{and} \quad
\left\lVert v_{k} -A x_{k}+b \right\rVert & = o \left( \dfrac{1}{k} \right) \nonumber \\
{\cal L} \left( x_{k} , \lambda_{*} \right)- {\cal L}\left( x_{*} , \lambda_{k} \right) & = o \left( \dfrac{1}{k } \right),\nonumber \\
\left\lvert \left( f+h \right) \left( x_{k} \right) - \left( f+h \right) \left( x_{*} \right) \right\rvert = o \left( \dfrac{1}{k} \right) \quad \textrm{and} \quad \left| \langle \lambda_{k}, A x_{k} - b\rangle \right| & = o \left( \frac{1}{k} \right),
\end{align*}
where, for every $k \geq 0$
\begin{align*}
u_{k+1} &:= \frac{1}{\gamma}y_{1,k}- \nabla h\left( w_{1,k} \right)- A^{*}w_{2,k} - \frac{1}{\gamma}x_{k+1} \in \partial f(x_{k+1}),\\
v_{k+1}&:= \frac{1}{\gamma}y_{2,k}  + A w_{1,k}-\frac{1}{\gamma}\lambda_{k+1}-b\in N_{\mathcal{K}^{*}}(\lambda_{k+1}).
\end{align*}
\end{itemize}
\end{thm}
\begin{proof}
Given a primal-dual optimal solution $\left( x_{*} , \lambda_{*} \right) \in \cX \times \cY$ of \eqref{convexcone1}, it holds (see, for instance, \cite[Proposition 27.17]{BauschkeCombettes2}) $\lambda_{*}\in \mathcal{K}^{*}$, $Ax_{*} -b \in -\mathcal{K}$ and $\langle \lambda_{*}, Ax_{*} - b\rangle=0$. The weak convergence of the sequence of primal-dual iterates to a primal-dual solution of \eqref{convexcone1}, and the statements in the first three blocks of (ii) follow directly from Theorem \ref{thm:comp:conv}.

In addition, we have for all $k \geq 1$ that $v_k \in N_{\mathcal{K}^{*}}(\lambda_{k})$, which implies that $\lambda_{k}\in \mathcal{K}^{*}$, $v_{k} \in -\mathcal{K}$ and $\langle \lambda_{k}, v_{k}\rangle=0$. Therefore,
\begin{equation*}
\left| \langle \lambda_{k}, A x_{k} - b\rangle \right| =\left| \langle \lambda_{k}, A x_{k} - b - v_{k}\rangle \right|\leq \|\lambda_{k}\|\|A x_{k} - b - v_{k}\|.
\end{equation*}
Since $(\lambda_k)_{k \geq 0}$ is bounded and $\left\lVert v_{k} - A x_k +b \right\rVert = o \left( \frac{1}{k} \right)$, we conclude that $\left| \langle \lambda_{k}, A x_{k} - b\rangle \right| = o \left( \frac{1}{k} \right)$ as $k \to +\infty$.

Since $u_k \in \partial f(x_k)$, we have for all $k \geq 1$
\begin{align*}
(f+h)(x_k) - (f+h)(x_*) \leq & \ \langle u_k + \nabla h(x_k), x_k -x_* \rangle\\
= & \ \langle u_k + \nabla h(x_k) + A^*\lambda_k, x_k -x_* \rangle + \langle \lambda_k, b- Ax_k \rangle \\
\leq & \ \|u_k + \nabla h(x_k) + A^*\lambda_k\| \|x_k - x_*\| + |\langle \lambda_k, b- Ax_k \rangle|.
\end{align*}
On the other hand, since $-A^{*}\lambda_{*} \in \partial (f+h) \left( x_{*} \right)$, $\langle \lambda_{*}, Ax_{*} - b\rangle=0$ and $\langle \lambda_{*},v_k\rangle \leq 0$, we for all $k \geq 1$
\begin{align*}
(f+h) \left( x_{k} \right) &\geq (f+h) \left( x_{*} \right) - \langle A^{*}\lambda_{*} , x_{k} - x_{*} \rangle = (f+h)(x_{*}) - \langle \lambda_{*} , Ax_{k} - b \rangle \\
&= (f+h)(x_{*}) -\langle \lambda_{*} , - v_{k}+ A x_{k} - b \rangle - \langle \lambda_{*},v_k\rangle \\
&\geq (f+h)(x_{*}) -\| \lambda_{*} \|\| v_{k} -A x_{k} + b \|.
\end{align*}
Therefore, we obtain for every $k \geq 1$
\begin{align*}
& \ |(f+h) \left( x_{k} \right) - (f+h)(x_{*})| \\
\leq & \ \max \{\|u_k + \nabla h(x_k) + A^*\lambda_k\| \|x_k - x_*\| + |\langle \lambda_k, b- Ax_k \rangle|,\| \lambda_{*} \|\| v_{k} -A x_{k} + b \|\}.
\end{align*}
Since $(x_k)_{k \geq 0}$ is bounded, the right-hand side converges to zero with a convergence rate of $o \left( \frac{1}{k} \right)$ as $k \to +\infty$.
\end{proof}

\section{Numerical experiments}\label{sec4}

In this section, we present numerical experiments to illustrate the convergence rates established for the proposed fast method and compare our algorithm with those in the existing literature.

\subsection{The role of the algorithm parameters}
In this subsection, we investigated the influence of the parameters $\alpha$ and $c$ on the algorithm's convergence behavior. Consider the convex optimization problem
\begin{align}\label{num:cone}
\min & \left\lVert x \right\rVert_{1} + \frac{1}{2}\langle x,Hx\rangle - \langle x,h\rangle, \\
\textrm{such that} \,	& Ax - b \in -\R_{+}^{n},
\end{align}
where $\R_{+}^{n}$ denotes the nonnegative orthant of in $\R^{n}$,
\[A: =\frac{1}{4}
\begin{pmatrix}
 &  &  & -1  & 1 \\
 &  &  \reflectbox{$\ddots$} & \reflectbox{$\ddots$} & \\
 & -1 & 1 &  &  \\
 -1 & 1 &  &   &  \\
 1 &  &   &  & \\
\end{pmatrix}\in \R^{n\times n},\,\, H:=2A^{T}A, \,\, b:=\frac{1}{4}\begin{pmatrix} 1\\ 1\\ \vdots \\ 1\\ -4 \\ \end{pmatrix}\in \R^n \,\, {\rm and}\,\, h:=\frac{1}{4}\begin{pmatrix} 0\\ 0\\ \vdots \\ 0\\ 1 \\ \end{pmatrix}\in \R^n.
\]
The associated Lagrangian is
\begin{equation*}
\Lag \left( x , \lambda \right) = \left\lVert x \right\rVert_{1} + \frac{1}{2}\langle x,Hx\rangle - \langle x,h\rangle + \left\langle \lambda , Ax - b \right\rangle - \delta_{\R_{+}^{n}} \left( \lambda \right) .
\end{equation*}
The numerical experiments were conducted with  $n=1000$ and a maximum of $10^4$ iterations. Guided by theoretical insights, various choices for $\alpha$ and $c$ were explored, with the stepsize selected as
\begin{equation*}
\gamma = \frac{0.99}{2\sqrt{\left(\left\lVert H \right\rVert + \left\lVert A  \right\rVert\right)^2 + \left\lVert A  \right\rVert^2}}.
\end{equation*}

\begin{figure}[ht!]
	\centering
	\begin{subfigure}[b]{\textwidth}
		\centering
        \includegraphics[width=0.48\linewidth]{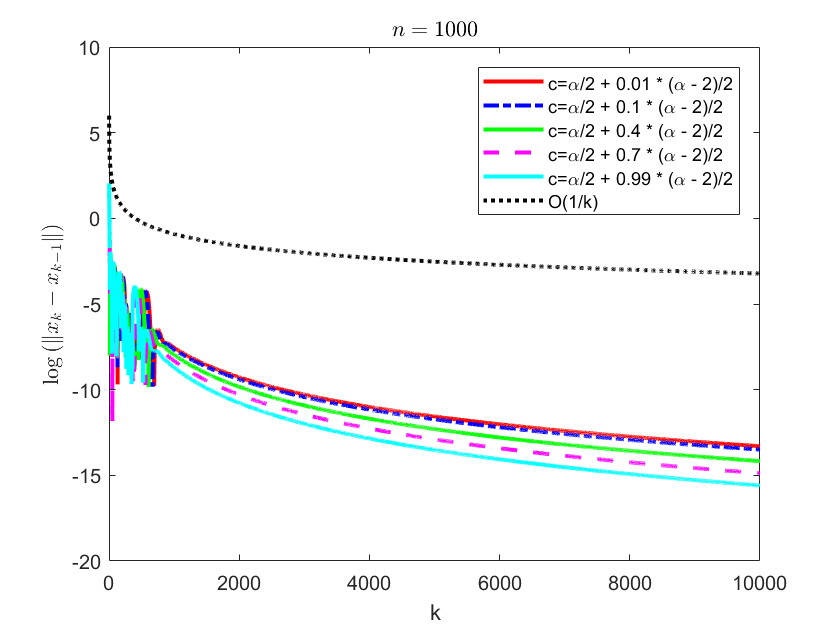}
        \hfill
		\includegraphics[width=0.48\linewidth]{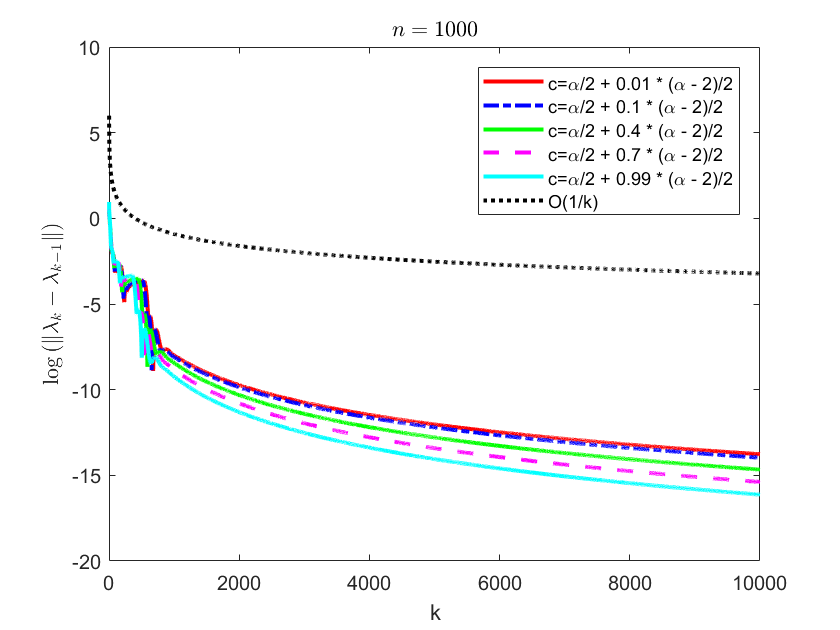}
  \end{subfigure}
	\vskip\baselineskip
	\begin{subfigure}[b]{\textwidth}
		\centering
        \includegraphics[width=0.5\linewidth]{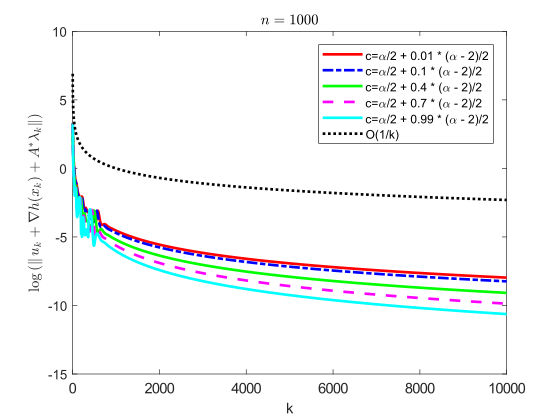}%
        \hfill
		\includegraphics[width=0.5\linewidth]{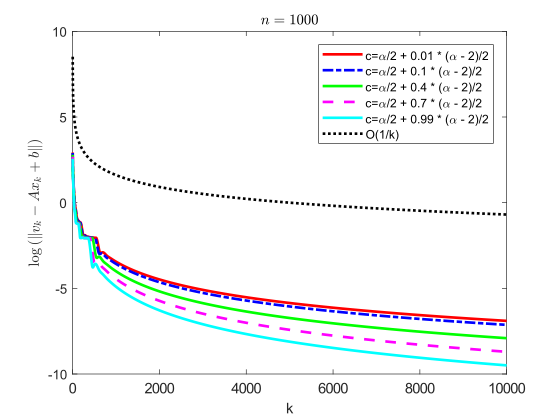}
  \end{subfigure}
	\vskip\baselineskip
	\begin{subfigure}[b]{\textwidth}
		\centering
        \includegraphics[width=0.5\linewidth]{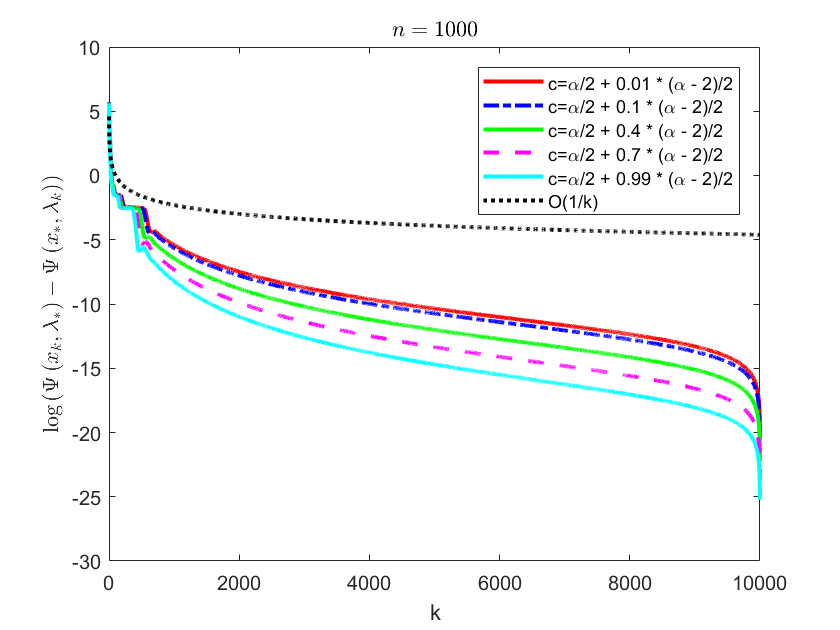}%
        \hfill
		\includegraphics[width=0.5\linewidth]{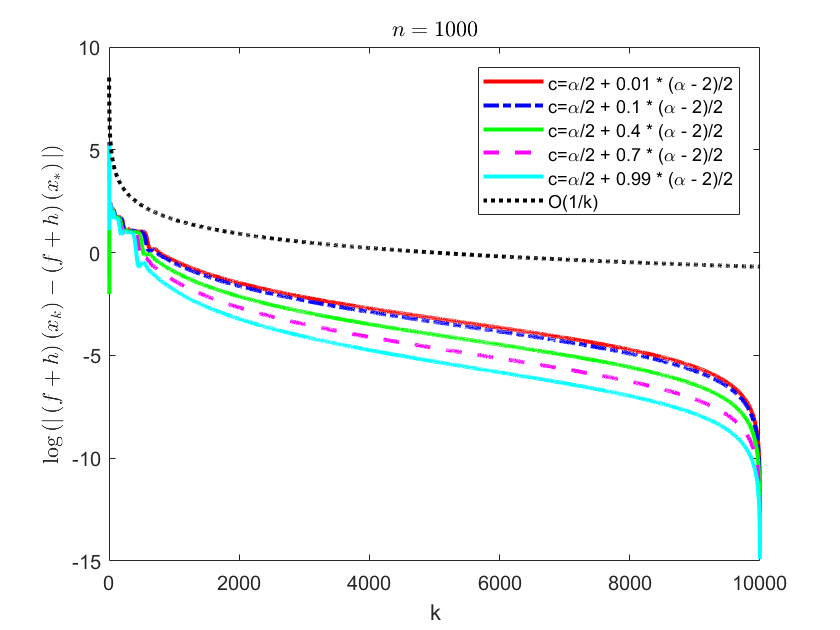}%
	\end{subfigure}
	\caption{The impact of the parameter $c$ on the convergence behavior of the discrete velocity, the tangent residual, the primal-dual gap, and the function values when $\alpha=3$.}\label{figure1}
\end{figure}

Figures \ref{figure1}-\ref{figure4} depict the convergence performance in terms of the discrete velocity, tangent residual, primal-dual gap, and function values, for $\alpha =3, 5, 10$ and $20$, and different values of $c \in \left(\frac{\alpha}{2}, \alpha -1 \right)$.  All the results are plotted on a semilog scale, namely, by taking the logarithm of the measured quantities.
The results demonstrate that increasing $c$ within the allowable range enhances the convergence behavior of the proposed algorithm. Furthermore, larger values of $\alpha$ significantly improve the algorithm's convergence, with the impact of $c$ lead to improved convergence of the algorithm and that for higher values of $\alpha$ the impact on $c$ on performance becoming more pronounced as $\alpha$ increases.  We also adopt a convergence rate of $\bO\left(\frac{1}{k}\right)$ as a predefined reference standard.  When compared against this baseline, our methods exhibit substantially superior performance.
\begin{figure}[ht!]
	\centering
	\begin{subfigure}[b]{\textwidth}
		\centering
        \includegraphics[width=0.48\linewidth]{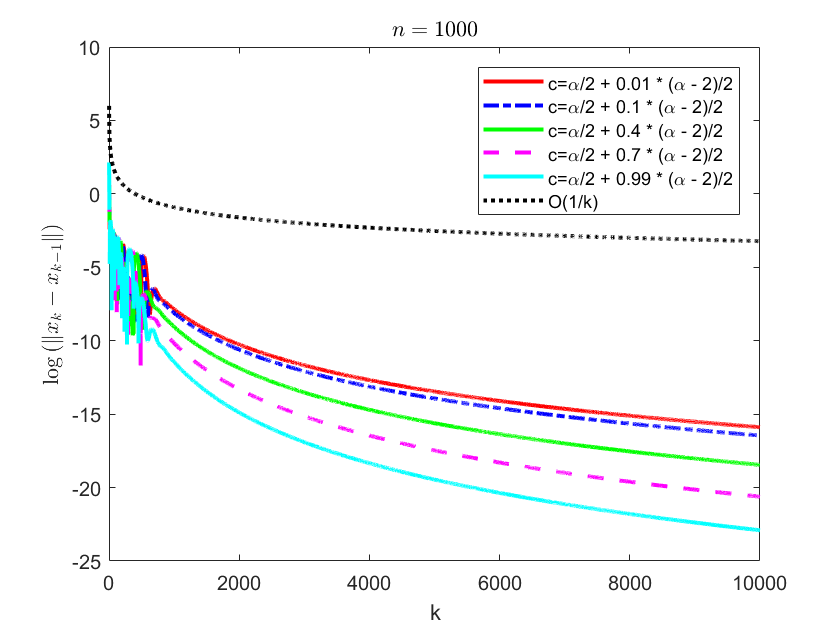}
        \hfill
		\includegraphics[width=0.48\linewidth]{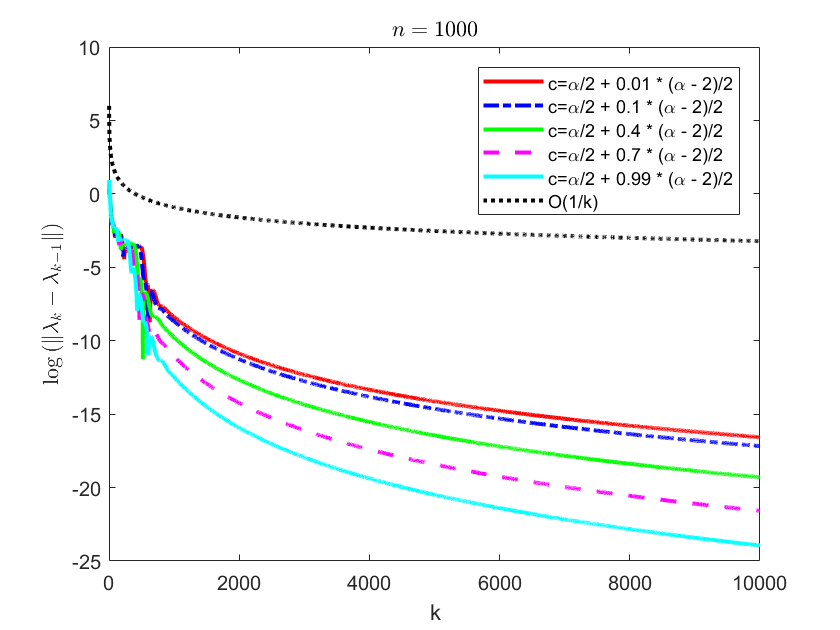}
    \end{subfigure}
	\vskip\baselineskip
	\begin{subfigure}[b]{\textwidth}
		\centering
        \includegraphics[width=0.48\linewidth]{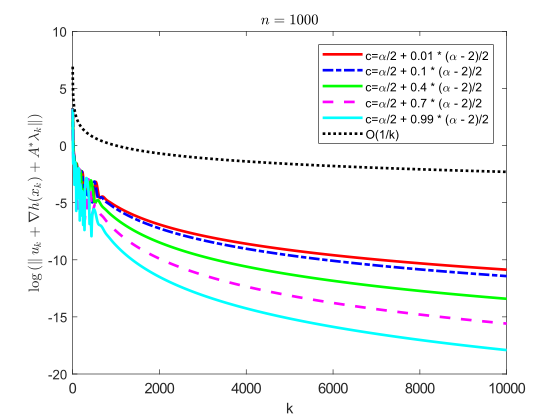}%
	    \hfill
		\includegraphics[width=0.48\linewidth]{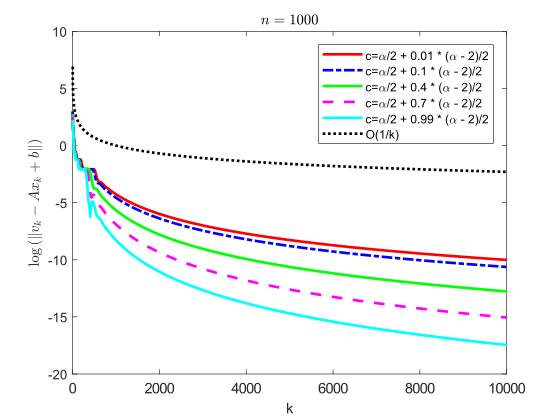}
    \end{subfigure}
	\vskip\baselineskip
	\begin{subfigure}[b]{\textwidth}
		\centering
        \includegraphics[width=0.48\linewidth]{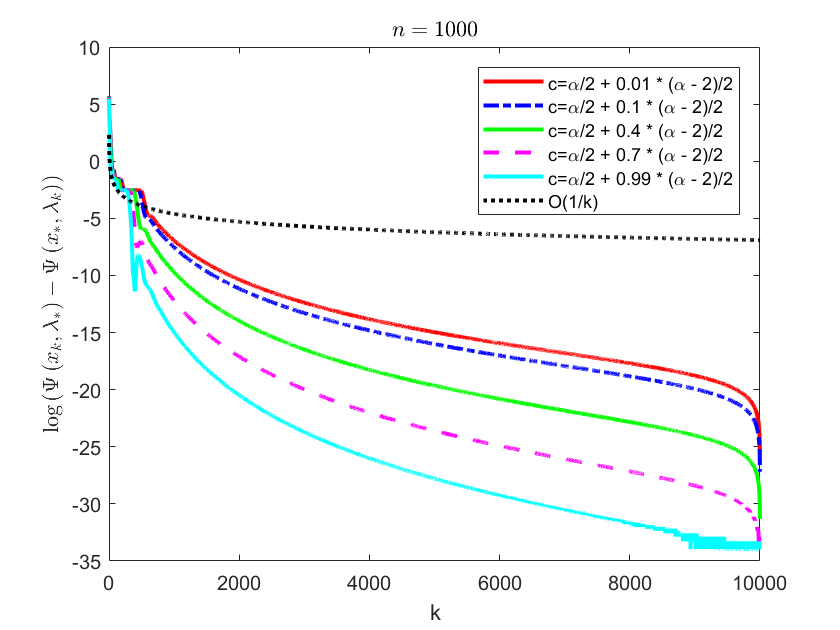}%
        \hfill
		\includegraphics[width=0.48\linewidth]{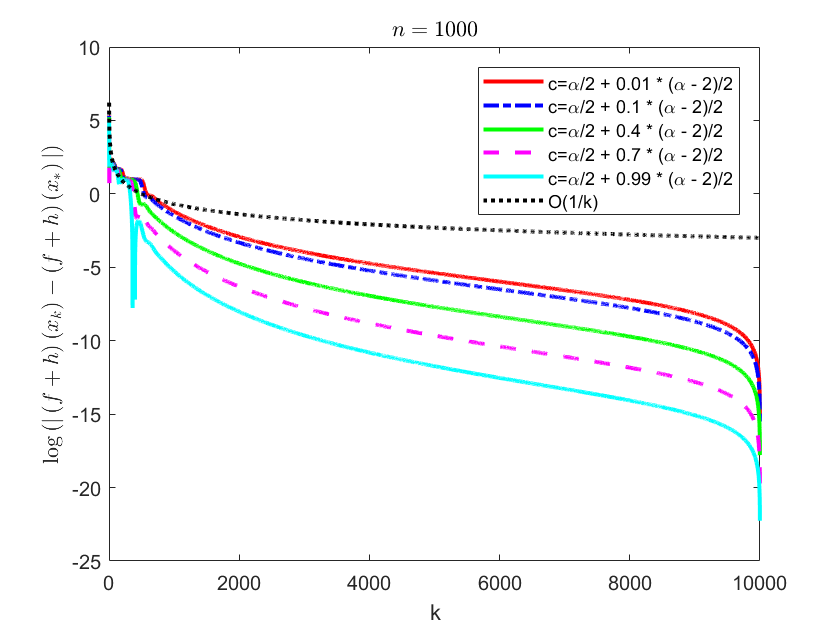}%
	\end{subfigure}
	\caption{The impact of the parameter $c$ on the convergence behavior of the discrete velocity, the tangent residual, the primal-dual gap, and the function values when $\alpha=5$.}
\end{figure}

\begin{figure}[ht!]
	\centering
	\begin{subfigure}[b]{\textwidth}
		\centering
        \includegraphics[width=0.48\linewidth]{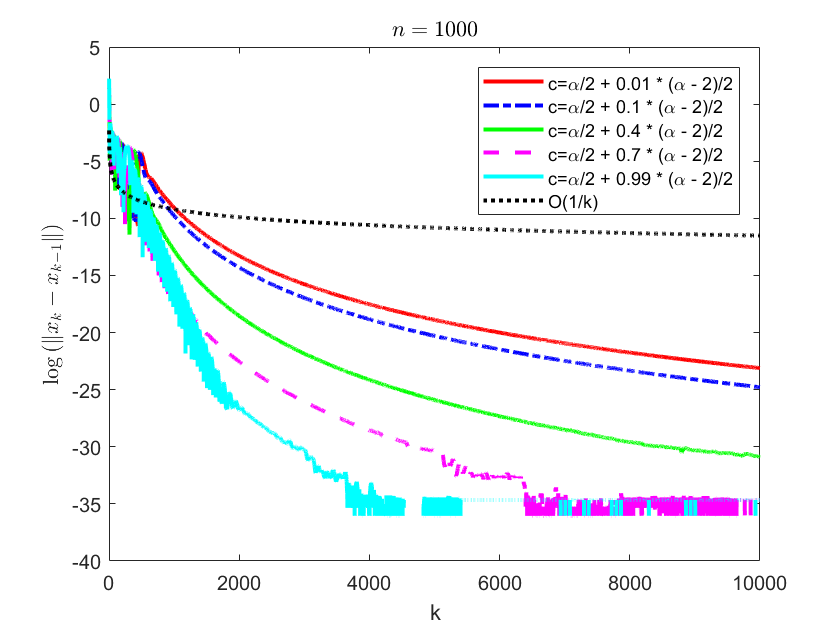}
        \hfill
		\includegraphics[width=0.48\linewidth]{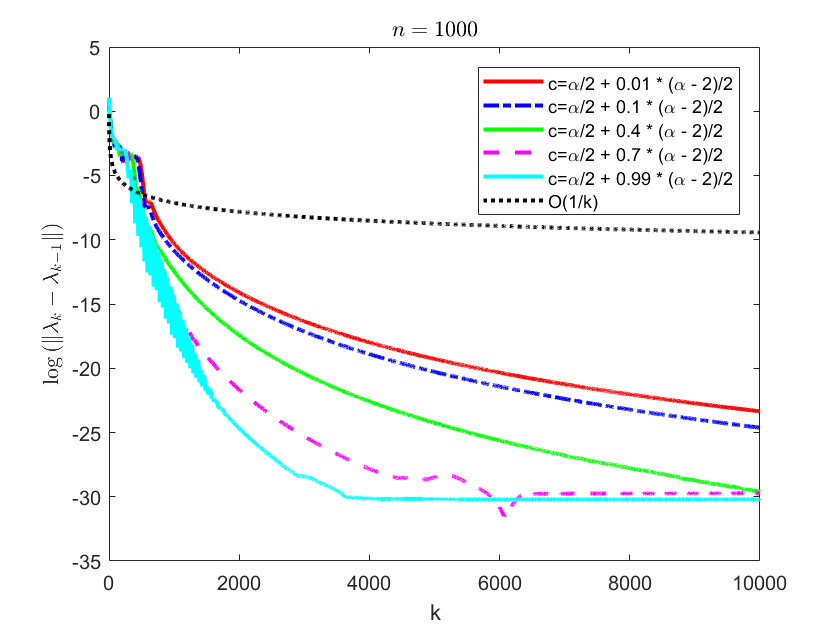}
  \end{subfigure}
	\vskip\baselineskip
	\begin{subfigure}[b]{\textwidth}
		\centering
        \includegraphics[width=0.48\linewidth]{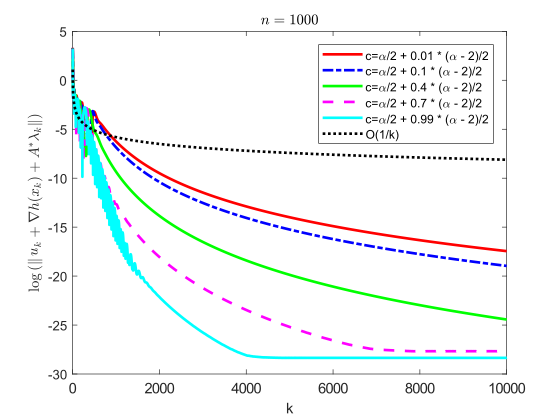}%
        \hfill
		\includegraphics[width=0.48\linewidth]{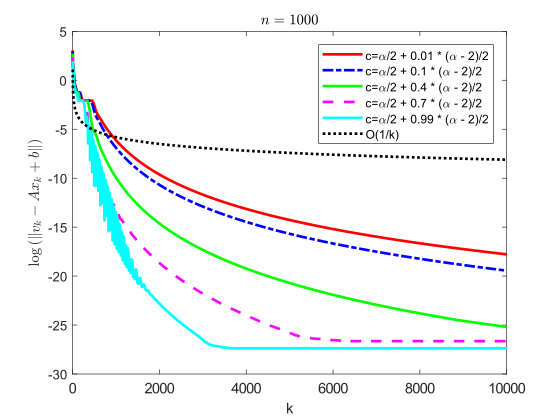}
  \end{subfigure}
	\vskip\baselineskip
	\begin{subfigure}[b]{\textwidth}
		\centering
        \includegraphics[width=0.48\linewidth]{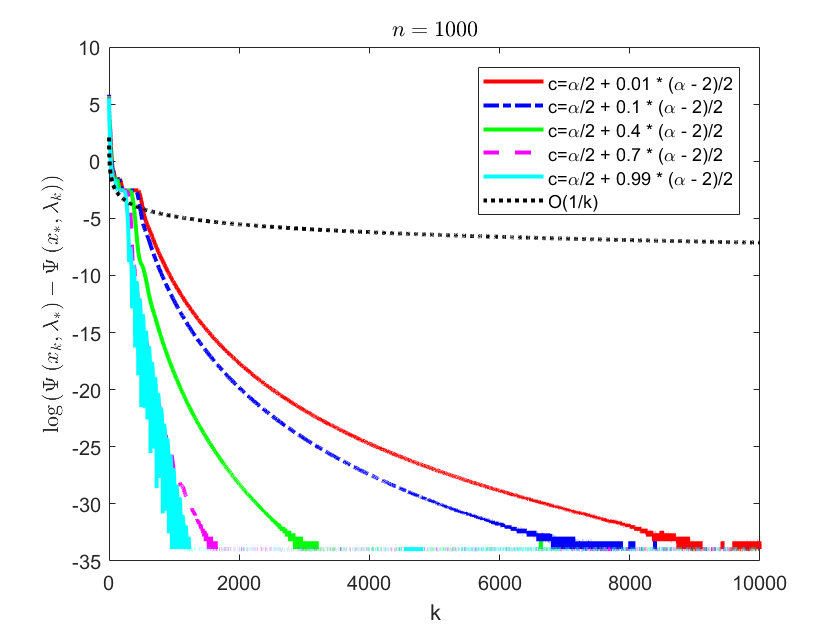}%
        \hfill
		\includegraphics[width=0.48\linewidth]{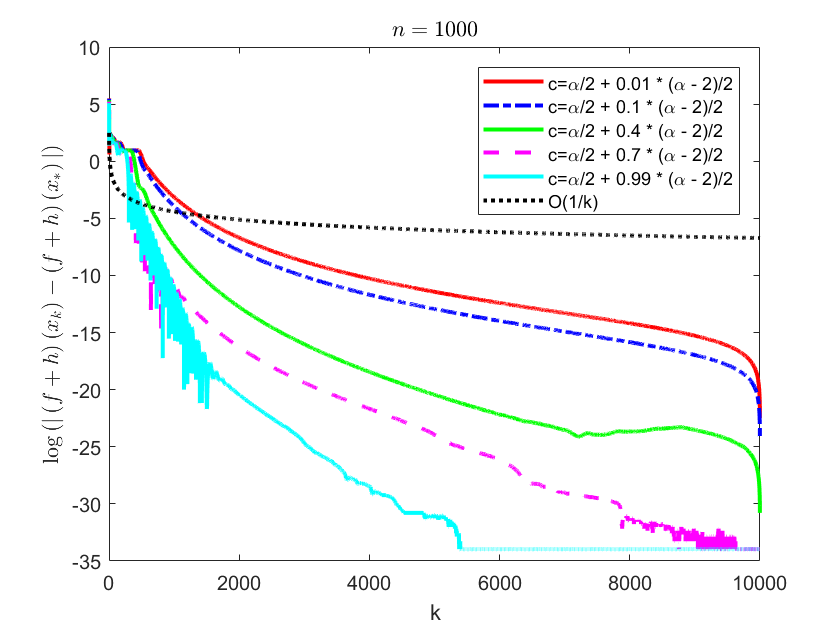}%
	\end{subfigure}
	\caption{The impact of the parameter $c$ on the convergence behavior of the discrete velocity, the tangent residual, the primal-dual gap, and the function values when $\alpha=10$.}
\end{figure}

\begin{figure}[ht!]
	\centering
	\begin{subfigure}[b]{\textwidth}
		\centering
        \includegraphics[width=0.48\linewidth]{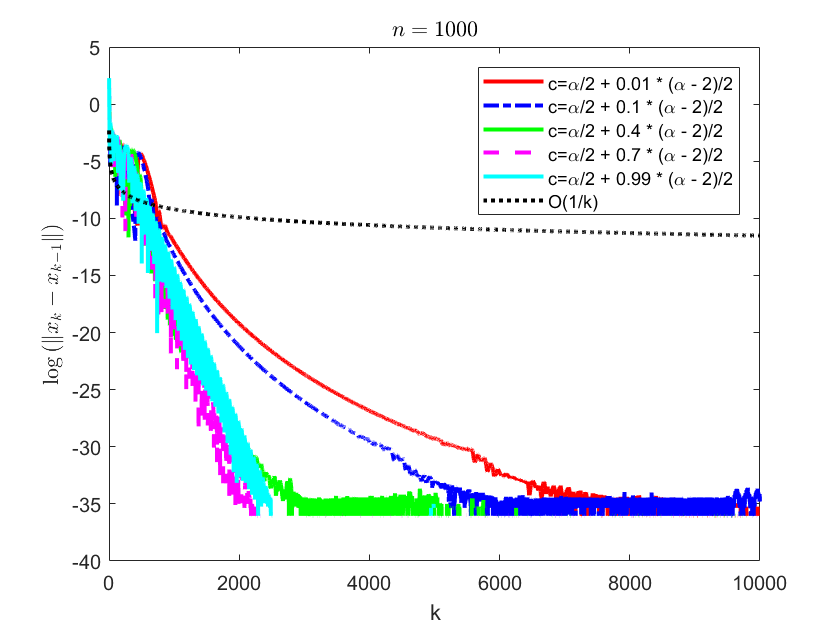}
        \hfill
		\includegraphics[width=0.48\linewidth]{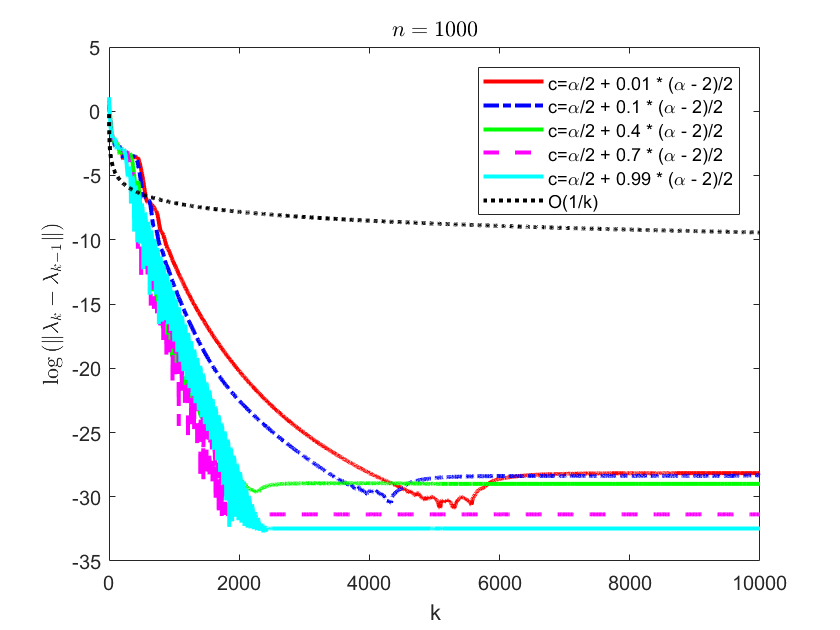}
    \end{subfigure}
	\vskip\baselineskip
	\begin{subfigure}[b]{\textwidth}
		\centering
        \includegraphics[width=0.48\linewidth]{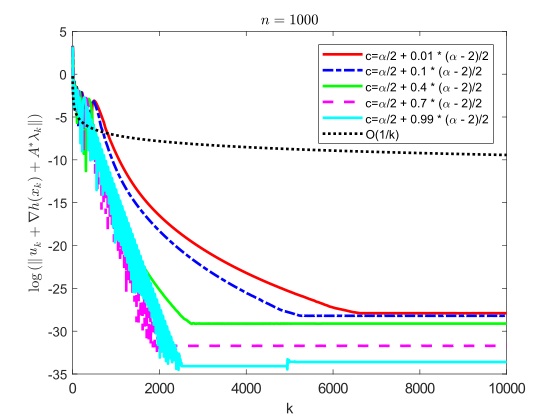}%
	    \hfill
		\includegraphics[width=0.48\linewidth]{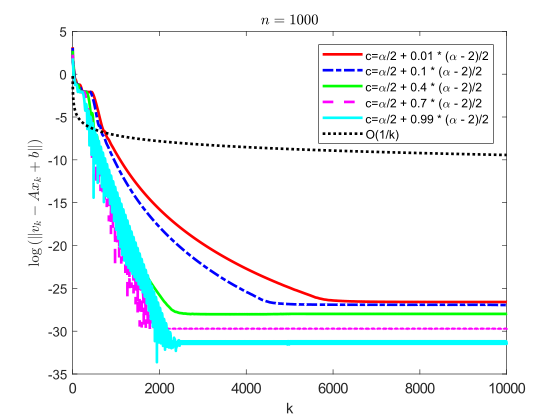}
    \end{subfigure}
	\vskip\baselineskip
	\begin{subfigure}[b]{\textwidth}
        \centering
        \includegraphics[width=0.48\linewidth]{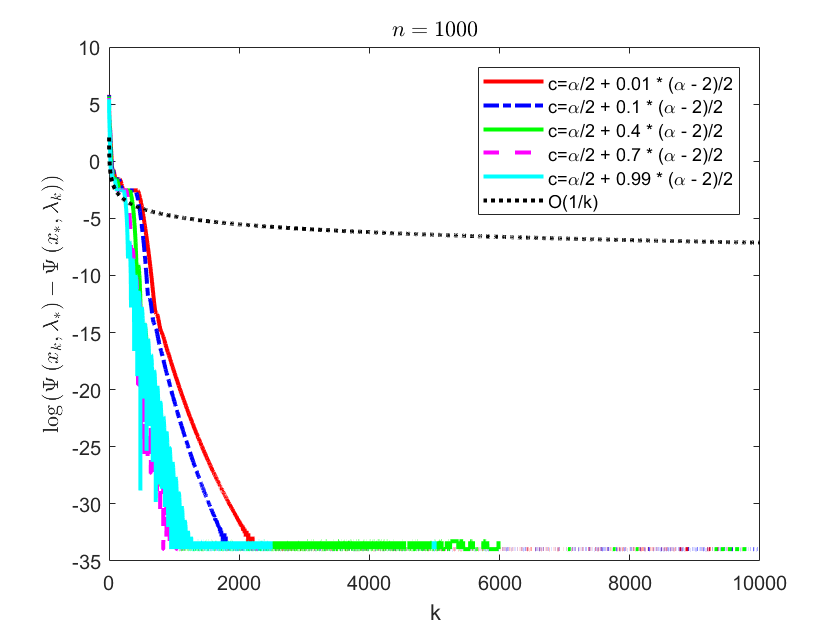}%
        \hfill
		\includegraphics[width=0.48\linewidth]{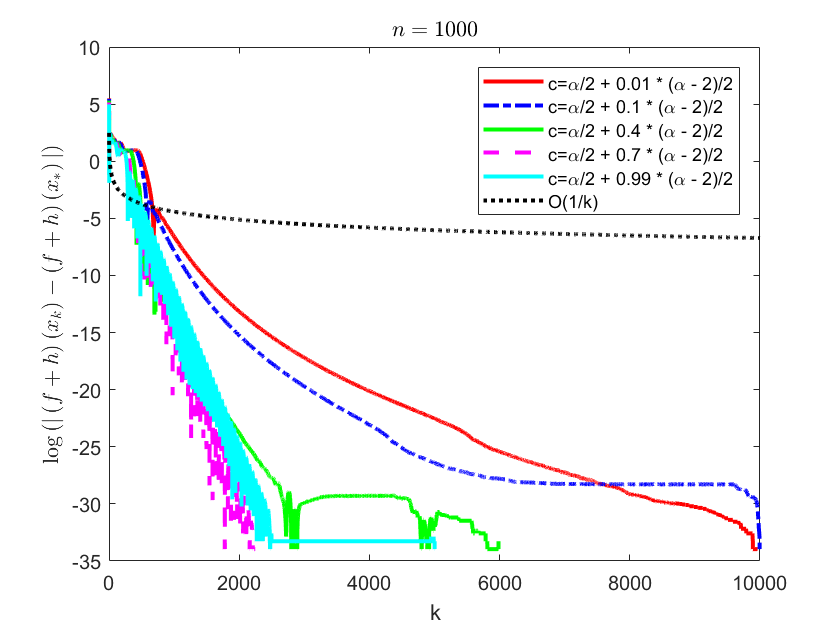}%
	\end{subfigure}
	\caption{The impact of the parameter $c$ on the convergence behavior of the discrete velocity, the tangent residual, the primal-dual gap, and the function values when $\alpha=20$.}\label{figure4}
\end{figure}

\subsection{Comparisons with other algorithms}
In this subsection, we will compare the performance of the Fast RFB algorithm with that of other algorithms from the literature when it comes to finding saddle points of
\begin{equation*}
\Lag \left( x , \lambda \right) = \left\lVert x \right\rVert_{1} + \frac{1}{2}\langle x,Hx\rangle - \langle x,h\rangle + \left\langle \lambda , Ax - b \right\rangle,
\end{equation*}
which is the Lagrangian associated with the optimization problem with linear equality constraints
\begin{align}\label{comp:line}
\min & \left\lVert x \right\rVert_{1} + \frac{1}{2}\langle x,Hx\rangle - \langle x,h\rangle.\\
\textrm{such that} \,	& Ax =b\nonumber
\end{align}
The matrices $A, H \in \R^{n \times n}$ and the vectors $h, b \in \R^n$ are chosen as in the previous subsection.   This problem amounts to solving the monotone inclusion problem  \eqref{intro:pb:eq} for $M(x,\lambda) = (\partial \|\cdot\|_1(x), 0)$ and $F(x,\lambda) = (Hx - h + A^*\lambda,  b-Ax)$.  The Lipschitz constant $L$ of the operator $F$ is  taken as
\begin{equation*}
L = \sqrt{\left(\left\lVert H \right\rVert + \left\lVert A \right\rVert\right)^2 + \left\lVert A  \right\rVert^2}.
\end{equation*}
\begin{figure}[ht!]
\centering
\begin{subfigure}[b]{\textwidth}
\centering
\includegraphics[width=0.48\linewidth]{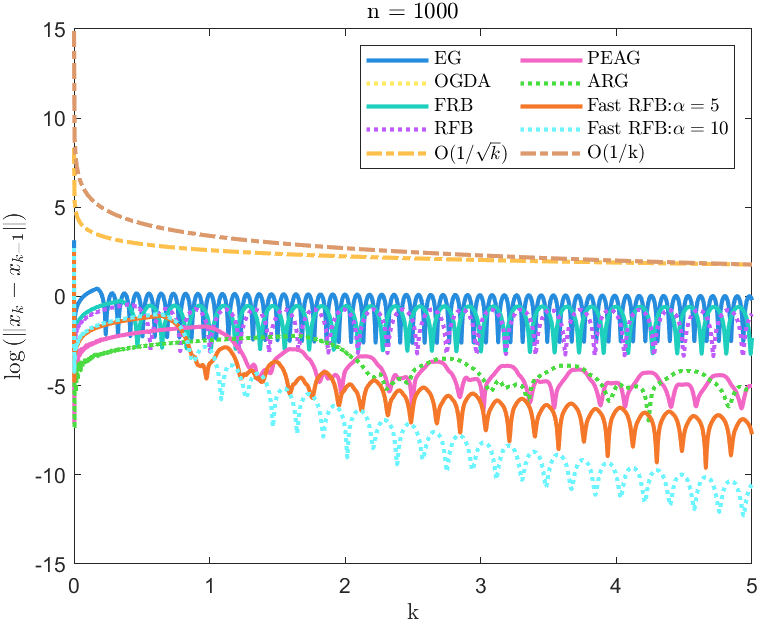}
\hfill
\includegraphics[width=0.48\linewidth]{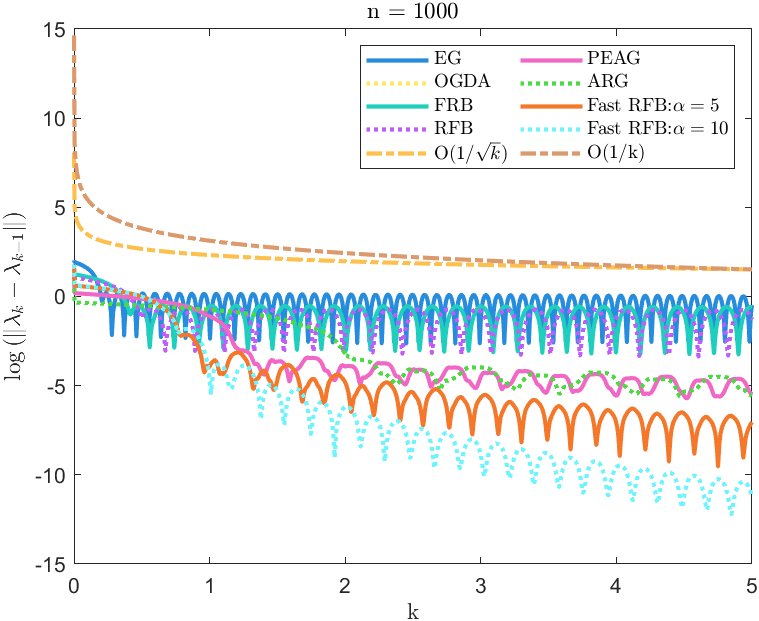}
\end{subfigure}
\vskip\baselineskip
\begin{subfigure}[b]{\textwidth}
\centering
\includegraphics[width=0.48\linewidth]{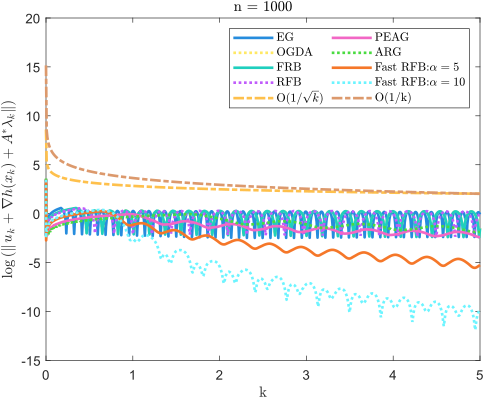}%
\hfill
\includegraphics[width=0.48\linewidth]{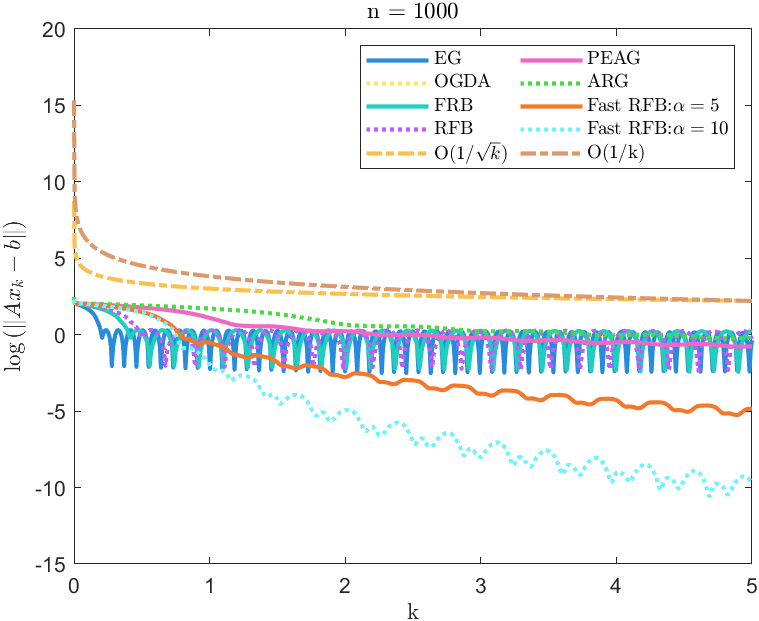}
\end{subfigure}
\vskip\baselineskip
\begin{subfigure}[b]{\textwidth}
\centering
\includegraphics[width=0.48\linewidth]{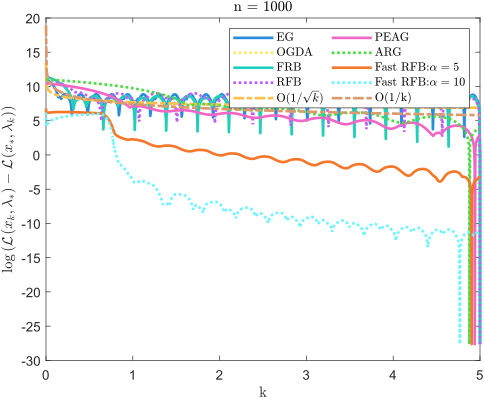}%
\hfill
\includegraphics[width=0.48\linewidth]{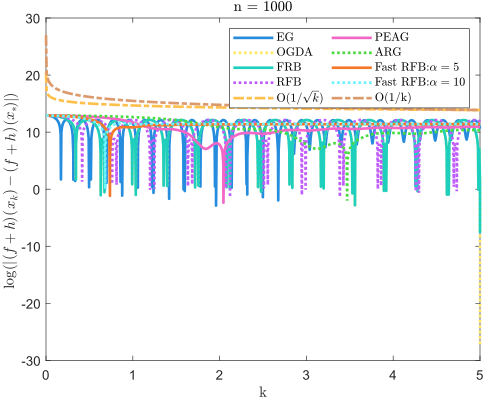}%
\end{subfigure}
\caption{A comparison of various methods in terms of discrete velocity,  tangent residual,  primal-dual gap, and function values for $n=1000$.}\label{figure5}
\end{figure}
In the following, we present the algorithms used in the numerical experiments together with their corresponding parameters:
\begin{itemize}
\item[(1)] EG: Extragradient method \eqref{algo:EG} (see \cite{Korpelevich1976, Antipin1976, Tran-Dinh2023}) with $\gamma=\frac{0.99}{L}$ and $\eta = 1$;
\item[(2)] OGDA: Optimistic Gradient Descent Ascent method \eqref{algo:OGDA} (see \cite{Popov1980, Tran-Dinh2023}) with $\gamma=\frac{0.99}{2L}$ and $\eta = 1$;
\item[(3)] FRB: Forward-Reflected-Backward method \eqref{algo:Malitsky-Tam} (see \cite{Malitsky2015,Cai2022Zheng}) with $\gamma=\frac{0.99}{2L}$;
\item[(4)] RFB: Reflected Forward-Backward method \eqref{algo:Malitsky} (see \cite{Cevher2020Vu}) with $\gamma=\frac{0.99(\sqrt{2}-1)}{L}$;
\item[(5)] PEAG: Past Extra-Anchored Gradient method (see \cite{Tran-Dinh2024,Tran-Dinh2023}) with $\gamma=\sqrt{\frac{2}{17}}\frac{0.99}{L}$;
\item[(6)] ARG: Accelerated Reflected Gradient method (see \cite{Cai2022Zheng}) with $\gamma=\frac{0.99}{\sqrt{24}L}$;
\item[(7)] Fast RFB: our Algorithm \ref{algo:line} (for ${\cal K} = \{0\}$) with $\gamma=\frac{0.99}{2L}$ and $c=\frac{\alpha+ 0.1(\alpha-2)}{2}$,  for $\alpha =5$ and $\alpha=10$.
\end{itemize}
Figure \ref{figure5} presents, on a semi-logarithmic scale, the convergence behavior of the discrete velocity, tangent residual, primal-dual gap, and function values generated by each of the above-mentioned algorithms for the case $n=1000$ after $5\times10^5$ iterations per algorithm.  We also included in our experiments Tseng's Forward-Backward-Forward (FBF) method \eqref{algo:Tseng} (see \cite{Tseng2000, Luo2022Tran-Dinh, Tran-Dinh2023}) with $\gamma=\frac{0.99}{L}$, and the Past Forward-Backward-Forward method \eqref{algo:PFBF} (see \cite{Luo2022Tran-Dinh}) with $\gamma=\frac{0.99}{2L}$.  However, since their performance was quite similar to that of EG and OGDA, respectively,  we decided not to include them in the plots.  From the plots, it can be seen that, for the considered instance, Fast RFB achieves the best convergence performance in most cases among all evaluated methods.

We also analyze the asymptotic behaviour of the residual
$$V \left( x_{k} , \lambda_{k} \right) = \begin{pmatrix}
u_k + \nabla h(x_k) + A^*\lambda_k\\
b-Ax_k
\end{pmatrix} = \begin{pmatrix}
u_k + Hx_k -h + A^*\lambda_k\\
b-Ax_k
\end{pmatrix},  \quad \mbox{where} \ u_k \in \partial \| \cdot\|_1(x_k).$$
Figure \ref{figure7} presents, on a semi-logarithmic scale, the convergence behaviour of the norm of the residual for $n = 200, 500, 800$ and $1000$ after $5\times10^5$ iterations per algorithm.  Fast RFB achieves the best convergence performance in all four instances.

\begin{figure}[ht!]
\centering
\begin{subfigure}[b]{\textwidth}
\centering
\includegraphics[width=0.48\linewidth]{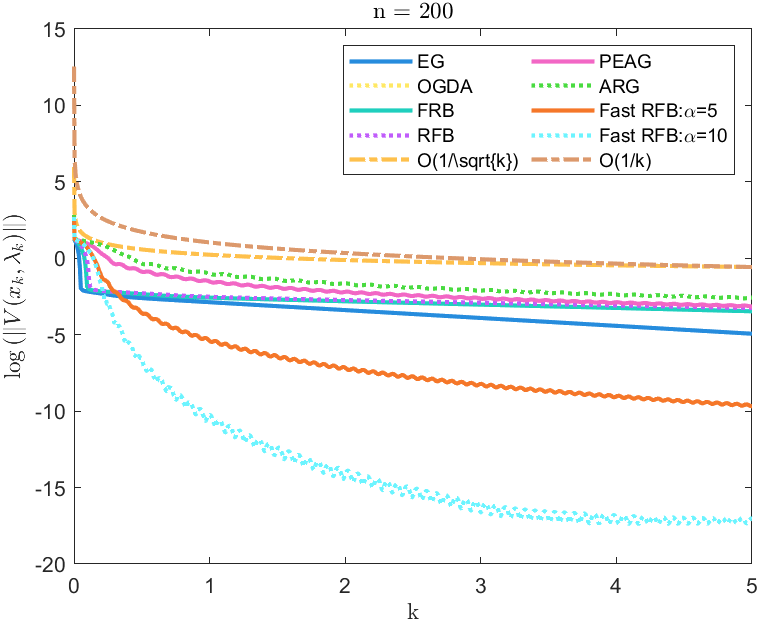}
\hfill
\includegraphics[width=0.48\linewidth]{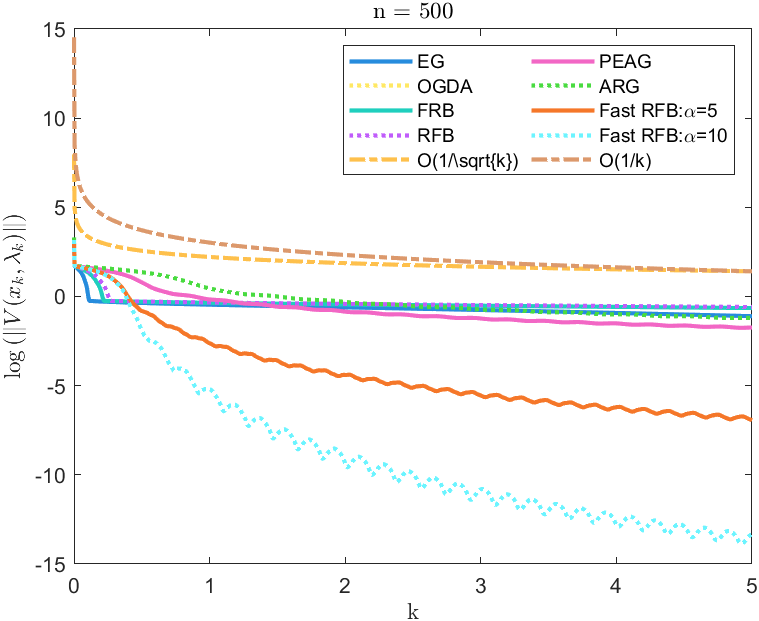}
\end{subfigure}
\vskip\baselineskip
\begin{subfigure}[b]{\textwidth}
\centering
\includegraphics[width=0.48\linewidth]{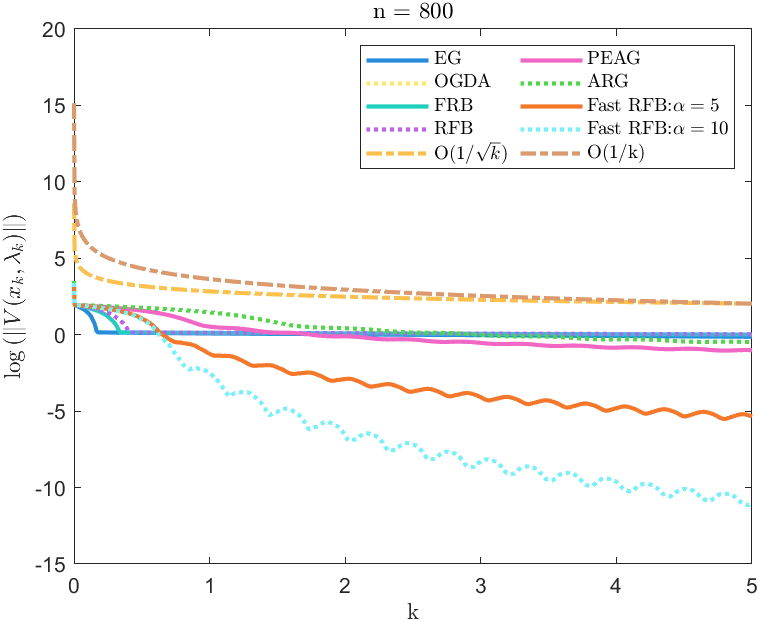}%
\hfill
\includegraphics[width=0.48\linewidth]{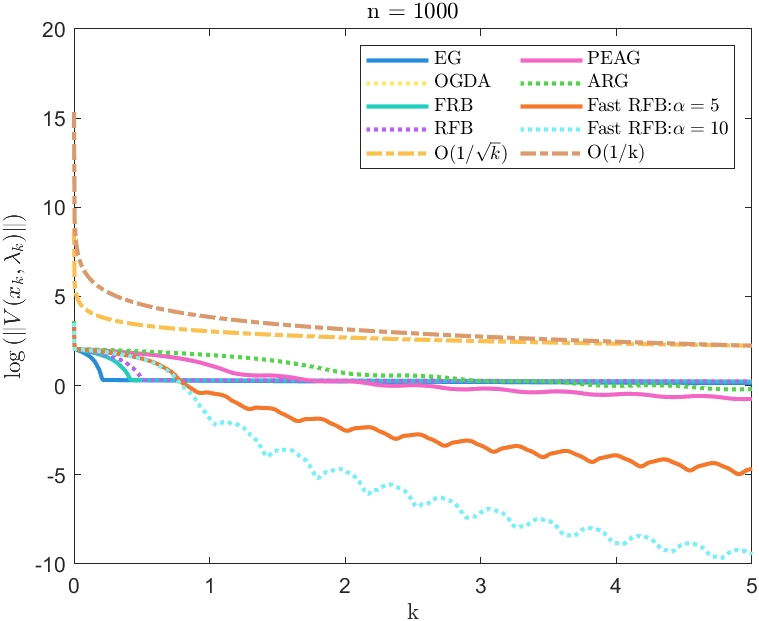}
\end{subfigure}
\caption{A comparison of various methods in terms of the norm of the residual $\left\lVert V \left( x_{k} , \lambda_{k} \right) \right\rVert$ for $n=200, 500, 800$ and $1000$.}\label{figure7}
\end{figure}

\begin{table}[ht!]
\centering
\renewcommand{\arraystretch}{1.2} 
\begin{tabular}{|c||c||c|c|c|c|c|}
\hline
Method & Success & Avg.  $\#$ iter. & Std.  dev.   $\#$ iter.  & Avg.  time& Std.  dev.  time \\
\hline\hline
EG        & 1.0 & 19501.2  & 956.464  & 15.191  & 4.188  \\ \hline
OGDA      & 1.0 & 42866.6  & 1970.930 & 36.133  & 6.013  \\ \hline
FRB       & 1.0 & 42878.3  & 1969.767 & 20.095  & 3.252  \\ \hline
RFB       & 1.0 & 52768.7  & 2408.202 & 23.718  & 4.128  \\ \hline
PEAG      & 1.0 & 215600.1 & 14.130   & 90.937  & 20.876 \\ \hline
ARG       & 1.0 & 365924.6 & 24.359   & 158.859 & 37.623 \\ \hline
Fast RFB: $\alpha=5$  & 1.0 & 32172.8  & 8.066    & 14.735  & 2.606  \\ \hline
Fast RFB: $\alpha=10$ & 1.0 & 21439.8  & 5.007    & 9.683   & 1.555  \\ \hline
\hline
\end{tabular}
\caption{Success rate of the methods in satisfying the stopping criterion 
$\norm{V \left( x_{k} , \lambda_{k} \right)} \leq 10^{-1}$, along with the corresponding runtimes and iteration counts.}
\label{table1e1}
\end{table}
\begin{table}[ht!]
\centering
\renewcommand{\arraystretch}{1.2} 
\begin{tabular}{|c||c||c|c|c|c|c|}
\hline
Method & Success & Avg.  $\#$ iter. & Std.  dev.   $\#$ iter.  & Avg.  time& Std.  dev.  time \\
\hline\hline
EG        & 1.0 & 434690.7 & 1034.704 & 261.149 & 13.192  \\ \hline
OGDA      & 0.0 & NaN      & NaN      & NaN     & NaN    \\ \hline
FRB       & 0.0 & NaN      & NaN      & NaN     & NaN    \\ \hline
RFB       & 0.0 & NaN      & NaN      & NaN     & NaN    \\ \hline
PEAG      & 0.0 & NaN      & NaN      & NaN     & NaN    \\ \hline
ARG       & 0.0 & NaN      & NaN      & NaN     & NaN    \\ \hline
Fast RFB: $\alpha=5$  & 1.0 & 76644.4  & 7.382    & 29.266  & 4.250  \\ \hline
Fast RFB: $\alpha=10$ & 1.0 & 34052.0  & 6.976    & 13.393  & 1.418  \\ \hline
\hline
\end{tabular}
\caption{Success rate of the methods in satisfying the stopping criterion $\norm{V \left( x_{k} , \lambda_{k} \right)} \leq 10^{-2}$, along with the corresponding runtimes and iteration counts.}
\label{table1e2}
\end{table}
\begin{table}[ht!]
\centering
\renewcommand{\arraystretch}{1.2} 
\begin{tabular}{|c||c||c|c|c|c|c|}
\hline
Method & Success & Avg.  $\#$ iter. & Std.  dev.   $\#$ iter.  & Avg.  time& Std.  dev.  time \\
\hline\hline
EG        & 1.0 & 881605.3 & 1719.348 & 579.387 & 10.389  \\ \hline
OGDA      & 0.0 & NaN      & NaN      & NaN     & NaN    \\ \hline
FRB       & 0.0 & NaN      & NaN      & NaN     & NaN    \\ \hline
RFB       & 0.0 & NaN      & NaN      & NaN     & NaN    \\ \hline
PEAG      & 0.0 & NaN      & NaN      & NaN     & NaN    \\ \hline
ARG       & 0.0 & NaN      & NaN      & NaN     & NaN    \\ \hline
Fast RFB: $\alpha=5$  & 1.0 & 179003.7  & 6.993    & 67.637  & 1.640  \\ \hline
Fast RFB: $\alpha=10$ & 1.0 & 51009.8  & 3.490    & 22.242  & 0.296  \\ \hline
\hline
\end{tabular}
\caption{Success rate of the methods in satisfying the stopping criterion 
$\norm{V \left( x_{k} , \lambda_{k} \right)} \leq 10^{-3}$, along with the corresponding runtimes and iteration counts.}
\label{table1e3}
\end{table}

Furthermore,  we compare the performance of all algorithms with termination criterion $\left\lVert V \left( x_{k} , \lambda_{k} \right) \right\rVert$ $\leq\varepsilon$ under varying precision thresholds $\varepsilon \in \{10^{-1},  10^{-2},  10^{-3}\}$ for $n=200$.  Each algorithm is terminated after $10^6$ iterations, even if the stopping criterion is not satisfied.  We run the experiment with 10 different initializations and record the average number of iterations, the standard deviation of the number of iterations, the average runtime,  and the standard deviation of the runtime (in seconds).  The results are summarized in Tables  \ref{table1e1}–\ref{table1e3}.

For low-accuracy settings,  such as $\varepsilon = 10^{-1}$,  EG requires the fewest iterations,  whereas our method reaches the solution in less computational time.  As the accuracy requirement increases, our method not only attains the stopping criterion faster than the others but also consistently requires less computational time.  Moreover, our method appears to be more stable, in the sense that its performance does not vary significantly across different initializations,  as indicated by the standard deviation of both the number of iterations and the runtime.

\vskip 2mm

\noindent \textbf{Data availability} Not applicable.

\section*{Declarations}

\textbf{Competing interests} The authors have no competing interests to declare.

\vskip 2mm

\noindent \textbf{Acknowledgements}  The authors would like to thank the two anonymous reviewers for their comments and suggestions, which have helped to improve the quality of the presentation.


\appendix
\section{Auxiliary results}
In the first part of the appendix, we provide some basic auxiliary results for the analysis
carried out in the paper.

The following result was introduced as Lemma A.5 in \cite{Bot2023RobertNguyen}.
\begin{lem}\label{lem:lim-u-k}
	Let $a \geq 1$ and $\left( q_{k} \right)_{k \geq 0}$ be a bounded sequence in $\cH$ such that
	\begin{equation*}
	\lim_{k \to + \infty} \left( q_{k+1} + \frac{k}{a} \left( q_{k+1} - q_{k} \right) \right) = p \in \cH .
	\end{equation*}
	Then $\lim_{k \to + \infty} q_{k} = p$.
\end{lem}

The following result concerns quasi-F\'ejer monotone sequences and is a particular instance of Lemma 5.31 in \cite{BauschkeCombettes2}.
\begin{lem}\label{lem:quasi-Fej}
	Let $ (a_{k})_{k \geq 0}$, $ (b_{k})_{k \geq 0}$, and $ (d_{k})_{k \geq 0} $ be sequences of real numbers. Assume that $ (a_{k})_{k \geq 0} $ is bounded from below, and $ (b_{k})_{k \geq 0} $ and $ (d_{k})_{k \geq 0} $ are nonnegative sequences such that $ \sum_{k \geq 0} d_{k} < + \infty$. If
	\begin{equation*}
	a_{k+1} \leq \left( 1 + d_{k} \right) a_{k} - b_{k}
	\quad \forall k \geq 0,
	\end{equation*}
	then the following statements are true:
\begin{itemize}
\item[(i)]  the sequence $ (b_{k})_{k \geq 0} $ is summable, i.e., $ \sum_{k \geq 0} b_{k} < +\infty$;

\item[(ii)] the sequence $ (a_{k})_{k \geq 0} $ is convergent.
\end{itemize}
\end{lem}

To show the convergence of the sequence of generated iterates we will use the following so-called Opial Lemma  \cite{BauschkeCombettes2}.
\begin{lem}\label{lem:opial}
Let $ S \subseteq { \cH} $ be a nonempty set and $(z_{k})_{k \geq 0} { \subseteq \cH}$ a sequence that satisfies the following assumptions:
\begin{itemize}
\item[(i)] for every $ z_{\ast} \in S $, $ \lim_{k \to + \infty} \norm{z_{k} - z_{\ast}} $ exists;

\item[(ii)]  every weak sequential cluster point of $ (z_{k})_{k \geq 0} $ belongs to $S$.
\end{itemize}
Then $ (z_{k})_{k \geq 0} $ converges weakly to an element in $S$.
\end{lem}

The convergence analysis also uses the following result.
\begin{lem}
\label{lem:quad}
Let $a, b, c \in \R$ such that $a \neq 0$ and $b^{2} - ac \leq 0$.
The following statements are true:
\begin{itemize}
\item[(i)]	if $a > 0$, then
\begin{equation*}
a \left\lVert x \right\rVert ^{2} + 2b \left\langle x, y \right\rangle + c \left\lVert y \right\rVert ^{2} \geq 0 \quad \forall x, y \in { \cH} ;
\end{equation*}

\item[(ii)]	if $a < 0$, then
\begin{equation*}
a \left\lVert x \right\rVert ^{2} + 2b \left\langle x, y \right\rangle + c \left\lVert y \right\rVert ^{2} \leq 0 \quad \forall x, y \in { \cH} .
\end{equation*}
\end{itemize}
\end{lem}

\section{Missing proofs}

\begin{proof}[Proof of Lemma \ref{lem:dE}]
Recall that by \eqref{split:defi:u-k-1-lambda}, we have for all $k \geq 1$
\begin{equation*}
u_{\lambda,s,k}
= 2 \lambda \left( z_{k} - z_{\ast} \right) + 2k \left( z_{k} - z_{k-1} \right) + s \gamma k v_{k}.
\end{equation*}
Similarly,
\begin{equation}\label{split:defi:u-k-lambda}
u_{\lambda,s,k+1}
= 2 \lambda \left( z_{k+1} - z_{\ast} \right) + 2 \left( k + 1 \right) \left( z_{k+1} - z_{k} \right) + s \gamma \left( k + 1 \right) v_{k+1} .
\end{equation}
After subtraction and by using \eqref{split:d-u}, we find that
\begin{equation}\label{split:defi:u-k-lambda:dif}
\begin{split}
& u_{\lambda,s,k+1} - u_{\lambda,s,k} \\
= & \ 2 \left( \lambda + 1 - \alpha \right) \left( z_{k+1} - z_{k} \right) + 2 \left( k + \alpha \right) \left( z_{k+1} - z_{k} \right) - 2k \left( z_{k} - z_{k-1} \right) \\
& + s \gamma v_{k+1} + s \gamma k \left( v_{k+1} - v_{k} \right) \\
= & \ 2 \left( \lambda + 1 - \alpha \right) \left( z_{k+1} - z_{k} \right) -  \left( 2c  - s \right) \gamma v_{k+1} - \gamma \Big((2-s)k + 2\left( \alpha - c \right)\Big)\left( v_{k+1} - v_{k} \right) .
\end{split}
\end{equation}
For all $k \geq 1$ we have
\begin{equation}\label{split:dif:u-lambda:pre}
\dfrac{1}{2} \left( \left\lVert u_{\lambda,s,k+1} \right\rVert ^{2} - \left\lVert u_{\lambda,s,k} \right\rVert ^{2} \right)
= \left\langle u_{\lambda,s,k+1}, u_{\lambda,s,k+1} - u_{\lambda,s,k} \right\rangle - \dfrac{1}{2} \left\lVert u_{\lambda,s,k+1} - u_{\lambda,s,k} \right\rVert ^{2},
\end{equation}
and for all $ k \geq 0 $
\begin{equation}\label{split:dif:norm}
\begin{split}
& 2 \lambda \left( \alpha - 1 - \lambda \right) \left( \left\lVert z_{k+1} - z_{\ast} \right\rVert ^{2} - \left\lVert z_{k} - z_{\ast} \right\rVert ^{2} \right) \\
= & \ 4 \lambda \left( \alpha - 1 - \lambda \right) \left\langle z_{k+1} - z_{\ast}, z_{k+1} - z_{k} \right\rangle - 2 \lambda \left( \alpha - 1 - \lambda \right) \left\lVert z_{k+1} - z_{k} \right\rVert ^{2}.
\end{split}
\end{equation}
We use \eqref{split:defi:u-k-lambda} and \eqref{split:defi:u-k-lambda:dif} to derive for all $k \geq 1$
\begin{equation}
\begin{aligned}\label{split:dif:u-lambda:inn}
& \left\langle u_{\lambda,s,k+1}, u_{\lambda,s,k+1} - u_{\lambda,s,k} \right\rangle\\
= & \ 4 \lambda \left( \lambda + 1 - \alpha \right) \left\langle z_{k+1} - z_{\ast}, z_{k+1} - z_{k} \right\rangle - 2 \lambda \gamma \left( 2c - s \right)  \left\langle z_{k+1} - z_{\ast}, v_{k+1} \right\rangle \\
& - 2\lambda \gamma \Big((2-s)k + 2\left( \alpha - c \right)\Big) \left\langle z_{k+1} - z_{\ast}, v_{k+1} - v_{k} \right\rangle
+ 4 \left( \lambda + 1 - \alpha \right) \left( k + 1 \right) \left\lVert z_{k+1} - z_{k} \right\rVert ^{2} \\
& + 2 \Bigl( s \left( \lambda + 1 - \alpha \right) +  s - 2c \Bigr) \gamma \left( k + 1 \right) \left\langle z_{k+1} - z_{k}, v_{k+1} \right\rangle \\
& - 2\gamma \Big((2-s)k + 2\left( \alpha - c \right)\Big) \left( k + 1 \right) \left\langle z_{k+1} - z_{k}, v_{k+1} - v_{k} \right\rangle \\
& -  s\left( 2c - s \right) \gamma^{2} \left( k + 1 \right) \left\lVert v_{k+1} \right\rVert ^{2} - s \gamma^{2} \Big((2-s)k + 2\left( \alpha - c \right)\Big) \left( k + 1 \right) \left\langle v_{k+1}, v_{k+1} - v_{k} \right\rangle,
\end{aligned}
\end{equation}
and
\begin{equation}\label{split:dif:u-lambda:norm}
\begin{split}
& \minus \dfrac{1}{2} \left\lVert u_{\lambda,s,k+1} - u_{\lambda,s,k} \right\rVert ^{2} \\
= & \minus 2 \left( \lambda + 1 - \alpha \right) ^{2} \left\lVert z_{k+1} - z_{k} \right\rVert ^{2} + 2 \left( 2c -s \right) \left( \lambda + 1 - \alpha \right) \gamma \left\langle z_{k+1} - z_{k}, v_{k+1} \right\rangle \\
& - \dfrac{1}{2} \left( 2c - s \right)^2 \gamma^{2} \left\lVert v_{k+1} \right\rVert ^{2}
- \dfrac{1}{2} \gamma^{2} \Big((2-s)k + 2\left( \alpha - c \right)\Big)^2 \left\lVert v_{k+1} - v_{k} \right\rVert ^{2} \\
& + 2\left( \lambda + 1 - \alpha \right) \gamma \Big((2-s)k + 2\left( \alpha - c \right)\Big)\left\langle z_{k+1} - z_{k}, v_{k+1} - v_{k} \right\rangle \\
& - \gamma^{2} \Big((2-s)k + 2\left( \alpha - c \right)\Big)( 2c - s ) \left\langle v_{k+1}, v_{k+1} - v_{k} \right\rangle .
\end{split}
\end{equation}
By plugging \eqref{split:dif:u-lambda:inn} and \eqref{split:dif:u-lambda:norm} into \eqref{split:dif:u-lambda:pre}, we get for all $k \geq 1$
\begin{equation}\label{split:dif:u-lambda}
\begin{split}
& \dfrac{1}{2} \left( \left\lVert u_{\lambda,s,k+1} \right\rVert ^{2} - \left\lVert u_{\lambda,s,k} \right\rVert ^{2} \right)\\
=  & \ 4 \lambda \left( \lambda + 1 - \alpha \right) \left\langle z_{k+1} - z_{\ast}, z_{k+1} - z_{k} \right\rangle \\
& - 2\left( 2c - s \right) \lambda \gamma \left\langle z_{k+1} - z_{\ast}, v_{k+1} \right\rangle  - 2\Big((2-s)k + 2\left( \alpha - c \right)\Big) \lambda \gamma  \left\langle z_{k+1} - z_{\ast}, v_{k+1} - v_{k} \right\rangle\\
& + 2 \left( \lambda + 1 - \alpha \right) \left( 2k + \alpha + 1 - \lambda \right) \left\lVert z_{k+1} - z_{k} \right\rVert ^{2}- \dfrac{1}{2}\gamma^2 \Big((2-s)k + 2\left( \alpha - c \right)\Big)^2\left\lVert v_{k+1} - v_{k} \right\rVert ^{2} \\
& + 2 \gamma \left(\Bigl(  s  \left( \lambda + 1 - \alpha \right) +s - 2c \Bigr) (k+1) + \left( \lambda + 1 - \alpha \right)(2 c - s)\right)\left\langle z_{k+1} - z_{k}, v_{k+1} \right\rangle \\
& - 2\gamma\Big((2-s)k + 2\left( \alpha - c \right)\Big) \left( k + \alpha - \lambda \right) \left\langle z_{k+1} - z_{k}, v_{k+1} - v_{k} \right\rangle
\\
& - \dfrac{1}{2 }\gamma^2  \left( 2c - s \right)  \left( 2sk + s + 2c\right) \left\lVert v_{k+1} \right\rVert ^{2} \\
& - \gamma^2 \Big((2-s)k + 2\left( \alpha - c \right)\Big) \bigl( s k + 2 c \bigr) \left\langle v_{k+1}, v_{k+1} - v_{k} \right\rangle.
\end{split}
\end{equation}
Furthermore, for all $k \geq 1$ we have
\begin{equation}\label{split:dif:vi}
\begin{split}
& 2\lambda \gamma\Big((2-s)(k+1) + 2\left( \alpha - c \right)\Big) \left\langle z_{k+1} - z_{\ast}, v_{k+1} \right\rangle - 2\lambda \gamma\Big((2-s)k + 2\left( \alpha - c \right)\Big) \left\langle z_{k} - z_{\ast}, v_{k} \right\rangle \\
= & \ 2\lambda \gamma\Big((2-s)k + 2\left( \alpha - c \right)\Big) \big( \left\langle z_{k+1} - z_{\ast}, v_{k+1} \right\rangle - \left\langle z_{k} - z_{\ast}, v_{k} \right\rangle \big) \\
& +2\lambda \gamma(2-s)\left\langle z_{k+1} - z_{\ast}, v_{k+1} \right\rangle\\
= & \ 2\lambda \gamma(2-s)\left\langle z_{k+1} - z_{\ast}, v_{k+1} \right\rangle
+ 2\lambda \gamma\Big((2-s)k + 2\left( \alpha - c \right)\Big) \left\langle z_{k+1} - z_{\ast}, v_{k+1} - v_{k} \right\rangle \\
& - 2\lambda \gamma\Big((2-s)k + 2\left( \alpha - c \right)\Big) \left\langle z_{k+1} - z_{k}, v_{k+1} - v_{k} \right\rangle \\
& + 2\lambda \gamma\Big((2-s)k + 2\left( \alpha - c \right)\Big) \left\langle z_{k+1} - z_{k}, v_{k+1} \right\rangle,
\end{split}
\end{equation}
and
\begin{equation}\label{split:dif:eq}
\begin{split}
& \frac{1}{2}\gamma^2\Big((2-s)(k+1) + 2\left( \alpha - c \right)\Big) \Bigl( s \left( k+1 \right) + 2 c \Bigr) \left\lVert v_{k+1} \right\rVert ^{2}\\
& - \frac{1}{2}\gamma^2\Big((2-s)k + 2\left( \alpha - c \right)\Big) \Bigl( s  k  + 2 c \Bigr) \left\lVert v_{k} \right\rVert ^{2} \\
= & \ \gamma^2\Big((2-s)k + 2\left( \alpha - c \right)\Big) ( s  k  + 2 c )\Bigl( \left\langle v_{k+1}, v_{k+1} - v_{k} \right\rangle - \frac{1}{2} \left\lVert v_{k+1} - v_{k} \right\rVert^{2}\Bigr)\\
& +\frac{1}{2}\gamma^2\big((2-s)(2ks + 2c +s) + 2(\alpha - c )s \big)\left\lVert v_{k+1}\right\rVert^{2}.
\end{split}
\end{equation}
Summing \eqref{split:dif:norm}, \eqref{split:dif:u-lambda}, \eqref{split:dif:vi} and \eqref{split:dif:eq}, yields \eqref{split:dE} for every $k \geq 1$.
\end{proof}

\begin{proof}[Proof of Lemma \ref{lem:reg}]
(i) By the definition of $\cG_{\lambda,s,k} $ in \eqref{split:defi:F}, we have for every $k \geq 2$
\begin{equation}\label{ene:dF}
\begin{split}
& \ \cG_{\lambda,s,k+1}  - \cG_{\lambda,s,k} \\
= & \ \cE_{\lambda, s, k+1} - \cE_{\lambda, s, k} - 2\gamma \left[ \big( \left( 2 - s \right)\left( k + 1 \right) + 2 \left( \alpha - c \right)\big) \left( k + 1 \right) \left\langle z_{k+1} - z_{k}, F \left( z_{k+1} \right) - F \left( w_{k} \right) \right\rangle \right.\\
&\left. \qquad \qquad \qquad \qquad \qquad - \big( \left( 2 - s \right)k + 2 \left( \alpha - c \right)\big) k \left\langle z_{k} - z_{k-1}, F \left( z_{k} \right) - F \left( w_{k-1} \right) \right\rangle \right] \\
&\quad + \gamma^4 L^2 c \left[ \big( \left( 2 - s \right)\left( k + 1 \right) + 2 \left( \alpha - c \right)\big) \sqrt{ \left( 2 - s \right)\left( k + 1 \right) + 2 \left( \alpha - c \right)} \left\lVert v_{k+1} - v_{k} \right\rVert ^{2}\right.\\
&\left. \qquad \qquad \qquad
- \big( \left( 2 - s \right)k + 2 \left( \alpha - c \right)\big) \sqrt{ \left( 2 - s \right)k + 2 \left( \alpha - c \right)} \left\lVert v_{k} - v_{k-1} \right\rVert ^{2} \right] \\
&\quad + \gamma^3L \left[\big( \left( 2 - s \right)\left( k + 1 \right) + 2 \left( \alpha - c \right)\big)(k + \alpha + 1 -c) \left\lVert v_{k+1} - v_{k} \right\rVert ^{2}\right.\\
&\left. \qquad \qquad \quad - \big( \left( 2 - s \right)k + 2 \left( \alpha - c \right)\big)(k + \alpha -c)  \left\lVert v_{k} - v_{k-1} \right\rVert ^{2} \right] .
\end{split}
\end{equation}
From \eqref{split:dE}, using the notations in \eqref{ene:const}, we have that for all $k \geq 2$
\begin{equation}\label{ene:dE}
\begin{split}
& \ \cE_{\lambda,s,k+1} - \cE_{\lambda,s,k} \\
= & \ \minus 4 \left( c - 1 \right) \lambda \gamma \left\langle z_{k+1} - z_{\ast}, v_{k+1} \right\rangle
+ 2\Bigl( \omega_{3} k +  \left( \lambda + 1 - \alpha \right) \left( \alpha + 1 \right) \Bigr) \left\lVert z_{k+1} - z_{k} \right\rVert ^{2} \\
& \ + 2 \gamma \left( \omega_{1} k + \omega_{2} \right) \left\langle z_{k+1} - z_{k}, v_{k+1} \right\rangle
- \gamma^{2} \left( k + \alpha \right)\big(\omega_{0}k + 2 \omega_{6}\big)\left\lVert v_{k+1} - v_{k} \right\rVert ^{2}\\
& \ - 2\gamma \left( k + \alpha \right) \big(\omega_{0}k + 2 \omega_{6} \big)\left\langle z_{k+1} - z_{k}, v_{k+1} - v_{k} \right\rangle \\
& \ + \gamma^{2} \big(\omega_4 k + \left( 1 - c \right)\left( 2c +s \right)+ s \omega_{6} \big) \left\lVert v_{k+1} \right\rVert ^{2}\\
\leq & \ \minus 4 \left( c - 1 \right) \lambda \gamma \left\langle z_{k+1} - z_{\ast}, v_{k+1} \right\rangle
+ 2 \omega_{3} k \left\lVert z_{k+1} - z_{k} \right\rVert ^{2}+ \gamma^{2} \big(\omega_4 k + s \omega_{6} \big) \left\lVert v_{k+1} \right\rVert ^{2} \\
& \ + 2 \gamma \left( \omega_{1} k + \omega_{2} \right) \left\langle z_{k+1} - z_{k}, v_{k+1} \right\rangle
- \gamma^{2} \left( k + \alpha \right)\big(\omega_{0}k + 2 \omega_{6} \big)\left\lVert v_{k+1} - v_{k} \right\rVert ^{2}\\
& \ - 2\gamma \left( k + \alpha \right) \big(\omega_{0}k + 2 \omega_{6} \big)\left\langle z_{k+1} - z_{k}, v_{k+1} - v_{k} \right\rangle,
\end{split}
\end{equation}
where the inequality follows from $0 \leq \lambda \leq \alpha - 1$ and $1 < c$.
Plugging \eqref{ene:dE} into \eqref{ene:dF} and using again \eqref{ene:const}, yields for all $k \geq 2$
\begin{equation}\label{ene:pre}
\begin{aligned}
& \ \cG_{\lambda,s,k+1}  - \cG_{\lambda,s,k} \\
\leq & \ \minus 4 \left( c - 1 \right) \lambda \gamma \left\langle z_{k+1} - z_{\ast}, v_{k+1} \right\rangle  - 2\gamma \left( k + \alpha \right) \big(\omega_{0}k + 2 \omega_{6} \big)\left\langle z_{k+1} - z_{k}, v_{k+1} - v_{k} \right\rangle\\
& \ - 2\gamma \left[ \big(\omega_{0}\left( k + 1 \right) + 2 \omega_{6}\big) \left( k + 1 \right) \left\langle z_{k+1} - z_{k}, F \left( z_{k+1} \right) - F \left( w_{k} \right) \right\rangle \right.\\
&\left. \qquad \quad - \big(\omega_{0} k + 2 \omega_{6}\big) k \left\langle z_{k} - z_{k-1}, F \left( z_{k} \right) - F \left( w_{k-1} \right) \right\rangle \right] \\
& \ + \gamma^4 L^2 c \left[ \big(\omega_{0} \left( k + 1 \right) + 2 \omega_{6}\big) \sqrt{\omega_{0} \left( k + 1 \right) + 2 \omega_{6}} \left\lVert v_{k+1} - v_{k} \right\rVert ^{2}\right.\\
&\left. \qquad \qquad \quad - \big(\omega_{0} k + 2 \omega_{6}\big) \sqrt{\omega_{0}k + 2 \omega_{6}} \left\lVert v_{k} - v_{k-1} \right\rVert ^{2} \right] \\
& \ + \gamma^3L \left[\big(\omega_{0} \left( k + 1 \right) + 2 \omega_{6} \big)(k + \alpha + 1 -c) \left\lVert v_{k+1} - v_{k} \right\rVert ^{2}\right.\\
&\left. \qquad \qquad - \big(\omega_{0} k + 2 \omega_{6}\big)(k + \alpha -c)  \left\lVert v_{k} - v_{k-1} \right\rVert ^{2} \right]\\
& \ + \left( 2 \left( \omega_{1} k + \omega_{2} \right) \gamma \left\langle z_{k+1} - z_{k}, v_{k+1} \right\rangle
+ 2 \omega_{3} k \left\lVert z_{k+1} - z_{k} \right\rVert ^{2}+ \gamma^{2} \big(\omega_4 k + s \omega_{6} \big) \left\lVert v_{k+1} \right\rVert ^{2}\right) \\
& \ - \gamma^{2} \left( k + \alpha \right)\big(\omega_{0} k + 2 \omega_{6} \big)\left\lVert v_{k+1} - v_{k} \right\rVert ^{2}.
\end{aligned}
\end{equation}
Our next aim is to derive upper estimates for the first two terms on the right-hand side of \eqref{ene:pre}, which will eventually simplify the following three terms.

On the one hand, from the Cauchy-Schwarz inequality and \eqref{split:Lip}, we have for all $k \geq 1$
\begin{equation}\label{ene:inn}
\begin{split}
& \ \minus 4 \left( c - 1 \right) \lambda \gamma \left\langle z_{k+1} - z_{\ast}, v_{k+1} \right\rangle = \minus 4 \left( c - 1 \right) \lambda \gamma \left\langle z_{k+1} - z_{\ast}, \xi_{k+1} + F \left( w_{k} \right) \right\rangle \\
= & \ \minus 4 \left( c - 1 \right) \lambda \gamma \left\langle z_{k+1} - z_{\ast}, \xi_{k+1} + F \left( z_{k+1} \right) \right\rangle + 4 \left( c - 1 \right) \lambda \gamma \left\langle z_{k+1} - z_{\ast}, F \left( z_{k+1} \right) - F \left( w_{k} \right) \right\rangle \\
\leq & \ \minus 4 \left( c - 1 \right) \lambda \gamma \left\langle z_{k+1} - z_{\ast}, \xi_{k+1} + F \left( z_{k+1} \right) \right\rangle + 4 \left( c - 1 \right) \lambda \gamma \left\lVert z_{k+1} - z_{\ast} \right\rVert \left\lVert F \left( z_{k+1} \right) - F \left( w_{k} \right) \right\rVert \\
\leq & \ \minus 4 \left( c - 1 \right) \lambda \gamma \left\langle z_{k+1} - z_{\ast}, \xi_{k+1} + F \left( z_{k+1} \right) \right\rangle + 4 \left( c - 1 \right) \lambda \gamma \left\lVert z_{k+1} - z_{\ast} \right\rVert \left\lVert v_{k+1} - v_{k} \right\rVert \\
\leq & \ \minus 4 \left( c - 1 \right) \lambda \gamma \left\langle z_{k+1} - z_{\ast}, \xi_{k+1} + F \left( z_{k+1} \right) \right\rangle\\
& \ + \frac{4 \left( c - 1 \right)^2}{ \left( k+1 \right) \sqrt{k + 1}} \lambda ^{2} \left\lVert z_{k+1} - z_{\ast} \right\rVert ^{2} +  \gamma^{2} \left( k+1 \right) \sqrt{k + 1} \left\lVert v_{k+1} - v_{k} \right\rVert ^{2} .
\end{split}
\end{equation}
On the other hand, the monotonicity of $M + F$, that $\xi_{k} \in M \left( z_{k} \right)$ and $\xi_{k+1} \in M \left( z_{k+1} \right)$, and \eqref{split:d-u}, yield for every $k \geq 2$
\begin{equation}\label{ene:mono}
\begin{split}
& \ \minus 2\gamma  \left( k + \alpha \right) \big(\omega_{0}k + 2\omega_{6}\big)\left\langle z_{k+1} - z_{k}, v_{k+1} - v_{k} \right\rangle \\
\leq & \ 2\gamma  \left( k + \alpha \right) \big(\omega_{0}k + 2 \omega_{6}\big) \left\langle z_{k+1} - z_{k}, \left(\xi_{k+1} + F \left( z_{k+1} \right) - v_{k+1} \right) - \left( \xi_{k} + F \left( z_{k} \right) - v_{k} \right) \right\rangle \\
= & \ 2\gamma  \left( k + \alpha \right) \big(\omega_{0}k + 2 \omega_{6}\big) \left\langle z_{k+1} - z_{k}, \left( F \left( z_{k+1} \right) - F \left( w_{k} \right) \right) - \left( F \left( z_{k} \right) - F \left( w_{k-1} \right) \right) \right\rangle \\
= & \ 2\gamma  \left( k + \alpha \right) \big(\omega_{0}k + 2 \omega_{6}\big) \left\langle z_{k+1} - z_{k}, F \left( z_{k+1} \right) - F \left( w_{k} \right) \right\rangle \\
& - 2\gamma  \left( k + \alpha \right) \big(\omega_{0}k + 2 \omega_{6}\big)  \left\langle z_{k+1} - z_{k}, F \left( z_{k} \right) - F \left( w_{k-1} \right) \right\rangle \\
= & \ 2\gamma  \left( k + \alpha \right) \big(\omega_{0}k + 2 \omega_{6}\big) \left\langle z_{k+1} - z_{k}, F \left( z_{k+1} \right) - F \left( w_{k} \right) \right\rangle\\
& \ - 2\gamma  \big(\omega_{0}k + 2 \omega_{6}\big) k \left\langle z_{k} - z_{k-1}, F \left( z_{k} \right) - F \left( w_{k-1} \right) \right\rangle \\
& \ +2\gamma^2  c \big(\omega_{0}k + 2 \omega_{6}\big) \left\langle v_{k+1}, F \left( z_{k} \right) - F \left( w_{k-1} \right) \right\rangle \\
& \ + 2\gamma^2  \left( k + \alpha - c \right) \big(\omega_{0}k + 2 \omega_{6} \big) \left\langle v_{k+1} - v_{k}, F \left( z_{k} \right) - F \left( w_{k-1} \right) \right\rangle \\
= & \ 2\gamma \left[ \big(\omega_{0}\left( k + 1 \right) + 2 \omega_{6}\big) \left( k + 1 \right) \left\langle z_{k+1} - z_{k}, F \left( z_{k+1} \right) - F \left( w_{k} \right) \right\rangle \right.\\
& \left. \qquad -\big(\omega_{0}k + 2 \omega_{6}\big) k \left\langle z_{k} - z_{k-1}, F \left( z_{k} \right) - F \left( w_{k-1} \right) \right\rangle  \right]\\
& \ + 2\gamma  \left(\omega_{5} \left(k + 1\right) + \omega_{7} \right)  \left\langle z_{k+1} - z_{k}, F \left( z_{k+1} \right) - F \left( w_{k} \right) \right\rangle \\
& \ +2\gamma^2  c \big(\omega_{0}k + 2 \omega_{6}\big) \left\langle v_{k+1}, F \left( z_{k} \right) - F \left( w_{k-1} \right) \right\rangle \\
& \ + 2\gamma^2  \left( k + \alpha - c \right) \big(\omega_{0}k + 2 \omega_{6}\big) \left\langle v_{k+1} - v_{k}, F \left( z_{k} \right) - F \left( w_{k-1} \right) \right\rangle,
\end{split}
\end{equation}
where we use that $\omega_{7}=(2 \omega_{6}-\omega_{0})(\alpha - 1)$.

By Young's inequality and \eqref{split:Lip}, using that $\omega_{5}, \omega_{7}>0$, we obtain for all $k \geq 2$
\begin{equation}\label{ene:Young:1}
\begin{split}
& \ 2\gamma  \Big(\omega_{5} \left( k + 1 \right) + \omega_{7}\Big) \left\langle z_{k+1} - z_{k}, F \left( z_{k+1} \right) - F \left( w_{k} \right) \right\rangle \\
\leq & \ \sqrt{\omega_{5}\left( k + 1 \right) + \omega_{7}} \Bigl( \left\lVert z_{k+1} - z_{k} \right\rVert ^{2}+ \gamma^{2}\left(\omega_{5}\left( k + 1 \right) + \omega_{7}\right) \left\lVert F \left( z_{k+1} \right) - F \left( w_{k} \right) \right\rVert ^{2} \Bigr) \\
\leq & \ \sqrt{\omega_{5} \left( k + 1 \right) + \omega_{7}} \Bigl( \left\lVert z_{k+1} - z_{k} \right\rVert ^{2} + \gamma^{2}\left(\omega_{5} \left( k + 1 \right) + \omega_{7} \right) \left\lVert v_{k+1}  - v_{k}  \right\rVert ^{2} \Bigr).
\end{split}
\end{equation}
In addition, by using \eqref{split:Lip}, for all $k \geq 2$ we derive
\begin{equation*}
\begin{split}
& \ 2 \gamma^2 c\big(\omega_{0} k+ 2 \omega_{6}\big) \left\langle v_{k+1}, F \left( z_{k} \right) - F \left( w_{k-1} \right) \right\rangle \\
\leq & \ \gamma^2 c \left(  \big(\omega_{0}k + 2 \omega_{6}\big)\sqrt{\omega_{0}k + 2 \omega_{6}} \left\lVert F \left( z_{k} \right) - F \left( w_{k-1} \right) \right\rVert ^{2} +\sqrt{\omega_{0}k + 2 \omega_{6}} \left\lVert v_{k+1} \right\rVert ^{2}\right) \\
\leq & \ \gamma^2 c \left(  \gamma^2 L^2\big(\omega_{0}k + 2 \omega_{6}\big)\sqrt{\omega_{0}k + 2 \omega_{6}} \left\lVert v_{k}  - v_{k-1} \right\rVert ^{2} +\sqrt{\omega_{0}k + 2 \omega_{6}} \left\lVert v_{k+1} \right\rVert ^{2}\right) \\
= & \  - \gamma^4 L^2 c \left[ \big(\omega_{0}\left( k + 1 \right) + 2 \omega_{6}\big) \sqrt{\omega_{0}\left( k + 1 \right) + 2 \omega_{6}}\left\lVert v_{k+1} - v_{k} \right\rVert ^{2}\right.\\
&\left. \qquad \qquad \quad
- \big(\omega_{0}k + 2 \omega_{6}\big) \sqrt{\omega_{0}k + 2 \omega_{6}} \left\lVert v_{k} - v_{k-1} \right\rVert ^{2} \right]\\
& \ +\gamma^4 L^2 c \big(\omega_{0}\left( k + 1 \right) + 2 \omega_{6}\big) \sqrt{\omega_{0}\left( k + 1 \right) + 2 \omega_{6}} \left\lVert v_{k+1} - v_{k} \right\rVert ^{2} +\gamma^2 c\sqrt{\omega_{0}k + 2 \omega_{6}} \left\lVert v_{k+1} \right\rVert ^{2},
\end{split}
\end{equation*}
and
\begin{equation}\label{ene:C-S}
\begin{split}
& \ 2\gamma^2 \big(\omega_{0}k + 2 \omega_{6}\big)\left(k+\alpha - c\right) \left\langle v_{k+1} - v_{k}, F \left( z_{k} \right) - F \left( w_{k-1} \right) \right\rangle \\
\leq & \ 2\gamma^3 L \big(\omega_{0}k + 2 \omega_{6}\big)\left(k+\alpha - c\right) \left\lVert v_{k+1} - v_{k} \right\rVert \left\lVert v_{k} - v_{k-1} \right\rVert \\
\leq & \ \gamma^3 L \big(\omega_{0}k + 2 \omega_{6}\big)\left(k+\alpha - c\right) \left( \left\lVert v_{k+1} - v_{k} \right\rVert ^{2} + \left\lVert v_{k} - v_{k-1} \right\rVert ^{2} \right) \\
\leq & \ \minus \gamma^3 L \left[\big(\omega_{0} \left(k + 1 \right) + 2 \omega_{6} \big)\left(k+ 1 + \alpha - c\right)  \left\lVert v_{k+1} - v_{k} \right\rVert ^{2} \right.\\
&\qquad \qquad  \left.- \big(\omega_{0}k + 2 \omega_{6}\big)\left(k + \alpha - c\right)\left\lVert v_{k} - v_{k-1} \right\rVert ^{2}  \right] \\
& \ + 2\gamma^3 L \big(\omega_{0}\left(k + 1 \right) + 2 \omega_{6}\big)\left(k+ 1 + \alpha - c\right) \left\lVert v_{k+1} - v_{k} \right\rVert ^{2}.
\end{split}
\end{equation}
By plugging \eqref{ene:Young:1} and \eqref{ene:C-S} into \eqref{ene:mono}, and then combining the resulting estimate with \eqref{ene:inn}, we obtain for all $k \geq 2$
\begin{equation}\label{ene:sum}
\begin{split}
& \ \minus 4 \left( c - 1 \right) \lambda \gamma \left\langle z_{k+1} - z_{\ast}, v_{k+1} \right\rangle
- 2\gamma \left( k + \alpha\right)\big(\omega_{0}k +2 \omega_{6}\big)  \left\langle z_{k+1} - z_{k}, v_{k+1} - v_{k} \right\rangle \\
\leq & \  2\gamma \left[ \big(\omega_{0}\left( k + 1 \right) + 2 \omega_{6}\big) \left( k + 1 \right) \left\langle z_{k+1} - z_{k}, F \left( z_{k+1} \right) - F \left( w_{k} \right) \right\rangle \right.\\
&\left. \qquad - \big(\omega_{0}k + 2 \omega_{6}\big) k \left\langle z_{k} - z_{k-1}, F \left( z_{k} \right) - F \left( w_{k-1} \right) \right\rangle \right] \\
& \ - \gamma^4 L^2 c \left[ \big(\omega_{0}\left( k + 1 \right) + 2 \omega_{6}\big) \sqrt{\omega_{0}\left( k + 1 \right) + 2 \omega_{6}} \left\lVert v_{k+1} - v_{k} \right\rVert ^{2}\right.\\
&\left. \qquad \qquad \quad
- \big(\omega_{0}k + 2 \omega_{6}\big) \sqrt{\omega_{0}k + 2 \omega_{6}} \left\lVert v_{k} - v_{k-1} \right\rVert ^{2} \right] \\
& \ - \gamma^3L \left[\big(\omega_{0} \left( k + 1 \right) + 2 \omega_{6} \big)(k + \alpha + 1 -c) \left\lVert v_{k+1} - v_{k} \right\rVert ^{2}\right.\\
&\left. \qquad \qquad
- \big(\omega_{0}k + 2 \omega_{6}\big)(k + \alpha -c)  \left\lVert v_{k} - v_{k-1} \right\rVert ^{2} \right] \\
& \ - 4 \left( c - 1\right) \lambda \gamma \left\langle z_{k+1} - z_{\ast}, \xi_{k+1} + F \left( z_{k+1} \right) \right\rangle + \frac{4\left( c - 1 \right)^2}{ \left( k+1 \right)\sqrt{ k+1}} \lambda ^{2} \left\lVert z_{k+1} - z_{\ast} \right\rVert ^{2}\\
& \ +  \Big( \left( k + \alpha\right)\big(\omega_{0}k + 2 \omega_{6}\big) - \mu_{k} \Big) \gamma^{2} \left\lVert v_{k+1} - v_{k} \right\rVert ^{2}  + \sqrt{\omega_{5}(k + 1) + \omega_{7}} \left\lVert z_{k+1} - z_{k} \right\rVert ^{2} \\
& \ + \gamma^{2}c \sqrt{\omega_{0}k + 2 \omega_{6}} \left\lVert v_{k+1} \right\rVert ^{2},
\end{split}
\end{equation}
where
\begin{equation*}
\begin{split}
\mu_{k}  = & \ \left( k + \alpha \right)\big(\omega_{0} k + 2 \omega_{6} \big) - 2\gamma L\big(\omega_{0} \left( k + 1\right) + 2 \omega_{6} \big)( k + 1 + \alpha -c)\\
& \ - (k + 1)\sqrt{k + 1} -(\omega_{5}(k+1) + \omega_{7})\sqrt{\omega_{5}(k+1) + \omega_{7}} \\
& \ -\gamma^2 L^2 c\big(\omega_{0}(k + 1) + 2 \omega_{6}\big)\sqrt{\omega_{0}(k + 1) + 2 \omega_{6}} \\
= & \ \omega_{0}(1 - 2\gamma L)k^2 + \big(2 \omega_{6} + \omega_{0}\alpha \big)k +  2\omega_{6}\alpha \\
& \  - 2 \gamma L\Big(\big( 2\left(\omega_{0} + 2 \omega_{6}\right) - s \omega_{6}\big)k + \big(\omega_{0} + 2 \omega_{6}\big)(\alpha +1 -c)\Big) - (k + 1)\sqrt{k + 1}\\
& \  -(\omega_{5}(k+1) + \omega_{7})\sqrt{\omega_{5}(k+1) + \omega_{7}} -\gamma^2 L^2 c\big(\omega_{0}(k + 1) + 2 \omega_{6}\big)\sqrt{\omega_{0}(k + 1) + 2 \omega_{6}}.
\end{split}
\end{equation*}
Finally, by summing \eqref{ene:pre} and \eqref{ene:sum}, we get the estimate we want.

(ii) We observe that for all $k \geq 1$
\begin{equation}
\begin{split}
& \ 2 \omega_{0} \lambda \gamma k \left\langle z_{k} - z_{\ast}, v_{k} \right\rangle
+ \dfrac{1}{2}  \omega_{0}s  \gamma^{2} k^{2} \left\lVert v_{k} \right\rVert ^{2} \nonumber\\
= & \  \frac{1}{s}\omega_{0}  \left( 2s \lambda \gamma k \left\langle z_{k} - z_{\ast}, v_{k} \right\rangle + \dfrac{ 1}{2 } s^2 \gamma^{2} k^{2} \left\lVert v_{k} \right\rVert ^{2} \right) \nonumber \\
= & \ \frac{1}{s} \omega_{0} \left( \dfrac{1}{2} \left\lVert 2 \lambda \left( z_{k} - z_{\ast} \right) + s\gamma k v_{k} \right\rVert ^{2} - 2 \lambda^{2} \left\lVert z_{k} - z_{\ast} \right\rVert ^{2} \right)
\end{split}
\end{equation}
and
\begin{equation*}
\begin{split}
& \ 4\lambda \gamma \omega_{6} \left\langle z_{k} - z_{\ast}, v_{k} \right\rangle
+ \frac{s\alpha\gamma^2 \omega_{6}^2}{\alpha - 2} \left\lVert v_{k} \right\rVert ^{2}  \\
= & \ 2 \lambda \gamma \omega_{6} \left(2 \left\langle z_{k} - z_{\ast}, v_{k} \right\rangle
+ \frac{s\alpha\gamma \omega_{6}}{2\left(\alpha - 2\right)\lambda} \left\lVert v_{k} \right\rVert ^{2} \right)  \\
= & \ 2\lambda \gamma \omega_{6} \left( \left\lVert \sqrt{\frac{2\left(\alpha - 2\right)\lambda}{s\alpha\gamma \omega_{6}}} \left( z_{k} - z_{\ast} \right) + \sqrt{\frac{s\alpha\gamma \omega_{6}}{2\left(\alpha - 2\right)\lambda}} v_{k}\right\rVert ^{2}   - \frac{2\left(\alpha - 2\right)\lambda}{s\alpha\gamma \omega_{6}} \left\lVert z_{k} - z_{\ast} \right\rVert ^{2} \right).
\end{split}
\end{equation*}
By using the identity
\begin{equation*}
\left\lVert x \right\rVert ^{2} + \left\lVert y \right\rVert ^{2}
= \frac{1}{2} \left( \left\lVert x + y \right\rVert ^{2} + \left\lVert x - y \right\rVert ^{2} \right)
\quad \forall x, y \in \cH,
\end{equation*}
we obtain for all $k \geq 1$
\begin{align*}
\cE_{\lambda,s, k} = & \ \dfrac{1}{2} \left\lVert u_{\lambda,s,k} \right\rVert ^{2} + 2 \lambda \left( \alpha - 1 - \lambda \right) \left\lVert z_{k} - z_{\ast} \right\rVert ^{2} + 2 \lambda \gamma \left(\omega_{0} k+ 2 \omega_{6} \right) \left\langle z_{k} - z_{\ast}, v_{k} \right\rangle\\
& \ + \frac{1}{2}\gamma^2\left(\omega_{0} k + 2 \omega_{6}\right) \left( s k + 2 c \right) \left\lVert v_{k} \right\rVert^{2}\\
= & \ \dfrac{1}{2} \left\lVert 2 \lambda \left( z_{k} - z_{\ast} \right) + 2k \left( z_{k} - z_{k-1} \right) + s \gamma k v_{k} \right\rVert^{2}\\
& \ + 2  \lambda \left( \alpha - 1 - \dfrac{\alpha(\omega_{0}+s+2)-4}{s\alpha}\lambda \right)  \left\lVert z_{k} - z_{\ast} \right\rVert^{2}
+ \frac{1}{2s} \omega_{0}  \left\lVert 2 \lambda \left( z_{k} - z_{\ast} \right) + s\gamma k v_{k} \right\rVert^{2}  \\
& \ + 2\lambda \gamma \omega_{6} \left\lVert \sqrt{\frac{2\left(\alpha - 2\right)\lambda}{s\alpha\gamma\omega_{6}}} \left( z_{k} - z_{\ast} \right) + \sqrt{\frac{s\alpha\gamma \omega_{6}}{2\left(\alpha - 2\right)\lambda}} v_{k}\right\rVert^{2}  \\
& \ + \frac{1}{2}\gamma^2 \left(2\big(\omega_{0} c+ s \omega_{6}\big) k + 4\omega_{6}c - \frac{2s\alpha \omega_{6}^2}{\alpha -2} \right)\left\lVert v_{k} \right\rVert ^{2}  \\
= & \ \dfrac{s - \omega_0}{2s} \left\lVert 2 \lambda \left( z_{k} - z_{\ast} \right) + 2k \left( z_{k} - z_{k-1} \right) + s \gamma k v_{k} \right\rVert ^{2} \\
& \ + 2  \lambda \left( \alpha - 1 - \dfrac{\alpha(\omega_{0}+s+2)-4}{s\alpha}\lambda \right)  \left\lVert z_{k} - z_{\ast} \right\rVert ^{2}
+ \frac{1}{s}\omega_{0} k^2 \left\lVert  z_{k} - z_{k-1} \right\rVert ^{2}  \\
& \ +\frac{1}{4s} \omega_{0} \left\lVert 4 \lambda \left( z_{k} - z_{\ast} \right) + 2k \left( z_{k} - z_{k-1} \right) + 2s \gamma k v_{k} \right\rVert ^{2}\\
& \ + 2 \lambda \gamma \omega_{6}  \left\lVert \sqrt{\frac{2\left(\alpha - 2\right)\lambda}{s\alpha\gamma \omega_{6}}} \left( z_{k} - z_{\ast} \right) + \sqrt{\frac{s\alpha\gamma \omega_{6}}{2\left(\alpha - 2\right)\lambda}} v_{k}\right\rVert ^{2}  \\
& \ + \gamma^2\left(\big(\omega_{0} c+ s \omega_{6} \big)  k +  2 c\omega_{6} -\frac{s\alpha  \omega_{6}^2}{\alpha -2} \right)\left\lVert v_{k} \right\rVert ^{2}.
\end{align*}
Consequently, for all $k \geq 2$
\begin{align*}
\cG_{\lambda,s,k}
= & \ \cE_{\lambda,s,k}- 2\gamma \big(\omega_{0}k + 2 \omega_{6}\big) k \left\langle z_{k} - z_{k-1}, F \left( z_{k} \right) - F \left( w_{k-1} \right) \right\rangle \\
& \ + \gamma^4 L^2 c  \big(\omega_{0}k + 2 \omega_{6}\big) \sqrt{\omega_{0}k + 2 \omega_{6}} \left\lVert v_{k} - v_{k-1} \right\rVert ^{2} + \gamma^3L \big(\omega_{0}k + 2 \omega_{6}\big)(k + \alpha -c)  \left\lVert v_{k} - v_{k-1} \right\rVert ^{2}   \nonumber \\
= & \ \dfrac{s - \omega_0}{ 2s} \left\lVert 2 \lambda \left( z_{k} - z_{\ast} \right) + 2k \left( z_{k} - z_{k-1} \right) + s \gamma k v_{k} \right\rVert ^{2} \\
& \ + 2  \lambda \left( \alpha - 1 - \dfrac{\alpha(\omega_{0}+s+2)-4}{s\alpha}\lambda \right)  \left\lVert z_{k} - z_{\ast} \right\rVert ^{2}
+ \frac{1}{s} \omega_{0} k^2 \left\lVert  z_{k} - z_{k-1} \right\rVert ^{2}  \\
& \ +\frac{1}{4s} \omega_{0} \left\lVert 4 \lambda \left( z_{k} - z_{\ast} \right) + 2k \left( z_{k} - z_{k-1} \right) + 2s \gamma k v_{k} \right\rVert ^{2}\\
& \ + 2\lambda \gamma \omega_{6}  \left\lVert \sqrt{\frac{2\left(\alpha - 2\right)\lambda}{s\alpha\gamma\omega_{6}}} \left( z_{k} - z_{\ast} \right) + \sqrt{\frac{s\alpha\gamma \omega_{6}}{2\left(\alpha - 2\right)\lambda}} v_{k}\right\rVert ^{2}\\
& \ + \gamma^2\left(\big(\omega_{0} c+ s \omega_{6}\big)  k +  2c \omega_{6} -\frac{s\alpha \omega_{6}^2}{\alpha -2} \right)\left\lVert v_{k} \right\rVert ^{2}\\
& \ - 2\gamma \big(\omega_{0}k + 2 \omega_{6}\big) k \left\langle z_{k} - z_{k-1}, F \left( z_{k} \right) - F \left( w_{k-1} \right) \right\rangle \\
& \ + \gamma^4 L^2 c  \big(\omega_{0}k + 2 \omega_{6}\big) \sqrt{ \omega_{0}k + 2 \omega_{6}} \left\lVert v_{k} - v_{k-1} \right\rVert ^{2} + \gamma^3L \big(\omega_{0}k + 2 \omega_{6}\big)(k + \alpha -c)  \left\lVert v_{k} - v_{k-1} \right\rVert ^{2}.
\end{align*}
We use relation \eqref{split:Lip} and $0< \gamma < \frac{1}{2L}$ to verify that for all $k \geq 2$
\begin{align*}
& \ \dfrac{1}{2} \omega_{0} k^2 \left\lVert z_{k} - z_{k-1} \right\rVert ^{2} - 2 \omega_{0} \gamma k^2 \left\langle z_{k} - z_{k-1}, F \left( z_{k} \right) - F \left( w_{k-1} \right) \right\rangle+ \omega_{0} \gamma^{3} L k^2 \left\lVert v_{k} - v_{k-1} \right\rVert ^{2} \nonumber\\
\geq & \ \omega_{0} k^2\left(\dfrac{1}{2} \left\lVert z_{k} - z_{k-1} \right\rVert ^{2} - 2 \gamma \left\langle z_{k} - z_{k-1}, F \left( z_{k} \right) - F \left( w_{k-1} \right) \right\rangle + 2\gamma^2 \left\lVert F \left( z_{k} \right) - F \left( w_{k-1} \right) \right\rVert ^{2}\right) \nonumber\\
\geq & \ 0
\end{align*}
and
\begin{align*}
& \ \dfrac{1}{4s}\omega_{0}^2k^2 \left\lVert z_{k} - z_{k-1} \right\rVert ^{2} - 4 \omega_{6}\gamma k \left\langle z_{k} - z_{k-1}, F \left( z_{k} \right) - F \left( w_{k-1} \right) \right\rangle + \dfrac{16\omega_{6}^2  s}{\omega_{0}^2} \gamma^4 L^2 \left\lVert v_{k} - v_{k-1} \right\rVert ^{2} \\
\geq & \ \dfrac{1}{4s}\omega_{0}^2k^2 \left\lVert z_{k} - z_{k-1} \right\rVert ^{2} - 4 \omega_{6} \gamma k \left\langle z_{k} - z_{k-1}, F \left( z_{k} \right) - F \left( w_{k-1} \right) \right\rangle + \dfrac{16 \omega_{6}^2  s}{\omega_{0}^2} \gamma^2  \left\lVert F(z_{k}) - F(w_{k-1}) \right\rVert ^{2}\\
\geq & \ 0.
\end{align*}
Since $\omega_{0}=2-s$,
\begin{equation*}
\frac{s-\omega_0}{2s}=\frac{s-1}{s}>0, \quad \frac{\alpha(\omega_{0}+s+2)-4}{s\alpha}=\frac{4(\alpha-1)}{s\alpha},
\end{equation*}
and
\begin{equation*}
\frac{1}{s} \omega_0 = \frac{1}{2} \omega_0 + \frac{1}{2s}\omega_0^2.
\end{equation*}
Therefore there exists a positive integer $k_1 \geq 2$ such that for all $k \geq k_1$ it holds
\begin{align*}
\cG_{\lambda,s,k}
\geq & \ \frac{1}{4s} \omega_{0} \left\lVert 4 \lambda \left( z_{k} - z_{\ast} \right) + 2k \left( z_{k} - z_{k-1} \right) + 2s \gamma k v_{k} \right\rVert ^{2} \nonumber \\
& \ + \frac{1}{4s} \omega_{0}^2k^{2} \left\lVert z_{k} - z_{k-1} \right\rVert ^{2}+ 2  \lambda \left( \alpha - 1 -\dfrac{4(\alpha-1)}{s\alpha}\lambda \right)\left\lVert z_{k} - z_{\ast} \right\rVert ^{2},
\end{align*}
which is the desired inequality.
\qedhere
\end{proof}

\begin{proof}[Proof of Lemma \ref{lem:trunc}]
(i) For $k \geq 1$ and the quadratic expression in $R_{k}$ we calculate
\begin{align*}
\dfrac{\Delta_{k}'}{4 \gamma^{2}} = & \ \left( \omega_{1} k + \omega_{2} \right) ^{2} - 2\delta^2\left( \omega_{3} k + \sqrt{\omega_{5}(k + 1) + \omega_{7}} \right) \Bigl( \omega_{4} k    + c\sqrt{\omega_{0}k + 2 \omega_{6}} + s \omega_{6}\Bigr) \nonumber \\
= & \ \left( \omega_{1}^{2} - 2\delta^2 \omega_{3} \omega_{4} \right) k^{2} - 2\delta^2 \Bigl( \omega_{3}\bigl( c\sqrt{\omega_{0}k + 2 \omega_{6}} + s \omega_{6}\bigr)+ \omega_4 \sqrt{\omega_{5}(k+1) + \omega_{7}} \Bigr)k \\
& \ + 2 \omega_1\omega_2 k + \omega_2^2-2\delta^2\Bigl( c\sqrt{\omega_{0}k + 2 \omega_{6}} + s \omega_{6}\Bigr)\sqrt{\omega_{5}(k+1) + \omega_{7}}.
\end{align*}
Since $\left( \omega_{1}^{2} - 2\delta^2 \omega_{3} \omega_{4} \right) k^{2}$ is the dominant term in the above polynomial, it suffices to guarantee that
\begin{equation}
\label{nonneg:gamma}
\omega_{1}^{2} - 2\delta^2 \omega_{3} \omega_{4} < 0,
\end{equation}
to ensure the existence of some integer $k_{\lambda} \geq 1$ such that $\Delta_{k}' \leq 0$ and $\omega_{4} k    + c\sqrt{\omega_{0}k + 2 \omega_{6}} + s \omega_{6} < 0$ for all $k \geq  k_{\lambda}$. By Lemma~\ref{lem:quad}~(ii),  from here it would follow that $R_{k} \leq 0$ for all $k \geq  k_{\lambda}$.

It remains to show that there exists a choice of $0 \leq \lambda \leq \alpha-1$ for which \eqref{nonneg:gamma} is true. We set $\xi := \lambda + 1 - \alpha \leq 0$ and get
\begin{align*}
\omega_{1}
&= (2-s)\lambda +s  \left( \lambda + 1 - \alpha \right) + s - 2c = 2\xi + \left(\alpha-1 \right)\left( 2 -s \right) +s - 2c, \\
\omega_{3} \omega_{4} 	& = 4s \left( 1 - c \right)  \left( \lambda + 1 - \alpha \right) = - 4s\left( c - 1 \right)  \xi.
\end{align*}
This means that we have in fact to guarantee that there exists a choice for $\xi \leq 0$ such that
\begin{align}\label{trunc:omega-a}
& \ \omega_{1}^{2} - 2\delta^2 \omega_{3} \omega_{4} \nonumber \\
 = & \ \big( 2  \xi + \left(\alpha - 1 \right) \left( 2 - s \right) + s - 2c  \big)^2  + 8 s \left( c - 1 \right)  \xi\delta^2  \nonumber \\
= & \ 4\xi^2 + 4\big( \left(\alpha - 1 \right) \left( 2 - s \right)+ s - 2c +2s\left( c - 1 \right)\delta^2 \big)\xi + \big( \left(\alpha - 1 \right)\left( 2 - s \right) + s - 2c\big)^2 < 0.
\end{align}
A direct computation shows that according to \eqref{condi:delta}
\begin{align*}
\Delta_{\xi}
 = & \ 16\big( \left(2 - s \right)\left( \alpha - 1\right)+ s - 2c +2s\left( c - 1 \right)\delta^2 \big)^2 - 16\big( \left(\alpha - 1 \right)\left( 2 - s \right) + s - 2c\big)^2\\
= & \ 64s \left( c - 1 \right)  \delta^2  \big( \left(2 - s \right)\left( \alpha - 1\right)+ s - 2c +s\left( c - 1 \right)\delta^2 \big)   > 0.
\end{align*}
Hence, in order to get \eqref{trunc:omega-a}, we have to choose $\xi$ between the two roots of the quadratic function arising in this formula, in other words
\begin{align*}
\xi_{1} \left( \alpha,c,s \right) & := \dfrac{1}{8 } \left( - 4  \big( \left(2 - s \right)\left( \alpha - 1\right)+ s - 2c +2s\left( c - 1 \right)\delta^2 \big)  - \sqrt{\Delta_{\xi}} \right) \nonumber \\
& < \xi = \lambda + 1 - \alpha\nonumber \\
&< \xi_{2} \left( \alpha,c,s \right) := \dfrac{1}{8 } \left( - 4  \big( \left(2 - s \right)\left( \alpha - 1\right)+ s - 2c +2s\left( c - 1 \right)\delta^2 \big)  + \sqrt{\Delta_{\xi}} \right) .
\end{align*}
Obviously $\xi_{1} \left( \alpha,c,s \right) < 0$ and from Vieta's formula $\xi_{1} \left( \alpha,c,s\right) \cdot \xi_{2} \left( \alpha,c,s \right) = \frac{\big( \left(\alpha - 1 \right)\left( 2 - s \right) + s - 2c\big)^2}{4 }>0$, it follows that $\xi_{2} \left( \alpha,c,s \right) < 0$.

Therefore, going back to $\lambda$, in order to be sure that $\omega_{1}^{2} - 2\delta^2\omega_{3} \omega_{4} < 0$ this must be chosen such that
\begin{equation*}
\alpha - 1 + \xi_{1} \left( \alpha,c,s \right) < \lambda < \alpha - 1 + \xi_{2} \left( \alpha,c,s \right) .
\end{equation*}
Next, we will show that
\begin{equation}
\label{trunc:check}
0 < \alpha - 1 -  \frac{1}{2}  \big( \left(2 - s \right)\left( \alpha - 1\right)+ s - 2c +2s\left( c - 1 \right)\delta^2 \big)  < \frac{\alpha s}{4}.
\end{equation}
Indeed, the inequality on the left-hand side follows immediately, since
\begin{equation*}
0 < \alpha - 1 -  \frac{1}{2}  \big( \left(2 - s \right)\left( \alpha - 1\right)+ s - 2c +2s\left( c - 1 \right)\delta^2 \big)\,\, \Longleftrightarrow \,\,\delta^2< \frac{ s \left( \alpha - 2 \right)+ 2c}{2s \left( c - 1 \right)},
\end{equation*}
which is true according to \eqref{condi:s} and \eqref{condi:delta}. On the other hand, for the inequality on the right-hand side of \eqref{trunc:check} we have
\begin{equation*}
\alpha - 1 -  \frac{1}{2}  \big( \left(2 - s \right)\left( \alpha - 1\right)+ s - 2c +2s\left( c - 1 \right)\delta^2 \big)  < \frac{\alpha s}{4}\,\,\Longleftrightarrow\,\,\delta^2> \frac{s\left(\alpha - 2\right) + 2\left(2c - s \right)}{4s\left(c - 1\right)},
\end{equation*}
which is true according to \eqref{condi:delta}.
From \eqref{trunc:check} we immediately deduce that
\begin{equation*}
0 < \alpha - 1 + \xi_{2} \left( \alpha,c,s \right)
\quad \textrm{ and } \quad
\alpha - 1 + \xi_{1} \left( \alpha,c,s \right) < \frac{\alpha s}{4},
\end{equation*}
which allows us to choose
\begin{equation*}
\underline{\lambda}\left( \alpha, c,s\right):=\max\left\{0,\alpha - 1 + \xi_{1} \left( \alpha,c,s \right)\right\}<\overline{\lambda}\left( \alpha,c,s \right):=\min\left\{\frac{\alpha s}{4},\alpha - 1 + \xi_{2} \left( \alpha,c,s \right)\right\}.
\end{equation*}
In conclusion, for $\lambda$ chosen to satisfy $\underline{\lambda}\left( \alpha, c,s\right)<\lambda <\overline{\lambda}\left( \alpha,c,s \right)$, we have
\begin{equation*}
\omega_{1}^2-2\delta^2 \omega_2\omega_4< 0,
\end{equation*}
and therefore there exists some integer ${ k_{\lambda}} \geq 1$ such that $R_k \leq 0$ for all $k\geq k_{\lambda}$.

(ii) For every $k \geq 2$ we have
\begin{equation*}
\begin{split}
\mu_{k} = & \ \omega_{0} \left(1 - 2\gamma L\right)k^2 + \left(2 \omega_{6} + \omega_{0}\alpha \right)k + 2 \omega_{6} \alpha  \\
& \ - 2 \gamma L\left(\left( 2(\omega_{0} + 2 \omega_{6})- s \omega_{6}\right)k + \left(\omega_{0} + 2 \omega_{6}\right)(\alpha +1 -c)\right)  - (k + 1)\sqrt{k + 1}\\
& \  -(\omega_{5}(k+1) + \omega_{7})\sqrt{\omega_{5}(k+1) + \omega_{7}} -\gamma^2 L^2 c\big(\omega_{0}(k + 1) + 2 \omega_{6}\big)\sqrt{\omega_{0}(k + 1) + 2 \omega_{6}},
\end{split}
\end{equation*}
and the conclusion follows since $\omega_{0} >0 $ and $\gamma < \frac{1}{2L}$.
\qedhere
\end{proof}

\end{document}